\documentclass[a4paper, 11pt]{article}

\usepackage[utf8]{inputenc}
\usepackage[T1]{fontenc}
\usepackage{lmodern}     
\usepackage[english]{babel}
\usepackage{cite, hyperref}

\usepackage{amsfonts, amsmath}
\usepackage{amssymb, amsthm}
\usepackage{bbm, mathrsfs}
\usepackage{calc}
\usepackage[shortlabels]{enumitem}

\usepackage[a4paper, top=2.5cm, bottom=2.5cm, left=3cm, right=3cm]{geometry}
\allowdisplaybreaks

\numberwithin{equation}{section}
\theoremstyle{plain}

\newtheorem{Lemma}{Lemma}[section]
\newtheorem{Proposition}[Lemma]{Proposition}
\newtheorem{Theorem}[Lemma]{Theorem}
\newtheorem{Corollary}[Lemma]{Corollary}

\theoremstyle{definition}
\newtheorem{Definition}[Lemma]{Definition}
\newtheorem{Example}[Lemma]{Example}
\newtheorem{Remark}[Lemma]{Remark}

\def\B{{\rm B}}
\def\E{\mathbb{E}}
\def\R{{\rm R}}
\def\Re{\mathbb{R}}
\def\N{\mathbb{N}}
\def\P{\mathbb{P}}
\DeclareMathOperator*{\esssup}{ess\,sup}
\begin{document}
\title{\textbf{Stochastic Volterra equations with random functional coefficients in Banach spaces}}
\author{Alexander Kalinin\footnote{Department of Mathematics, LMU Munich, Germany. E-mail: {\tt kalinin@math.lmu.de}
}}
\maketitle

\begin{abstract}
We derive unique Banach-valued solutions to stochastic Volterra equations with random coefficients that may depend on pure chance and involve singular kernels. In particular, for controlled and distribution-dependent coefficients these solutions become strong, as a measurability analysis of the Wasserstein metric confirms. The presented novel approach is based on the proof that a stochastic Volterra integral admits a progressively measurable modification in a weak sense and on sharp moment estimates for non-negative product measurable processes.
\end{abstract}


\section{Introduction}

In contrast to diffusion processes that solve stochastic differential equations (SDEs) driven by Brownian motions, Volterra processes arising as solutions to stochastic Volterra integral equations may fail to be semimartingales and do not need to possess continuous paths. This behaviour comes from the additional dependence of the coefficients on the time variable. Thus, many proven methods for SDEs that rely on It{\^o}'s formula and the inequalities of Gronwall and Burkholder-Davis-Gundy no longer work once kernels with singularities are involved. 

But the probabilistic properties of solutions to stochastic Volterra equations extend the range of applications of SDEs. For instance, initiated by~\cite{GatJaiRos18}, one-dimensional Volterra processes are used in~\cite{EleRos18, EleRos19, BonPulSco24} for rough volatility modelling in mathematical finance. Moreover, various stochastic control problems of the last years, as those in \cite{Wan22, NunGio24, AbiNeu25}, incorporate stochastic Volterra equations, and~\cite{AccBacCar19} shows that McKean-Vlasov SDEs can also be handled in stochastic optimization. Motivated by such applications, the aim of this work is to develop a solution theory in Banach spaces that allows for controlled and distribution-dependent coefficients and singular kernels.

In what follows, let $I$ be a non-degenerate interval in $\Re_{+}$ with $0\in I$ and $E$ be a separable Banach space with complete norm $|\cdot|$ such that $E\neq\{0\}$. Further, let $(\Omega,\mathcal{F},\P)$ be a  probability space and $\mathbb{F} = (\mathcal{F}_{t})_{t\in I}$ be a filtration of $\mathcal{F}$ such that the usual conditions hold. That is, $\mathbb{F}$ is right-continuous and any subset of a null event lies in $\mathcal{F}_{0}$.

We consider the separable Hilbert space $\ell^{2}$ of all sequences $a =(a_{i})_{i\in\N}$ of real numbers that are square-summable, endowed with the inner product given by $\langle a,b\rangle_{\ell^{2}} := \sum_{i=1}^{\infty} a_{i}b_{i}$, and we recall from the Riesz-Fischer Theorem that any separable Hilbert space of infinite dimension is isometrically isomorphic to $\ell^{2}$.

Based on~\cite{Ond04, NeeVerWei07}, there is a norm that turns the linear space $\mathcal{L}_{2}(\ell^{2},E)$ of all $E$-valued linear continuous maps on $\ell^{2}$ that are radonifying into a separable Banach space. Under the hypothesis that $E$ is $2$-smooth, the stochastic integral
\begin{equation*}
\int_{0}^{\cdot}U_{s}\,\mathrm{d}W_{s}
\end{equation*}
of an $\mathcal{L}_{2}(\ell^{2},E)$-valued $\mathbb{F}$-progressively measurable process $U$ with locally square-integrable paths with respect to a sequence $W = (W^{(i)})_{i\in\N}$ of independent standard $\mathbb{F}$-Brownian motions can be constructed as continuous $\mathbb{F}$-local martingale.

We let $\xi:I\times\Omega\rightarrow E$ be an $\mathbb{F}$-progressively measurable process and $\mathcal{D}$ be a non-empty set of $E$-valued random vectors such that if $Y\in\mathcal{D}$ and $Z$ is an $E$-valued random vector with $Y = Z$ a.s., then $Z\in\mathcal{D}$. Moreover, let
\begin{equation*}
\B:I\times I\times\Omega\times\mathcal{D}\rightarrow E\quad\text{and}\quad\Sigma:I\times I\times\Omega\times\mathcal{D}\rightarrow\mathcal{L}_{2}(\ell^{2},E)
\end{equation*}
be two maps that are admissible in the sense of Definition~\ref{de:admissible maps}. Within this probabilistic framework, we consider the following \emph{stochastic Volterra integral equation with random functional coefficients} coupled with a value condition:
\begin{equation}\label{eq:stochastic Volterra equation}
X_{t} = \xi_{t} + \int_{0}^{t}\B_{t,s}(X_{s})\,\mathrm{d}s + \int_{0}^{t}\Sigma_{t,s}(X_{s})\,\mathrm{d}W_{s}\quad\text{a.s.}
\end{equation}
for $t\in I$. As Banach-valued functionals, $\B_{t,s}(\cdot)(\omega)$ and $\Sigma_{t,s}(\cdot)(\omega)$ depend on the random vector $X_{s}$ of any solution $X$, and their domain $\mathcal{D}$ could be the linear space $\mathcal{L}^{p}(\Omega,E)$ of all $E$-valued $p$-fold integrable random vectors, where $s,t\in I$ with $s\leq t$, $\omega\in\Omega$ and $p\geq 1$.

Hence,~\eqref{eq:stochastic Volterra equation} represents a \emph{novel type of stochastic Volterra equation} for which we will derive unique solutions in Theorems~\ref{th:existence of unique solutions 1} and~\ref{th:existence of unique solutions 2}. In particular, for \emph{controlled and distribution-dependent coefficients}~\eqref{eq:controlled and distribution-dependent coefficients} we obtain strong solutions, as Example~\ref{ex:controlled and distribution-dependent stochastic Volterra equations} shows. In Example~\ref{ex:random coefficients} we will also consider the \emph{probabilistic representation}~\eqref{eq:random coefficients}, which allows for random coefficients of affine type.

This analysis of~\eqref{eq:stochastic Volterra equation} is based on novel contributions to the existing literature on kernels, stochastic processes, stochastic integrals and metrics that are of independent interest and which can be described as follows:
\begin{enumerate}[label=(\arabic*), leftmargin=\widthof{(6)} + \labelsep, topsep = 3 pt, itemsep = 0 pt]
\item To handle singular kernels, as those in Example~\ref{ex:types of kernels}, \emph{two kinds of integral estimates for the iterated kernels} are given in Propositions~\ref{pr:integral estimate for iterated kernels of first type} and~\ref{pr:integral estimate for iterated kernels of second type}. In particular, the first kind extends several estimates in the proof of Theorem~1 in~\cite{Gri80}.

\item The notion of a \emph{weak modification of a process} is introduced in Definition~\ref{de:weak modification} and \emph{sufficient conditions for two weakly modified processes to be indistinguishable} are given in Proposition~\ref{pr:weakly modified processes}.

\item We deduce \emph{progressively measurable weak modifications of stochastic Volterra integrals} in Proposition~\ref{pr:progressively measurable stochastic Volterra integrals} by handling \emph{the convergence in probability of sequences of stochastic Volterra integrals} in Proposition~\ref{pr:approximation of stochastic Volterra integrals}.

\item To allow for the controlled and distribution-dependent coefficients~\eqref{eq:controlled and distribution-dependent coefficients}, the \emph{Borel measurability in the $p$th Wasserstein space of the law map} of a product measurable $p$-fold integrable process is established in Proposition~\ref{pr:Borel measurability of the law map} for $p\geq 1$.

\item This measurability follows from the \emph{pointwise approximation of processes} with values in separable metrisable spaces, as described in Corollary~\ref{co:approximation of product measurable processes}, and the \emph{representations of the Wasserstein distance} between a probability measure and a convex combination of Dirac measures in Proposition~\ref{pr:representations of the Wasserstein metric}.

\item The \emph{sharp moment inequalities} for processes in Proposition~\ref{pr:resolvent sequence inequality for processes} and Corollary~\ref{co:resolvent inequality for processes} are derived from Minkowski's integral inequality and the resolvent inequalities in~\cite{Kal24}. As a result, we obtain the \emph{integral estimates} in Corollaries~\ref{co:integral resolvent sequence inequality for processes} and~\ref{co:integral resolvent inequality for processes}.

\item If $I$ is not compact, then the completeness of the solution spaces, introduced at~\eqref{eq:seminorm for processes} and~\eqref{eq:seminorm for processes 2}, is inferred from the \emph{metrical decomposition principle} in Corollary~\ref{co:metrical decomposition}, based on the \emph{functional construction of a metric} in Proposition~\ref{pr:functional metric}.
\end{enumerate}

In Euclidean spaces, stochastic Volterra equations were initially investigated by Berger and Mizel~\cite{BerMiz80, Bermiz80-2}, in particular, see~\cite[Theorems~3.A,~3.B and~3.C]{BerMiz80}. The extension to path-dependent coefficients under Lipschitz conditions was studied in~\cite{Pro85,Kal21}. Namely, Protter~\cite[Theorem~4.3]{Pro85} has deduced unique solutions to one-dimensional stochastic Volterra equations driven by semimartingales, and the existence of unique strong solutions to multidimensional stochastic Volterra equations with Brownian motions as drivers has been shown in~\cite[Lemma~1.1]{Kal21}. In this context, the support of the law of a solution in the Hölder norm has been characterised in~\cite[Theorem~1.2]{Kal21}. 

Based on the Skorohod integral, which extends the It{\^o} integral, stochastic Volterra equations with anticipating coefficients satisfying Lipschitz conditions have been solved by Pardoux and Protter~\cite[Theorem~5.1]{ParPro90} and Al{\`os} and Nualart~\cite[Theorems~3.4 and 3.5]{AloNua97}. Cochran, Lee and Potthoff~\cite[Theorem~1.2]{CocLeePot95} derive unique solutions in a distributional sense, in the case of vanishing drift coefficients, linear diffusion coefficients and singular kernels. For unique solutions when the coefficients satisfy Osgood conditions and the singular kernels are of certain type, see the article~\cite[Theorems~1.1 and~3.1]{Wan08} by Wang. Moreover, Zhang~\cite[Theorems~3.1 and~3.7]{Zha10} provides unique Banach-valued solutions for stochastic Volterra equations with singular kernels.

Without any Lipschitz conditions, affine Volterra processes were derived in a weak sense by Abi Jaber, Larsson and Pulido~\cite[Theorems~3.4 and~3.6]{AbiLarPul19}. In a more general setting, the existence and stability of weak solutions to stochastic Volterra equations with singular convolution kernels and semimartingales as drivers have been proven by Abi Jaber, Cuchiero, Larsson and Pulido~\cite[Theorems~1.2 and~1.6]{AbiCucLar21}. For further stability results and unique weak solutions when the singular convolution kernels are out of scope of~\cite{AbiCucLar21}, see~\cite[Theorems~2.8 and~2.13]{Abi21}. In one dimension, weak solutions to stochastic Volterra equations driven by Brownian motions, with singular kernels that do not need to be of convolution type, have been derived by Prömel and Scheffels~\cite[Theorem~3.3]{ProSch23}. 

For unique strong solutions to stochastic Volterra equations with L{\'e}vy processes as drivers, see the paper~\cite[Theorem~1]{DaiXia20} by Dai and Xiao. Given distribution-dependent drift and Hölder continuous diffusion coefficients and regular kernels, unique strong solutions have been deduced by Jie, Luo and Zhang~\cite[Theorem~1.3]{JieLuoZha24} in one dimension. This extends the Yamada-Watanabe approach in~\cite{ProSch23-2} to the case of law-dependent drifts. Moreover, see~\cite[Theorem~3.1]{KalMeyPro24} for unique strong solutions to SDEs with law-dependent drift and locally Hölder continuous diffusion coefficients. Finally, stochastic Volterra equations with singular kernels, where both coefficients are distribution-dependent and satisfy Lipschitz conditions, have recently been studied by Prömel and Scheffels~\cite[Theorem~2.3]{ProSch25}.

This paper was intended to complement these works by allowing for random coefficients that depend in a probabilistic way on the solution, as stated in~\eqref{eq:stochastic Volterra equation}. To this end, a novel approach is required, which includes the following exemplary case: $\mathcal{D} = \mathcal{L}^{1}(\Omega,E)$,
\begin{equation}\label{eq:exemplary coefficients}
\begin{split}
\B_{t,s}(X_{s}) &= f(t - s)\big(\kappa_{s} + f_{1}(X_{s},\alpha_{s}) + \E\big[f_{2}(X_{s},\alpha_{s})\big]\big),\\
\quad \Sigma_{t,s}(X_{s}) &= g(t - s)\big(\eta_{s} + g_{1}(X_{s},\alpha_{s}) + \E\big[g_{2}(X_{s},\alpha_{s})\big]\big)
\end{split}
\end{equation}
for all $s,t\in I$ with $s < t$ and $X_{s}\in\mathcal{D}$, where $f,g:I\rightarrow\Re$ are measurable, $A$ is a separable Banach space, $\kappa$, $\eta$ and $\alpha$ are $\mathbb{F}$-progressively measurable bounded processes with values in $E$, $\mathcal{L}_{2}(\ell^{2},E)$ and $A$, respectively, and
\begin{equation*}
f_{i}:E\times A\rightarrow E,\quad g_{i}:E\times A\rightarrow\mathcal{L}_{2}(\ell^{2},E)
\end{equation*}
are Lipschitz continuous for $i\in\{1,2\}$. Further, $f$ and $g^{2}$ are required to be locally integrable. Under these conditions, which can be weakened, the following four assertions hold:
\begin{enumerate}[label=(\roman*), leftmargin=\widthof{(iv)} + \labelsep, itemsep = 0 pt]
\item If the $p$th moment function $I\rightarrow [0,\infty]$, $t\mapsto \E[|\xi_{t}|^{p}]$ is locally essentially bounded for $p\geq 2$, then Theorem~\ref{th:existence of unique solutions 1} yields a unique solution $X^{\xi}$ to~\eqref{eq:stochastic Volterra equation} in the sense of Definition~\ref{de:notion of a solution} such that $\E[|X^{\xi}|^{p}]$ is also locally essentially bounded.

\item If $\E[|\xi|^{p}]$ is just locally $p$-fold integrable for $p\geq 2$, then Theorem~\ref{th:existence of unique solutions 2} applies and we obtain a unique solution to~\eqref{eq:stochastic Volterra equation} whose $p$th moment function is locally integrable.

\item Let $\xi = \xi_{0}$ a.s.~and $\E[|\xi_{0}|^{p}] < \infty$ for all $p\geq 2$ and suppose that there is $\hat{\beta}\in ]0,\frac{1}{2}]$ such that for any $T\in I$ there is $c \geq 0$ satisfying
\begin{equation}\label{eq:specific condition}
\begin{split}
&\int_{0}^{t-s}|f(u)|\,\mathrm{d}u + \bigg(\int_{0}^{t-s}g(u)^{2}\,\mathrm{d}u\bigg)^{\frac{1}{2}}\\
& + \int_{0}^{s}|f(t - s + u) - f(u)|\,\mathrm{d}u +  \bigg(\int_{0}^{s}\big(g(t-s + u) - g(u)\big)^{2}\,\mathrm{d}u\bigg)^{\frac{1}{2}} \leq c(t-s)^{\hat{\beta}}
\end{split}
\end{equation}
for all $s,t\in [0,T]$ with $s\leq t$. Then Theorem~\ref{th:existence of unique solutions 1} gives a unique regular solution $\hat{X}^{\xi_{0}}$ to~\eqref{eq:stochastic Volterra equation} for which $\E[|\hat{X}^{\xi_{0}}|^{p}]$ is locally bounded for all $p\geq 2$ and whose paths are locally $\beta$-Hölder continuous for any $\beta\in ]0,\hat{\beta}[$.

\item In particular, if $f(u) = u^{\beta -1}$ and $g(u) = u^{\gamma - 1/2}$ for all $u\in I\setminus\{0\}$ with $\beta\in]0,1]$ and $\gamma\in ]0,\frac{1}{2}]$, then the condition~\eqref{eq:specific condition} is valid for $\hat{\beta} = \beta\wedge\gamma$, by the inequality~\eqref{eq:basic inequalities}.
\end{enumerate}

We note that, as in the cases of controlled and distribution-dependent coefficients~\eqref{eq:controlled and distribution-dependent coefficients} and random coefficients of type~\eqref{eq:random coefficients}, the stochastic Volterra equation~\eqref{eq:stochastic Volterra equation} with the exemplary coefficients~\eqref{eq:exemplary coefficients} has not been solved so far.\smallskip

This work is structured as follows. In Section~\ref{se:2} the probabilistic framework, which contains analytic aspects, is introduced and the preliminary and main results are stated. First, Section~\ref{se:2.1} recalls the iterated kernels and resolvents of non-negative kernels and Section~\ref{se:2.2} is concerned with the properties of stochastic Volterra integrals with values in Banach spaces. In Section~\ref{se:2.3} we study admissible coefficients and give a definition of a solution to~\eqref{eq:stochastic Volterra equation} that extends the classical solution concept for an SDE. In Sections~\ref{se:2.4} and~\ref{se:2.5} we derive unique solutions to~\eqref{eq:stochastic Volterra equation} with locally essentially bounded and locally integrable moment functions, respectively.

Section~\ref{se:3} provides general results on the theory of stochastic processes. Namely, in Section~\ref{se:3.1} weakly modified processes are introduced and analysed, and in Section~\ref{se:3.2} we approximate processes pointwise and investigate the measurability of their distribution maps. Sharp moment and integral inequalities for processes are deduced in Section~\ref{se:3.3} and applied to stochastic Volterra processes in Section~\ref{se:3.4}.

Section~\ref{se:4} contains supplementary analytic results on kernels and metrics. While in Section~\ref{se:4.1} two types of integral estimates for iterated kernels are deduced, in Section~\ref{se:4.2} a metrical decomposition principle is established. Eventually, the preliminary and main results of Section~\ref{se:2} are proven in Section~\ref{se:5}.

\section{Preliminaries and main results}\label{se:2}

We write $\mathcal{L}(\ell^{2},E)$ for the Banach space of all $E$-valued linear continuous maps on $\ell^{2}$, endowed with the operator norm that is also denoted by $|\cdot|$. By convention, for any monotone function $f:\Re_{+}\rightarrow\Re$ we set $f(\infty) := \lim_{x\uparrow\infty} f(x)$, and for $d,m\in\N$ the transpose of a matrix $A\in\Re^{m\times d}$ is denoted by
$A^{\top}$.

\subsection{Iterated kernels and resolvents}\label{se:2.1}

By a \emph{non-negative kernel} on $I$, which in this section is just a non-degenerate interval in $\Re$, we shall mean an $[0,\infty]$-valued measurable function $k$ on the triangular set of all $(t,s)\in I\times I$ with $s\leq t$.

As introduced in~\cite{Kal24}, the iterated kernels of $k$ relative to a $\sigma$-finite Borel measure $\mu$ on $I$ are represented by the sequence $(\R_{k,\mu,n})_{n\in\N}$ of non-negative kernels on $I$ recursively defined via
\begin{equation}\label{eq:resolvent sequence}
\R_{k,\mu,1}(t,s) := k(t,s)\quad\text{and}\quad\R_{k,\mu,n+1}(t,s):=\int_{[s,t]}k(t,\tilde{s})\R_{k,\mu,n}(\tilde{s},s)\,\mu(\mathrm{d}\tilde{s}),
\end{equation}
based on Fubini's theorem. Then $\R_{k,\mu}:=\sum_{n=1}^{\infty}\R_{k,\mu,n}$ is the \emph{resolvent} of $k$ with respect to $\mu$, which yields for each $s\in I$ a solution to the linear Volterra integral equation $\R_{k,\mu}(t,s) = k(t,s) + \int_{[s,t]}k(t,\tilde{s})\R_{k,\mu}(\tilde{s},s)\,\mu(\mathrm{d}\tilde{s})$ for $t\in I$ with $t\geq s$.

In the case that $\mu$ is the Lebesgue measure on $I$, we shall simply write $\R_{k,n}$ and $\R_{k}$ for $\R_{k,\mu,n}$ and $\R_{k,\mu}$, respectively, where $n\in\N$. The iterated kernels and the resolvent are indispensable for the sharp moment and integral estimates in Section~\ref{se:3.3}.

Given $q\geq 1$, we will use the following tractable types of kernels for the regularity conditions on the coefficients of the stochastic Volterra equation~\eqref{eq:stochastic Volterra equation}.

\begin{Definition}\label{de:convex cones of non-negative kernels}
Let $\mathcal{K}_{\infty}^{q}$ denote the set of all non-negative kernels $k$ on $I$ satisfying $\esssup_{t\in [0,T]}\int_{0}^{t}k(t,s)^{q}\,\mathrm{d}s < \infty$ for all $T\in I$ and $\mathcal{K}^{q}$ be the set of all $k\in\mathcal{K}_{\infty}^{q}$ such that
\begin{equation*}
\lim_{\delta\downarrow 0}\esssup_{\substack{r,t\in [0,T]:\\ r \leq t \leq r + \delta}} \int_{r}^{t}k(t,s)^{q}\,\mathrm{d}s = 0\quad\text{for each $T\in I$.}
\end{equation*}
Furthermore, let $\hat{\mathcal{K}}^{q}$ stand for the set of all non-negative kernels $k$ on $I$ that satisfy $\esssup_{s\in [0,T]}\int_{s}^{T}k(t,s)^{q}\,\mathrm{d}t < \infty$ and
\begin{equation*}
\lim_{\delta\downarrow 0}\esssup_{\substack{r,t\in [0,T]:\\
r \leq t \leq r + \delta}} \int_{r}^{t}k(s,r)^{q}\,\mathrm{d}s = 0\quad\text{for any $T\in I$.}
\end{equation*}
\end{Definition}

\begin{Remark}
We have $\mathcal{K}_{\infty}^{q + \varepsilon} \subset\mathcal{K}^{q}$ for any $\varepsilon > 0$. Similarly, if $k$ is a non-negative kernel on $I$ such that $\esssup_{s\in [0,T]}\int_{s}^{T}k(t,s)^{q + \varepsilon}\,\mathrm{d}t < \infty$ for all $T\in I$, then $k\in\hat{\mathcal{K}}^{q}$.
\end{Remark}

These three convex cones possess two relevant properties. First, for a non-negative kernel $k$ on $I$ we have $k\in\mathcal{K}_{\infty}^{q}$ whenever $k\leq l$ for some $l\in\mathcal{K}_{\infty}^{q}$. Secondly, if $c:I\rightarrow [0,\infty]$ is measurable and locally essentially bounded, then the kernels
\begin{equation*}
\{(t,s)\in I\times I\,|\,s\leq t\}\rightarrow [0,\infty],\quad (t,s)\mapsto c(t)k(t,s)
\end{equation*}
and $\{(t,s)\in I\times I\,|\,s\leq t\}\rightarrow [0,\infty]$, $(t,s)\mapsto k(t,s)c(s)$ lie in $\mathcal{K}_{\infty}^{q}$ as soon as $k$ does. The same assertions hold for $\mathcal{K}^{q}$ and $\hat{\mathcal{K}}^{q}$. Let us next consider three representations of $k$.

\begin{Example}\label{ex:types of kernels}
(i) \emph{(Kernels with separate variables)} Let $k_{0},k_{1}:I\rightarrow [0,\infty]$ be measurable and $k(t,s) = k_{0}(t)k_{1}(s)$ for all $s,t\in I$ with $s <t$. If $k_{0}$ is locally essentially bounded, then
\begin{equation*}
k\in\mathcal{K}_{\infty}^{q}
\end{equation*}
once $k_{1}$ is locally $q$-fold integrable, and $k\in\mathcal{K}^{q}$ if $k_{1}$ is locally $\tilde{q}$-fold integrable for some $\tilde{q} > q$. By changing the roles of $k_{0}$ and $k_{1}$, we obtain sufficient conditions for $k\in\hat{\mathcal{K}}^{q}$.\smallskip

\noindent
(ii) \emph{(Convolution kernels)} Suppose that $f:I\rightarrow [0,\infty]$ is a measurable function such that $k(t,s) = f(t-s)$ for any $s,t\in I$ with $s < t$. Then
\begin{equation*}
k\in\mathcal{K}^{q}\cap\hat{\mathcal{K}}^{q}\quad\Leftrightarrow\quad \text{$f$ is locally $q$-fold integrable.}
\end{equation*}
If instead $\int_{0}^{T}f(u)^{q}\,\mathrm{d}u = \infty$ for some $T\in I$, then $k\in\mathcal{K}_{\infty}^{q}$ is impossible and we readily see that $\esssup_{s\in [0,T]}\int_{s}^{T}k(t,s)^{q}\,\mathrm{d}t = \infty$.\smallskip

\noindent
(iii) {\it (Transformed fractional kernels)} Let $\alpha > 0$, $\beta\in [0,\alpha[$ and $\varphi:I\rightarrow\Re$ be strictly increasing and locally absolutely continuous such that
\begin{equation*}
k(t,s) = \dot{\varphi}(s)\big(\varphi(t) - \varphi(s)\big)^{\alpha - 1}\big(\varphi(s) - \varphi(0)\big)^{-\beta}
\end{equation*}
for all $s,t\in I$ with $0 < s < t$, where $\dot{\varphi}$ is a positive weak derivative of $\varphi$. If $\beta \geq 1$, then $k\notin\mathcal{K}_{\infty}^{1}$, and $k\in\mathcal{K}^{1}$, otherwise. For the choice $\varphi(t) = t^{\gamma}$ for all $t\in I$ with $\gamma > 0$ we have
\begin{equation}\label{eq:condition on the exponents}
k\in\mathcal{K}^{q}\cap\hat{\mathcal{K}}^{q}\quad\Leftrightarrow\quad \alpha > 1 - \frac{1}{q}\quad\text{and}\quad \min\{\alpha,1\}\gamma > 1 - \frac{1}{q} + \beta\gamma.
\end{equation}
If, however, one of the two inequalities in~\eqref{eq:condition on the exponents} fails, then $k\notin\mathcal{K}_{\infty}^{q}$ and it follows that $\esssup_{s\in [0,T]}\int_{s}^{T}k(t,s)^{q}\,\mathrm{d}t$ for some $T\in I$.
\end{Example}

\subsection{Stochastic Volterra integrals in Banach spaces}\label{se:2.2}

By means of the progressive $\sigma$-field $\mathcal{A}$ on $I\times\Omega$, which consists of all sets $A$ in $I\times\Omega$ for which $\mathbbm{1}_{A}$ is $\mathbb{F}$-progressively measurable, the \emph{progressive measurability of Bochner integral processes of Volterra type} follows from Corollary~\ref{co:approximation of product measurable processes}.

\begin{Proposition}\label{pr:progressively measurable integral processes}
Let $U:I\times I\times\Omega\rightarrow E$ be $\mathcal{B}(I)\otimes\mathcal{A}$-measurable and $\mu$ be a Borel measure on $I$ such that $\mu([0,t])$ and $\int_{[0,t]}|U_{t,s}|\,\mu(\mathrm{d}s)$ are finite for all $t\in I$. Then
\begin{equation*}
X:I\times\Omega\rightarrow E,\quad X_{t}(\omega) := \int_{[0,t]}U_{t,s}(\omega)\,\mu(\mathrm{d}s)
\end{equation*}
is an $\mathbb{F}$-progressively measurable process.
\end{Proposition}

Let us recall the stochastic integral in Banach spaces. For an orthonormal basis $(b_{i})_{i\in\N}$ of $\ell^{2}$ and a sequence $(Y_{i})_{i\in\N}$ of independent and standard normally distributed random variables, a map $L\in\mathcal{L}(\ell^{2},E)$ is called radonifying if the series $\sum_{i=1}^{\infty}L(b_{i})Y_{i}$ converges in second moment. 

As shown in~\cite{Ond04}, this property is independent of the choice of the basis $(b_{i})_{i\in\N}$ and the sequence $(Y_{i})_{i\in\N}$ of random variables, and if it holds, then the distribution of $\sum_{i=1}^{\infty}Y_{i}L(b_{i})$ is the Gaussian measure $\mathcal{N}(0,L\circ L^{*}$) on $E$.

Here, the adjoint of $L$ is the unique $\ell^{2}$-valued linear continuous operator $L^{*}$ on the topological dual space $E^{*}$ of $E$ that satisfies $(x^{*}\circ L)(a) = \langle a, L^{*}(x^{*})\rangle_{\ell^{2}}$ for all $a\in\ell^{2}$ and $x^{*}\in E^{*}$. Note that the linear space $\mathcal{L}_{2}(\ell^{2},E)$ of all radonifying $L\in\mathcal{L}(\ell^{2},E)$, equipped with the norm
\begin{equation*}
|L|_{2}:= \E\bigg[\bigg|\sum_{i=1}^{\infty}Y_{i}L(b_{i})\bigg|^{2}\bigg]^{\frac{1}{2}},
\end{equation*}
turns into a separable Banach space and $|L| \leq |L|_{2}$ for all $L\in\mathcal{L}_{2}(\ell^{2},E)$. See Theorem~2.2, Definition~2.3, the consecutive discussion and Proposition~2.5 in~\cite{Ond04} for details.

\begin{Example}[Hilbert-Schmidt norm]
For $d,m\in\N$ let $E = \Re^{m}$ and $A\in\Re^{m\times d}$. If $L:\ell^{2}\rightarrow\Re^{m}$ is given by $L(a) := A(a_{1},\dots,a_{d})^{\top}$, then $L\in\mathcal{L}_{2}(\ell^{2},\Re^{m})$ and
\begin{equation*}
|L|_{2}^{2} = \E\big[|AY|^{2}\big] = \mathrm{tr}(A^{\top}A) = |A|^{2}
\end{equation*}
for the random vector $Y := (Y_{1},\dots,Y_{d})^{\top}$ and the Hilbert-Schmidt norm $|\cdot|$ on $\Re^{m\times d}$.
\end{Example}

Next, let $E$ be $2$-smooth in the sense of~\cite{Pis75}. By recalling this concept in Section~\ref{se:5.1}, it follows from Lemma~\ref{le:smoothness of a normed space} that this property holds if and only if there are a norm $\|\cdot\|$ that is equivalent to $|\cdot|$ and a real constant $\hat{c}\geq 2$ such that
\begin{equation}\label{eq:2-smoothness condition}
\|x + y\|^{2} + \|x - y\|^{2} \leq 2\|x\|^{2} + \hat{c}\|y\|^{2}\quad\text{for all $x,y\in E$.}
\end{equation}
In this case, $E$ is uniformly smooth and reflexive. If equality holds in~\eqref{eq:2-smoothness condition} for $\hat{c} = 2$, then we recover the parallelogram identity and $\|\cdot\|$ is induced by an inner product. In general, an equivalent norm to $|\cdot|$ that turns $E$ into a Hilbert space does not need to exist.

\begin{Example}
For a $\sigma$-finite measure space $(S,\mathcal{B},\mu)$ and $p\geq 1$, we recall the linear space $\mathcal{L}^{p}(S,\mathcal{B},\mu)$ of all real-valued $\mathcal{B}$-measurable $p$-fold $\mu$-integrable functions on $S$, endowed with the complete seminorm
\begin{equation*}
\mathcal{L}^{p}(S,\mathcal{B},\mu)\rightarrow\Re_{+},\quad f\mapsto\bigg(\int_{S}|f|^{p}\,\mathrm{d}\mu\bigg)^{\frac{1}{p}},
\end{equation*}
and let $\mathcal{N}^{p}$ be the linear space of all $f\in \mathcal{L}^{p}(S,\mathcal{B},\mu)$ with $f = 0$ $\mu$-a.e. Then the Banach space $L^{p}(S,\mathcal{B},\mu)$, which equals the quotient space $\mathcal{L}^{p}(S,\mathcal{B},\mu)/\mathcal{N}^{p}$, is $2$-smooth if $p\geq 2$. For justifications of this fact see~\cite[Proposition~2.1]{Pin94} and~\cite[Lemma~2.4]{DueGeeVerWel10}.
\end{Example}

Following~\cite{Ond04}, for any $\mathcal{L}_{2}(\ell^{2},E)$-valued $\mathbb{F}$-progressively measurable process $U$ satisfying $\int_{0}^{\cdot}|U_{s}|_{2}^{2}\,\mathrm{d}s < \infty$ a.s., the stochastic integral $I\times\Omega\rightarrow E$, $(t,\omega)\mapsto \int_{0}^{t}U_{s}\,\mathrm{d}W_{s}(\omega)$ is a continuous $\mathbb{F}$-local martingale such that the following three conditions hold:
\begin{enumerate}[(i)]
\item $\int_{0}^{\cdot}\alpha U_{s} + V_{s}\,\mathrm{d}W_{s} = \alpha\int_{0}^{\cdot}U_{s}\,\mathrm{d}W_{s} + \int_{0}^{\cdot}V_{s}\,\mathrm{d}W_{s}$ a.s.~for all $\alpha\in\Re$ and any $\mathcal{L}_{2}(\ell^{2},E)$-valued $\mathbb{F}$-progressively measurable process $V$ such that $\int_{0}^{\cdot}|V_{s}|_{2}^{2}\,\mathrm{d}s < \infty$ a.s.

\item $\int_{0}^{\cdot\wedge\tau} U_{s}\,\mathrm{d}W_{s} = \int_{0}^{\cdot}\mathbbm{1}_{[0,\tau]}(s)U_{s}\,\mathrm{d}W_{s}$ a.s.~for every $\mathbb{F}$-stopping time $\tau$.

\item For each $\mathcal{L}_{2}(\ell^{2},E)$-valued process $V$ as in~(i) and any $K,L\in E^{*}$ the covariation of $K(\int_{0}^{\cdot}U_{s}\,\mathrm{d}W_{s})$ and $L(\int_{0}^{\cdot}V_{s}\,\mathrm{d}W_{s})$ is indistinguishable from $\int_{0}^{\cdot}\langle U_{s}^{*}(K),V_{s}^{*}(L)\rangle_{\ell^{2}}\,\mathrm{d}s$.
\end{enumerate}

Moreover, the following version of the Burkholder-Davis-Gundy inequality holds: For any $p\geq 2$ there is a real constant $w_{p} > 0$ such that
\begin{equation}\label{eq:stochastic integral estimate}
\E\bigg[\sup_{t\in [0,T]}\bigg|\int_{0}^{t}\,U_{s}\,\mathrm{d}W_{s}\bigg|^{p}\bigg]^{\frac{1}{p}} \leq w_{p}\E\bigg[\bigg(\int_{0}^{T}|U_{s}|_{2}^{2}\,\mathrm{d}s\bigg)^{\frac{p}{2}}\bigg]^{\frac{1}{p}}
\end{equation}
for any $\mathcal{L}_{2}(\ell^{2},E)$-valued $\mathbb{F}$-progressively measurable process $U$ with $\int_{0}^{\cdot}|U_{s}|_{2}^{2}\,\mathrm{d}s < \infty$ a.s. and each $T\in I$. From these properties we infer the \emph{convergence in probability of a sequence of stochastic Volterra integrals to another stochastic Volterra integral}.

\begin{Proposition}\label{pr:approximation of stochastic Volterra integrals}
Let $(U^{(n)})_{n\in\N}$ be a sequence of $\mathcal{L}_{2}(\ell^{2},E)$-valued $\mathcal{B}(I)\otimes\mathcal{A}$-measurable maps on $I\times I\times\Omega$ such that
\begin{equation}\label{eq:condition for the approximation of stochastic Volterra integrals}
\int_{0}^{t}|U_{t,s}^{(k)}|_{2}^{2} + |U_{t,s}|_{2}^{2}\,\mathrm{d}s < \infty\quad\text{a.s.}\quad\text{and}\quad\lim_{n\uparrow\infty} \P\bigg(\int_{0}^{t}|U_{t,s}^{(n)} - U_{t,s}|_{2}^{2}\,\mathrm{d}s \geq \varepsilon\bigg) = 0
\end{equation}
for all $k\in\N$, $t\in I$ and $\varepsilon > 0$ and some $\mathcal{B}(I)\otimes\mathcal{A}$-measurable map $U:I\times I\times\Omega\rightarrow\mathcal{L}_{2}(\ell^{2},E)$. Then
\begin{equation}\label{eq:approximation of stochastic Volterra integrals}
\lim_{n\uparrow\infty}\P\bigg(\sup_{t\in [0,T]}\bigg|\int_{0}^{t}U_{T,s}^{(n)}\,\mathrm{d}W_{s} - \int_{0}^{t}U_{T,s}\,\mathrm{d}W_{s}\bigg|\geq\varepsilon\bigg) = 0 
\end{equation}
for any $T\in I$ and $\varepsilon > 0$.
\end{Proposition}

\begin{Remark}
If in fact $|U_{t,s}^{(k)}|_{2} \leq V_{t,s}$ and $\lim_{n\uparrow\infty} U_{t,s}^{(n)} = U_{t,s}$ for all $k\in\N$ and $s,t\in I$ with $s\leq t$ and some $\mathcal{B}(I)\otimes\mathcal{A}$-measurable function $V:I\times I\times\Omega\rightarrow [0,\infty]$ satisfying
\begin{equation*}
\int_{0}^{t}V_{t,s}^{2}\,\mathrm{d}s < \infty\quad\text{a.s.}\quad\text{for any $t\in I$,}
\end{equation*}
then $\lim_{n\uparrow\infty} \int_{0}^{t}|U_{t,s}^{(n)} - U_{t,s}|_{2}^{2}\,\mathrm{d}s = 0$ and $\int_{0}^{t}|U_{t,s}|_{2}^{2}\,\mathrm{d}s$ $\leq \int_{0}^{t}V_{t,s}^{2}\,\mathrm{d}s$ a.s.~for all $t\in I$, by dominated convergence. In such a case, condition~\eqref{eq:condition for the approximation of stochastic Volterra integrals} is redundant.
\end{Remark}

Based on Proposition~\ref{pr:approximation of stochastic Volterra integrals}, we can prove that a \emph{stochastic Volterra integral admits a progressively measurable weak modification} in the sense of Definition~\ref{de:weak modification}.

\begin{Proposition}\label{pr:progressively measurable stochastic Volterra integrals}
Let $U:I\times I\times\Omega\rightarrow\mathcal{L}_{2}(\ell^{2},E)$ be a $\mathcal{B}(I)\otimes\mathcal{A}$-measurable map such that $\int_{0}^{t}|U_{t,s}|_{2}^{2}\,\mathrm{d}s< \infty$ a.s.~for any $t\in I$. Then there is an $E$-valued $\mathbb{F}$-progressively measurable process $X$ such that
\begin{equation}\label{eq:weak modification of a stochastic Volterra integral}
X_{t} = \int_{0}^{t}U_{t,s}\,\mathrm{d}W_{s}\quad\text{a.s.}\quad\text{for a.e.~$t\in I$.}
\end{equation}
\end{Proposition}

\subsection{Admissible coefficients and notions of solutions}\label{se:2.3}

For a separable Banach space $\tilde{E}$ with complete norm denoted by $|\cdot|$, we consider maps with a certain measurability property and which do not depend on the particular choice of a weak modification described as follows.

\begin{Definition}\label{de:admissible maps}
A map $F:I\times I\times\Omega\times\mathcal{D}\rightarrow \tilde{E}$, $(t,s,\omega,Y)\mapsto F_{t,s}(Y)(\omega)$ will be called \emph{admissible} if for each $E$-valued $\mathbb{F}$-progressively measurable process $X$ satisfying $X_{s}\in\mathcal{D}$ for all $s\in I$ the map
\begin{equation}\label{eq:measurable map}
I\times I\times\Omega\rightarrow\tilde{E},\quad (t,s,\omega)\mapsto F_{t,s}(X_{s})(\omega)
\end{equation}
is $\mathcal{B}(I)\otimes\mathcal{A}$-measurable and for any $t\in I$ and every $\mathbb{F}$-progressively measurable weak modification $\tilde{X}$ of $X$ in the sense of Definition~\ref{de:weak modification}, we have
\begin{equation}\label{eq:weak modification}
F_{t,s}(X_{s}) = F_{t,s}(\tilde{X}_{s})\quad\text{a.s.}\quad\text{for a.e.~$s\in I$.}
\end{equation}
\end{Definition}

\begin{Remark}\label{re:admissible maps}
The $\mathcal{B}(I)\otimes\mathcal{A}$-measurability of the map~\eqref{eq:measurable map} entails that the process $I\times\Omega\rightarrow\tilde{E}$, $(s,\omega)\mapsto F_{t,s}(X_{s})(\omega)$ is $\mathbb{F}$-progressively measurable and~\eqref{eq:weak modification} ensures that
\begin{equation*}
\int_{I}|F_{t,s}(X_{s}) - F_{t,s}(\tilde{X}_{s})|^{q}\,\mathrm{d}s = 0\quad\text{a.s.}
\end{equation*}
for all $t\in I$ and $q > 0$, as an application of Fubini's theorem shows.
\end{Remark}

Maps that admit the subsequent integral representation, which involves the law of any $Y\in\mathcal{D}$, are admissible.

\begin{Proposition}\label{pr:admissible maps 1}
Let $D\in\mathcal{B}(E)$ be such that each $Y\in\mathcal{D}$ takes all its values in $D$. If $G:I\times I\times\Omega\times D\times D\rightarrow\tilde{E}$ is product measurable when $I\times\Omega$ is endowed with $\mathcal{A}$ and 
\begin{equation*}
\text{$G_{t,s}(Y(\omega),\cdot)(\omega)$ is $\mathcal{L}(Y)$-integrable}
\end{equation*}
for all $s,t\in I$, $\omega\in\Omega$ and $Y\in\mathcal{D}$, then the map $F:I\times I\times\Omega\times\mathcal{D}\rightarrow\tilde{E}$ defined by $F_{t,s}(Y)(\omega) := \int_{D}G_{t,s}(Y(\omega),y)(\omega)\,\mathcal{L}(Y)(\mathrm{d}y)$ is admissible.
\end{Proposition}

\begin{Example}\label{ex:admissible maps 1}
Let $A$ be a separable metrisable space, $\alpha$ be an $A$-valued $\mathbb{F}$-progressively measurable process and
\begin{equation*}
g:I\times I\times D\times A\times D\rightarrow\tilde{E}
\end{equation*}
be a Borel measurable map such that $g(t,s,x,a,\cdot)$ is bounded for any $s,t\in I$, $x\in D$ and $a\in A$. Then from the measure transformation formula we infer that the map
\begin{equation*}
I\times I\times\Omega\times\mathcal{D}\rightarrow\tilde{E},\quad (t,s,\omega,Y)\mapsto \E\big[g(t,s,x,a,Y)\big]_{|(x,a) = (Y,\alpha_{s})(\omega)}
\end{equation*}
is admissible by choosing $G_{t,s}(x,y) = g(t,s,x,\alpha_{s},y)$ for all $s,t\in I$ and $x,y\in D$ in Proposition~\ref{pr:admissible maps 1}.
\end{Example}

Definition~\ref{de:admissible maps} also allows for maps that depend on the joint distribution of a control process and any $Y\in\mathcal{D}$. Here and subsequently, for any separable metrisable space $S$ let $\mathcal{P}_{0}(S)$ denote the convex space of all Borel probability measures on $S$.

\begin{Proposition}\label{pr:admissible maps 2}
Let $D\in\mathcal{B}(E)$ be such that any $Y\in\mathcal{D}$ is $D$-valued. Further, $A$ let be a separable metrisable space, $\alpha$ be an $A$-valued $\mathbb{F}$-progressively measurable process and $\mathcal{P}$ be a separable metrisable space in $\mathcal{P}_{0}(E\times A)$ such that
\begin{equation*}
\mathcal{L}(Y,\alpha_{s})\in\mathcal{P}\quad\text{for all $Y\in\mathcal{D}$ and $s\in I$.}
\end{equation*}
If $G:I\times I\times\Omega\times D\times\mathcal{P}\rightarrow\tilde{E}$ is product measurable when $I\times\Omega$ is equipped with $\mathcal{A}$, then the map $F:I\times I\times\Omega\times\mathcal{D}\rightarrow\tilde{E}$ defined by
\begin{equation*}
F_{t,s}(Y)(\omega) := G_{t,s}\big(Y(\omega),\mathcal{L}(Y,\alpha_{s})\big)(\omega)
\end{equation*}
is admissible under the condition that for each $D$-valued product measurable process $X$ satisfying $\mathcal{L}(X_{s},\alpha_{s})\in\mathcal{P}$ for any $s\in I$ the map
\begin{equation}\label{eq:Borel measurable distribution map}
I\rightarrow\mathcal{P},\quad s\mapsto\mathcal{L}(X_{s},\alpha_{s})
\end{equation}
is Borel measurable.
\end{Proposition}

\begin{Example}\label{ex:admissible maps 2}
(i) If the Borel measurability of the map~\eqref{eq:Borel measurable distribution map} is always ensured, then for every Borel measurable map $g:I\times I\times D\times A\times\mathcal{P}\rightarrow\tilde{E}$ we see that the map
\begin{equation*}
I\times I\times\Omega\times\mathcal{D}\rightarrow\tilde{E},\quad (t,s,\omega,Y)\mapsto  g\big(t,s,Y(\omega),\alpha_{s}(\omega),\mathcal{L}(Y,\alpha_{s})\big) 
\end{equation*}
is admissible by taking $G_{t,s}(x,\mu) = g(t,s,x,\alpha_{s},\mu)$ for any $s,t\in I$, $x\in D$ and $\mu\in\mathcal{P}$ in Proposition~\ref{pr:admissible maps 2}.\smallskip

\noindent
(ii) By Proposition~\ref{pr:Borel measurability of the law map}, the map~\eqref{eq:Borel measurable distribution map} is Borel measurable if $A$ is a non-empty, convex and closed subspace of a separable Banach space and $\mathcal{P}$ is the $p$th Wasserstein space $\mathcal{P}_{p}(E\times A)$ for $p\geq 1$, which is generally recalled at~\eqref{eq:Wasserstein metric}.
\end{Example}

So, the coefficients $\B$ and $\Sigma$ of~\eqref{eq:stochastic Volterra equation} are ought to be admissible. Then for each $E$-valued $\mathbb{F}$-progressively measurable process $X$ satisfying $X_{s}\in\mathcal{D}$ for a.e.~$s\in I$, the set
\begin{equation}\label{eq:null event}
\bigg\{\int_{0}^{t}|\B_{t,s}(X_{s})| + |\Sigma_{t,s}(X_{s})|_{2}^{2}\,\mathrm{d}s = \infty\bigg\}
\end{equation}
lies in $\mathcal{F}_{t}$ for any $t\in I$, by Proposition~\ref{pr:progressively measurable integral processes} and Corollary~\ref{co:approximation of product measurable processes}. Further, Fubini's theorem ensures that the set of all $t\in I$ for which the event~\eqref{eq:null event} is null is Borel.

\begin{Definition}\label{de:notion of a solution}
A \emph{solution} to~\eqref{eq:stochastic Volterra equation} is an $E$-valued $\mathbb{F}$-progressively measurable process $X$ such that $X_{s}\in\mathcal{D}$ for a.e.~$s\in I$ and the set of all $t\in I$ for which
\begin{equation}\label{eq:notion of a solution}
\begin{split}
\int_{0}^{t}&|\B_{t,s}(X_{s})| + |\Sigma_{t,s}(X_{s})|_{2}^{2}\,\mathrm{d}s < \infty\quad\text{and}\\
X_{t} &= \xi_{t} + \int_{0}^{t}\B_{t,s}(X_{s})\,\mathrm{d}s + \int_{0}^{t}\Sigma_{t,s}(X_{s})\,\mathrm{d}W_{s}\quad\text{a.s.}
\end{split}
\end{equation}
fails has Lebesgue measure zero. Moreover, if in addition $X_{t}\in\mathcal{D}$ and~\eqref{eq:notion of a solution} holds for every $t\in I$, then $X$ is called \emph{regular}.
\end{Definition}

\begin{Remark}\label{re:notion of a solution}
(i) Since $\B$ and $\Sigma$ are admissible, every $\mathbb{F}$-progressively measurable weak modification of a solution to~\eqref{eq:stochastic Volterra equation} is another solution, due to Remark~\ref{re:admissible maps}.\smallskip

\noindent (ii) Let $\B$ and $\Sigma$ be independent of the first variable $t\in I$ and $\xi$ be right-continuous. Then any right-continuous solution $X$ to~\eqref{eq:stochastic Volterra equation}  satisfies~\eqref{eq:notion of a solution} for any $t\in I$ with $t < \sup I$, by Proposition~\ref{pr:weakly modified processes}. If in fact
\begin{equation*}
\xi = \xi_{0},
\end{equation*}
then~\eqref{eq:stochastic Volterra equation} reduces to an \emph{SDE with random coefficients} coupled with an initial condition. In this case, any continuous regular solution in the sense of Definition~\ref{de:notion of a solution} can be viewed as solution to the SDE in the original meaning.\smallskip
\end{Remark}

We conclude this section with a relevant example of the coefficients $\B$ and $\Sigma$. To this end, let $A$ be a separable metrisable space, $\alpha$ be an $A$-valued $\mathbb{F}$-progressively measurable process and $\mathcal{P}$ be a separable metrisable space in $\mathcal{P}_{0}(E\times A)$.

We denote the first and second marginal distribution of any $\mu\in\mathcal{P}_{0}(E\times A)$ by $\mu_{1}$ and $\mu_{2}$, respectively, and suppose that there are two separable metrisable spaces $\mathcal{P}_{1}$ and $\mathcal{P}_{2}$ in $\mathcal{P}_{0}(E)$ and $\mathcal{P}_{0}(A)$, respectively, such that $\mu\in\mathcal{P}$ $\Leftrightarrow$ $\mu_{1}\in\mathcal{P}_{1}$ and $\mu_{2}\in\mathcal{P}_{2}$.

\begin{Example}
By utilising the Prokhorov metric, we see that for $\mathcal{P} = \mathcal{P}_{0}(E\times A)$ we may take $\mathcal{P}_{1} = \mathcal{P}_{0}(E)$ and $\mathcal{P}_{2} = \mathcal{P}_{0}(A)$. If instead the $p$th Wasserstein metric is used for $p\geq 1$, then for $\mathcal{P} = \mathcal{P}_{p}(E\times A)$ the choices $\mathcal{P}_{1} = \mathcal{P}_{p}(E)$ and $\mathcal{P}_{2} = \mathcal{P}_{p}(A)$ are feasible.
\end{Example}

Finally, let $b$ and $\sigma$ be Borel measurable maps on $I\times I\times E\times A\times\mathcal{P}$ with respective values in $E$ and $\mathcal{L}_{2}(\ell^{2},E)$ and $\mathcal{D}$ be the set of all $E$-valued random vectors $Y$ satisfying $\mathcal{L}(Y)\in\mathcal{P}_{1}$ such that
\begin{equation}\label{eq:controlled and distribution-dependent coefficients}
\B_{t,s}(X_{s}) = b\big(t,s,X_{s},\alpha_{s},\mathcal{L}(X_{s},\alpha_{s})\big),\quad \Sigma_{t,s}(X_{s}) = \sigma\big(t,s,X_{s},\alpha_{s},\mathcal{L}(X_{s},\alpha_{s})\big)
\end{equation}
for all $s,t\in I$ and $X_{s}\in\mathcal{D}$. Here, the map~\eqref{eq:Borel measurable distribution map} is required to be Borel measurable for each $E$-valued product measurable process $X$ satisfying $\mathcal{L}(X_{s})\in\mathcal{P}_{1}$ for all $s\in I$.

Then $\B$ and $\Sigma$ are admissible, by Proposition~\ref{pr:admissible maps 2} and Example~\ref{ex:admissible maps 2}, and~\eqref{eq:stochastic Volterra equation} reduces to a \emph{stochastic Volterra integral equation with distribution-dependent coefficients} that are \emph{controlled by $\alpha$}.

As the space $\mathcal{P}$ appears in the domains of $b$ and $\sigma$, we recall from Definition~\ref{de:notion of a solution} that a solution to~\eqref{eq:stochastic Volterra equation} is an $E$-valued $\mathbb{F}$-progressively measurable process $X$ such that $\mathcal{L}(X_{s})\in\mathcal{P}_{1}$ for a.e.~$s\in I$ and~\eqref{eq:notion of a solution} holds for a.e.~$t\in I$.

Further, we can consider \emph{strong solutions}. Namely, let $(\mathcal{E}_{t}^{\xi,\alpha})_{t\in I}$ be the natural filtration of $\xi$, $\alpha$ and the sequence $(W^{(i)})_{i\in\N}$ of standard $\mathbb{F}$-Brownian motions. That means,
\begin{equation}\label{eq:natural filtration}
\mathcal{E}_{t}^{\xi,\alpha} = \sigma(\xi_{s}:s\in [0,t])\vee\sigma(\alpha_{s}:s\in [0,t])\vee\sigma\big(W_{s}^{(i)}:s\in [0,t],\,i\in\N\big)
\end{equation}
for all $t\in I$. Then a solution $X$ to~\eqref{eq:stochastic Volterra equation} is called \emph{strong} if it is progressively measurable with respect to the right-continuous filtration of the augmented filtration of $(\mathcal{E}_{t}^{\xi,\alpha})_{t\in I}$.

\subsection{Solutions with locally essentially bounded moment functions}\label{se:2.4}

We seek to derive a unique solution $X$ to~\eqref{eq:stochastic Volterra equation} as limit of a Picard iteration such that the measurable $p$th moment function $I\rightarrow [0,\infty]$, $t\mapsto \E[|X_{t}|^{p}]$ is locally essentially bounded for $p\geq 2$.

For this purpose, we consider the linear space $\mathcal{L}_{loc}^{\infty,p}(I\times\Omega,E)$ of all $\mathbb{F}$-progressively measurable processes $X:I\times\Omega\rightarrow E$ such that $\esssup_{t\in [0,T]}\E[|X_{t}|^{p}] < \infty$ for any $T\in I$, endowed with the topology of convergence with respect to the seminorm
\begin{equation}\label{eq:seminorm for processes}
\mathcal{L}_{loc}^{\infty,p}(I\times\Omega,E)\rightarrow\Re_{+},\quad X\mapsto \esssup_{t\in [0,T]}\E\big[|X_{t}|^{p}\big]^{\frac{1}{p}}
\end{equation}
for each $T\in I$. So, a sequence $(X^{(n)})_{n\in\N}$ converges to a process $X$ in this space if and only if it converges in $p$th moment to $X$, locally essentially uniformly in $t\in I$. That means,
\begin{equation}\label{eq:convergence in pth moment, locally essentially uniformly in time}
\lim_{n\uparrow\infty} \esssup_{t\in [0,T]}\E\big[|X_{t}^{(n)} - X_{t}|^{p}\big] = 0\quad\text{for all $T\in I$.}
\end{equation}
As verified in Lemma~\ref{le:completely metrisable space}, this space is completely pseudometrisable. Thus, if $\mathcal{N}_{p}$ is the linear space of all $X\in\mathcal{L}_{loc}^{\infty,p}(I\times\Omega,E)$ such that $X_{t} = 0$ a.s.~for a.e.~$t\in I$, then the quotient space $\mathcal{L}_{loc}^{\infty,p}(I\times\Omega,E)/\mathcal{N}_{p}$ is completely metrisable.

Based on the tractable types of kernels in Definition~\ref{de:convex cones of non-negative kernels}, let us introduce an \emph{affine growth condition} on $\B$ and $\Sigma$:
\begin{enumerate}[label=(C.\arabic*), ref=C.\arabic*, leftmargin=\widthof{(C.1)} + \labelsep]
\item\label{co:1} There are $k_{1}\in\mathcal{K}_{\infty}^{1}$, $k_{2}\in\mathcal{K}_{\infty}^{2}$, $l_{1}\in\mathcal{K}^{1}$ and $l_{2}\in\mathcal{K}^{2}$ such that
\begin{equation}\label{eq:affine growth condition}
\E\big[|\B_{t,s}(Y)|^{p}\big]^{\frac{1}{p}}\mathbbm{1}_{\{1\}}(i) + \E\big[|\Sigma_{t,s}(Y)|_{2}^{p}\big]^{\frac{1}{p}}\mathbbm{1}_{\{2\}}(i) \leq k_{i}(t,s) + l_{i}(t,s)\E\big[|Y|^{p}\big]^{\frac{1}{p}}
\end{equation}
for all $i\in\{1,2\}$, $s,t\in I$ with $s < t$ and $Y\in\mathcal{D}$.
\end{enumerate}

Under~\eqref{co:1} and by using the constant $w_{p}$ in the moment estimate~\eqref{eq:stochastic integral estimate}, we define a measurable locally essentially bounded function $k_{0}:I\rightarrow [0,\infty]$ by
\begin{equation}\label{eq:special integral function}
k_{0}(t) := \int_{0}^{t}k_{1}(t,s)\,\mathrm{d}s + w_{p}\bigg(\int_{0}^{t}k_{2}(t,s)^{2}\,\mathrm{d}s\bigg)^{\frac{1}{2}}.
\end{equation}
As $\mathcal{K}^{2}$ is a convex cone that satisfies the properties stated after Definition~\ref{de:convex cones of non-negative kernels}, we may introduce a non-negative kernel $l\in\mathcal{K}^{2}$ by
\begin{equation}\label{eq:special transformed kernel}
l(t,s) := 2\max\bigg\{\min\bigg\{l_{1}(t,s)\int_{0}^{t}l_{1}(t,\tilde{s})\,\mathrm{d}\tilde{s},\int_{0}^{t}l_{1}(t,\tilde{s})^{2}\,\mathrm{d}\tilde{s}\bigg\}^{\frac{1}{2}},w_{p}l_{2}(t,s)\bigg\}.
\end{equation}
Then Proposition~\ref{pr:integral estimate for iterated kernels of first type}, Remark~\ref{re:integral estimate for iterated kernels of first type} and Fubini's theorem show that $l$ induces the measurable locally essentially bounded function series $\mathrm{I}_{l}:I\rightarrow [0,\infty]$ given by
\begin{equation}\label{eq:function series}
\mathrm{I}_{l}(t) := \sum_{n=1}^{\infty}\bigg(\int_{0}^{t}\R_{l^{2},n}(t,s)\,\mathrm{d}s\bigg)^{\frac{1}{2}}.
\end{equation}
By means of these coefficients we obtain \emph{$L^{p}$-growth estimates} for solutions and Picard iterations to~\eqref{eq:stochastic Volterra equation}.

\begin{Proposition}\label{pr:growth estimates for solutions and Picard iterations 1}
Let~\eqref{co:1} hold and $\xi,X^{(0)}\in\mathcal{L}_{loc}^{\infty,p}(I\times\Omega,E)$. If $X$ is a solution to~\eqref{eq:stochastic Volterra equation} in $\mathcal{L}_{loc}^{\infty,p}(I\times\Omega,E)$, then
\begin{equation}\label{eq:growth estimate for solutions 1}
\begin{split}
\E\big[|X_{t} - \xi_{t}|^{p}\big]^{\frac{1}{p}} &\leq k_{0}(t) + \sum_{n=1}^{\infty}\bigg(\int_{0}^{t}\R_{l^{2},n}(t,s)k_{0}(s)^{2}\,\mathrm{d}s\bigg)^{\frac{1}{2}}\\
&\quad + \sum_{n=1}^{\infty}\bigg(\int_{0}^{t}\R_{l^{2},n}(t,s)\E\big[|\xi_{s}|^{p}\big]^{\frac{2}{p}}\,\mathrm{d}s\bigg)^{\frac{1}{2}}
\end{split}
\end{equation}
for a.e.~$t\in I$. Further, if $\mathcal{L}^{p}(\Omega,E)\subset\mathcal{D}$, then the sequence $(X^{(n)})_{n\in\N}$ in $\mathcal{L}_{loc}^{\infty,p}(I\times\Omega,E)$, recursively given by requiring that the set $I_{n+1}$ of all $t\in I$ satisfying
\begin{equation}\label{eq:Picard iteration}
\begin{split}
&\E\big[|X_{s}^{(n)}|^{p}\big] < \infty\quad\text{for a.e.~$s\in [0,t]$}\quad\text{and}\\
\int_{0}^{t}&\big|\B_{t,s}\big(X_{s}^{(n)}\big)\big| + \big|\Sigma_{t,s}\big(X_{s}^{(n)}\big)\big|_{2}^{2}\,\mathrm{d}s < \infty\quad\text{and}\\
X_{t}^{(n+1)} &= \xi_{t} + \int_{0}^{t}\B_{t,s}\big(X_{s}^{(n)}\big)\,\mathrm{d}s + \int_{0}^{t}\Sigma_{t,s}\big(X_{s}^{(n)}\big)\,\mathrm{d}W_{s}\quad\text{a.s.}
\end{split}
\end{equation}
is Borel and has full measure for any $n\in\N_{0}$, is well-defined and satisfies
\begin{equation}\label{eq:growth estimate for Picard iterations 1}
\begin{split}
&\E\big[|X_{t}^{(n)} - \xi_{t}|^{p}\big]^{\frac{1}{p}} \leq k_{0}(t) + \sum_{i=1}^{n-1}\bigg(\int_{0}^{t}\R_{l^{2},i}(t,s)k_{0}(s)^{2}\,\mathrm{d}s\bigg)^{\frac{1}{2}}\\
& + \bigg(\int_{0}^{t}\R_{l^{2},n}(t,s)\E\big[|X_{s}^{(0)} - \xi_{s}|^{p}\big]^{\frac{2}{p}}\,\mathrm{d}s\bigg)^{\frac{1}{2}} + \sum_{i=1}^{n}\bigg(\int_{0}^{t}\R_{l^{2},i}(t,s)\E\big[|\xi_{s}|^{p}\big]^{\frac{2}{p}}\,\mathrm{d}s\bigg)^{\frac{1}{2}}
\end{split}
\end{equation}
for all $n\in\N$ and $t\in I_{\infty}$, where the Borel set $I_{\infty} :=\bigcap_{n\in\N} I_{n}$ has full measure.
\end{Proposition}

\begin{Remark}
By the definition of $\mathrm{I}_{l}$, the terms on the right-hand side in~\eqref{eq:growth estimate for solutions 1} cannot exceed $k_{0}(t) +  \mathrm{I}_{l}(t)(\esssup_{s\in [0,t]}k_{0}(s) + \esssup_{s\in [0,t]} \E[|\xi_{s}|^{p}]^{1/p})$ for any $t\in I$.
\end{Remark}

Next, let us consider a \emph{Lipschitz condition} on $\B$ and $\Sigma$:
\begin{enumerate}[label=(C.\arabic*), ref=C.\arabic*, leftmargin=\widthof{(C.2)} + \labelsep]
\setcounter{enumi}{1}
\item\label{co:2} There are $\lambda_{1}\in\mathcal{K}^{1}$ and $\lambda_{2}\in\mathcal{K}^{2}$ such that
\begin{equation}\label{eq:Lipschitz condition}
\begin{split}
\E\big[|\B_{t,s}(Y) &- \B_{t,s}(\tilde{Y})|^{p}\big]^{\frac{1}{p}}\mathbbm{1}_{\{1\}}(i)\\
&\quad + \E\big[|\Sigma_{t,s}(Y) - \Sigma_{t,s}(\tilde{Y})|_{2}^{p}\big]^{\frac{1}{p}}\mathbbm{1}_{\{2\}}(i) \leq \lambda_{i}(t,s)\E\big[|Y - \tilde{Y}|^{p}\big]^{\frac{1}{p}}
\end{split}
\end{equation}
for any $i\in\{1,2\}$, $s,t\in I$ with $s < t$ and $Y,\tilde{Y}\in\mathcal{D}$.
\end{enumerate}

\begin{Remark}
From~\eqref{co:2} the affine growth condition~\eqref{co:1} follows as soon as the two non-negative kernels $k_{1}$ and $k_{2}$ on $I$ given by
\begin{equation*}
k_{1}(t,s) := \E\big[|\B_{t,s}(\hat{Y})|^{p}\big]^{\frac{1}{p}}\quad\text{and}\quad k_{2}(t,s) := \E\big[|\Sigma_{t,s}(\hat{Y})|_{2}^{p}\big]^{\frac{1}{p}}
\end{equation*}
lie in $\mathcal{K}_{\infty}^{1}$ and $\mathcal{K}_{\infty}^{2}$, respectively, where $\hat{Y}\in\mathcal{D}$ satisfies $\E[|\hat{Y}|^{p}] < \infty$.
\end{Remark}

Under~\eqref{co:2}, we define $\lambda\in\mathcal{K}^{2}$ and $\mathrm{I}_{\lambda}:I\rightarrow [0,\infty]$ by~\eqref{eq:special transformed kernel} and~\eqref{eq:function series} when $l$, $l_{1}$ and $l_{2}$ are replaced by $\lambda$, $\lambda_{1}$ and $\lambda_{2}$, respectively. Further, we take an $E$-valued $\mathbb{F}$-progressively measurable process $\tilde{\xi}$ to state an \emph{$L^{p}$-comparison estimate}.

\begin{Corollary}\label{pr:comparison of solutions 1}
Suppose that~\eqref{co:2} is valid. Then any two solutions $X$ and $\tilde{X}$ to~\eqref{eq:stochastic Volterra equation} in $\mathcal{L}_{loc}^{\infty,p}(I\times\Omega,E)$ with respective value conditions $\xi$ and $\tilde{\xi}$ satisfy
\begin{equation}\label{eq:comparison estimate for solutions 1}
\E\big[|X_{t} - \tilde{X}_{t}|^{p}\big]^{\frac{1}{p}} \leq \E\big[|\xi_{t} - \tilde{\xi}_{t}|^{p}\big]^{\frac{1}{p}} + \sum_{n=1}^{\infty}\bigg(\int_{0}^{t}\R_{\lambda^{2},n}(t,s)\E\big[|\xi_{s} - \tilde{\xi}_{s}|^{p}\big]^{\frac{2}{p}}\,\mathrm{d}s\bigg)^{\frac{1}{2}}
\end{equation}
for a.e.~$t\in I$. In particular,~\eqref{eq:stochastic Volterra equation} admits, up to a weak modification, at most a unique solution in $\mathcal{L}_{loc}^{\infty,p}(I\times\Omega,E)$.
\end{Corollary}

\begin{Remark}
The series in~\eqref{eq:comparison estimate for solutions 1} is bounded by $\mathrm{I}_{\lambda}(t)\esssup_{s\in [0,t]}\E[|\xi_{s} - \tilde{\xi}_{s}|^{p}]^{1/p}$ for any $t\in I$.
\end{Remark}

Under the regularity condition on $\xi$, $\B$ and $\Sigma$ below, it follows from Lemma~\ref{le:regularity of the stochastic Volterra integral operator} that every solution to~\eqref{eq:stochastic Volterra equation} in $\mathcal{L}_{loc}^{\infty,p}(I\times\Omega,E)$ admits a continuous weak modification.
\begin{enumerate}[label=(C.\arabic*), ref=C.\arabic*, leftmargin=\widthof{(C.3)} + \labelsep]
\setcounter{enumi}{2}
\item\label{co:3} Additionally to~\eqref{co:1} there are measurable functions $f_{1},f_{2}:I\times I\times I\rightarrow [0,\infty]$ such that
\begin{equation}\label{eq:regularity condition}
\begin{split}
\E\big[|\B_{t,r}(Y) &- \B_{s,r}(Y)|^{p}\big]^{\frac{1}{p}}\mathbbm{1}_{\{1\}}(i)\\
&\quad + \E\big[|\Sigma_{t,r}(Y) - \Sigma_{s,r}(Y)|_{2}^{p}\big]^{\frac{1}{p}}\mathbbm{1}_{\{2\}}(i) \leq f_{i}(t,s,r)\big(1 + \E\big[|Y|^{p}\big]^{\frac{1}{p}}\big)
\end{split}
\end{equation}
for any $i\in\{1,2\}$, $r,s,t\in I$ with $r < s < t$ and $Y\in\mathcal{D}$. There are sequences $(t_{i})_{i\in\N}$, $(\beta_{i})_{i\in\N}$ and $(\hat{c}_{i})_{i\in\N}$ in $I$, $]\frac{1}{p},1]$ and $\Re_{+}$, respectively, such that
\begin{equation}\label{eq:co.4}
\begin{split}
\E\big[|\xi_{s} - \xi_{t}|^{p}\big]^{\frac{1}{p}} & + \int_{s}^{t}(k_{1}\vee l_{1})(t,\tilde{s})\,\mathrm{d}\tilde{s} + \bigg(\int_{s}^{t}(k_{2}\vee l_{2})(t,\tilde{s})^{2}\,\mathrm{d}\tilde{s}\bigg)^{\frac{1}{2}}\\
&\quad + \int_{0}^{s}f_{1}(t,s,r)\,\mathrm{d}r + \bigg(\int_{0}^{s}f_{2}(t,s,r)^{2}\,\mathrm{d}r\bigg)^{\frac{1}{2}} \leq \hat{c}_{i}(t - s)^{\beta_{i}}
\end{split}
\end{equation}
for all $i\in\N$ and $s,t\in [0,t_{i}]$ with $s\leq t$. Further, $t_{1} > 0$, $(t_{i})_{i\in\N}$ is strictly increasing and $\lim_{i\uparrow\infty} t_{i} = \sup I$.
\end{enumerate}

Finally, under~\eqref{co:2}, we infer from Proposition~\ref{pr:integral estimate for iterated kernels of first type} and Remark~\ref{re:integral estimate for iterated kernels of first type} that the measurable locally essentially bounded function $\mathrm{I}_{\lambda}$, given by~\eqref{eq:function series} for $l = \lambda$, satisfies
\begin{equation}\label{eq:error series estimate}
\esssup_{t\in [0,T]} \mathrm{I}_{\lambda}(t) \leq \sum_{n=1}^{\infty}\esssup_{t\in [0,T]}\bigg(\int_{0}^{t}\R_{\lambda^{2},n}(t,s)\,\mathrm{d}s\bigg)^{\frac{1}{2}} < \infty\quad\text{for any $T\in I$.}
\end{equation}
Hence, an \emph{existence result} with an \emph{error estimate} for the appearing Picard sequence and an \emph{analysis of the paths} can be stated.

For the coefficients~\eqref{eq:controlled and distribution-dependent coefficients}, the derived solutions become \emph{strong} if we choose $\mathbb{F}$ to be the right-continuous filtration of the augmented filtration of $(\mathcal{E}_{t}^{\xi,\alpha})_{t\in I}$, given by~\eqref{eq:natural filtration}.

\begin{Theorem}\label{th:existence of unique solutions 1}
Let~\eqref{co:1} and~\eqref{co:2} be valid, $\mathcal{L}^{p}(\Omega,E)\subset\mathcal{D}$ and $\xi, X^{(0)}\in\mathcal{L}_{loc}^{\infty,p}(I\times\Omega,E)$. Then, up to a weak modification, there is a unique solution
\begin{equation*}
\text{$X^{\xi}$ to~\eqref{eq:stochastic Volterra equation} in $\mathcal{L}_{loc}^{\infty,p}(I\times\Omega,E)$}
\end{equation*}
and the sequence $(X^{(n)})_{n\in\N}$ in $\mathcal{L}_{loc}^{\infty,p}(I\times\Omega,E)$, recursively given by~\eqref{eq:Picard iteration} for any $n\in\N_{0}$, converges in $p$th moment to $X^{\xi}$, locally essentially uniformly in $t\in I$. In fact,
\begin{equation}\label{eq:error moment estimate 1}
\E\big[|X_{t}^{(n)} - X_{t}^{\xi}|^{p}\big]^{\frac{1}{p}} \leq \sum_{i=n}^{\infty} \bigg(\int_{0}^{t}\R_{\lambda^{2},i}(t,s)\Delta(s)^{2}\,\mathrm{d}s\bigg)^{\frac{1}{2}}
\end{equation}
for all $n\in\N$ for a.e.~$t\in I$, where the measurable locally essentially bounded function $\Delta:I\rightarrow [0,\infty]$ is defined by $\Delta(t) := \E[|X_{t}^{(0)} - X_{t}^{(1)}|^{p}]^{1/p}$ and
\begin{equation}\label{eq:convergence of the error coefficients 1}
\lim_{n\uparrow\infty}\sum_{i=n}^{\infty}\esssup_{t\in [0,T]}\bigg(\int_{0}^{t}\R_{\lambda^{2},i}(t,s)\,\mathrm{d}s\bigg)^{\frac{1}{2}} = 0\quad\text{for all $T\in I$.}
\end{equation}
Let also~\eqref{co:3} hold. Then $X^{\xi}$ and $X^{(n)}$, where $n\in\N$, admit $\mathbb{F}$-adapted locally Hölder continuous weak modifications $\hat{X}^{\xi}$ and $\hat{X}^{(n)}$, respectively, satisfying
\begin{equation*}
\sup_{t\in [0,T]}\E\big[|\hat{X}_{t}^{\xi}|^{p}\big] < \infty\quad\text{and}\quad \sup_{n\in\N}\sup_{t\in [0,T]}\E\big[|\hat{X}_{t}^{(n)}|^{p}\big] < \infty
\end{equation*}
for any $T\in I$. Moreover, the paths of $\hat{X}^{\xi}$ and $\hat{X}^{(n)}$ are $\beta$-Hölder continuous on $[0,t_{i}]$ for any $i\in\N$, $\beta\in ]0,\beta_{i} - \frac{1}{p}[$ and $n\in\N$, and $\hat{X}^{\xi}$ is a regular solution to~\eqref{eq:stochastic Volterra equation}.
\end{Theorem}

\begin{Remark}\label{re:existence of unique solutions 1}
If $X^{(0)} = \xi$, then $\Delta(t) \leq k_{0}(t) + (\int_{0}^{t}l(t,s)^{2}\E[|\xi_{s}|^{p}]^{2/p}\,\mathrm{d}s)^{1/2}$ for a.e.~$t\in I$, by Proposition~\ref{pr:growth estimates for solutions and Picard iterations 1}. For $X^{(0)} = X^{\xi}$ we have $\Delta = 0$ a.e.~and $X^{(n)}$ becomes a weak modification of $X^{\xi}$ for all $n\in\N$.
\end{Remark}

For a relevant application, let us consider the \emph{controlled and distribution-dependent coefficients}~\eqref{eq:controlled and distribution-dependent coefficients} when the $p$th Wasserstein space, which is generally recalled at~\eqref{eq:Wasserstein metric}, is used.

\begin{Example}\label{ex:controlled and distribution-dependent stochastic Volterra equations}
Let $\mathcal{D} = \mathcal{L}^{p}(\Omega,E)$, $A$ be a non-empty, convex and closed subspace of a separable Banach space with complete norm $|\cdot|$ and $\alpha$ be an $A$-valued $\mathbb{F}$-progressively measurable process. Further, for some Borel measurable maps
\begin{equation*}
b:I\times I\times E\times A\times\mathcal{P}_{p}(E\times A)\rightarrow E,\quad \sigma:I\times I\times E\times A\times\mathcal{P}_{p}(E\times A)\rightarrow\mathcal{L}_{2}(\ell^{2},E)
\end{equation*}
let the representation~\eqref{eq:controlled and distribution-dependent coefficients} hold. Then $\B$ and $\Sigma$ are admissible, by Propositions~\ref{pr:admissible maps 2} and~\ref{pr:Borel measurability of the law map}, and~\eqref{co:1}-\eqref{co:3} follow from the subsequent respective conditions:
\begin{enumerate}[(1)]
\item There are $k_{1}\in\mathcal{K}_{\infty}^{1}$, $k_{2}\in\mathcal{K}_{\infty}^{2}$, $l_{1}\in\mathcal{K}^{1}$, $l_{2}\in\mathcal{K}^{2}$ and non-negative kernels $m_{1},m_{2}$ on $I$ such that 
\begin{align*}
&|b(t,s,x,a,\mu)|\mathbbm{1}_{\{1\}}(i) + |\sigma(t,s,x,a,\mu)|_{2}\mathbbm{1}_{\{2\}}(i)\\
&\leq k_{i}(t,s) + l_{i}(t,s)\big(|x| + \mathcal{W}_{p}(\mu_{1},\delta_{0})\big) + m_{i}(t,s)\big(|a| + \mathcal{W}_{p}(\mu_{2},\delta_{0})\big)
\end{align*}
for any $i\in\{1,2\}$, $s,t\in I$ with $s < t$, $x\in E$, $a\in A$ and $\mu\in\mathcal{P}_{p}(E\times A)$. In addition, $\esssup_{t\in [0,T]}\int_{0}^{t}m_{1}(t,s)\E[|\alpha_{s}|^{p}]^{1/p} + m_{2}(t,s)^{2}\E[|\alpha_{s}|^{p}]^{2/p}\,\mathrm{ds} < \infty$ for all $T\in I$.

\item There are $\lambda_{1}\in\mathcal{K}^{1}$, $\lambda_{2}\in\mathcal{K}^{2}$ and non-negative kernels $m_{3}, m_{4}$ on $I$ such that
\begin{align*}
&|b(t,s,x,a,\mu) - b(t,s,\tilde{x},a,\tilde{\mu})|\mathbbm{1}_{\{1\}}(i) + |\sigma(t,s,x,a,\mu) - \sigma(t,s,\tilde{x},a,\tilde{\mu})|_{2}\mathbbm{1}_{\{2\}}(i)\\
& \leq \big(\lambda_{i}(t,s) + m_{i + 2}(t,s)\mathcal{W}_{p}(\mu_{2},\delta_{0})\big)\big(|x - \tilde{x}| + \mathcal{W}_{p}(\mu,\tilde{\mu})\big) + m_{i + 2}(t,s)|a|\mathcal{W}_{p}(\mu,\tilde{\mu})
\end{align*}
for any $i\in\{1,2\}$, $s,t\in I$, $x,\tilde{x}\in E$, $a\in A$ and $\mu,\tilde{\mu}\in\mathcal{P}_{p}(E\times A)$ with $s < t$ and $\mu_{2} = \tilde{\mu}_{2}$. Further, the non-negative kernel $m_{i + 2,\alpha}$ on $I$ given by 
\begin{equation*}
\text{$m_{i + 2,\alpha}(t,s) := m_{i + 2}(t,s)\E\big[|\alpha_{s}|^{p}\big]^{\frac{1}{p}}$ lies in $\mathcal{K}^{i}$ for $i\in\{1,2\}$.}
\end{equation*}

\item Additionally to~(1) there are measurable functions $f_{1}, f_{2}, g_{1}, g_{2}:I\times I\times I\rightarrow [0,\infty]$ satisfying
\begin{align*}
&|b(t,r,x,a,\mu) - b(s,r,x,a,\mu)|\mathbbm{1}_{\{1\}}(i) + |\sigma(t,r,x,a,\mu) - \sigma(s,r,x,a,\mu)|_{2}\mathbbm{1}_{\{2\}}(i)\\
 &\leq f_{i}(t,s,r)\big(1 + |x| + \mathcal{W}_{p}(\mu_{1},\delta_{0})\big) + g_{i}(t,s,r)\big(|a| + \mathcal{W}_{p}(\mu_{2},\delta_{0})\big)
\end{align*}
for all $i\in\{1,2\}$, $r,s,t\in I$ with $r < s < t$, $x\in E$, $a\in A$ and $\mu\in\mathcal{P}_{p}(E\times A)$. There are sequences $(t_{i})_{i\in\N}$, $(\beta_{i})_{i\in\N}$ and $(\hat{c}_{i})_{i\in\N}$ in $I$, $]\frac{1}{p},1]$ and $\Re_{+}$, respectively, satisfying~\eqref{eq:co.4} and
\begin{equation*}
\begin{split}
& \int_{s}^{t}m_{1}(t,\tilde{s})\E\big[|\alpha_{\tilde{s}}|^{p}\big]^{\frac{1}{p}}\,\mathrm{d}\tilde{s} + \bigg(\int_{s}^{t}m_{2}(t,\tilde{s})^{2}\E\big[|\alpha_{\tilde{s}}|^{p}\big]^{\frac{2}{p}}\,\mathrm{d}\tilde{s}\bigg)^{\frac{1}{2}}\\
&\quad + \int_{0}^{s}g_{1}(t,s,r)\E\big[|\alpha_{r}|^{p}\big]^{\frac{1}{p}}\,\mathrm{d}r + \bigg(\int_{0}^{s}g_{2}(t,s,r)^{2}\E\big[|\alpha_{r}|^{p}\big]^{\frac{2}{p}}\,\mathrm{d}r\bigg)^{\frac{1}{2}} \leq \hat{c}_{i}(t - s)^{\beta_{i}}
\end{split}
\end{equation*}
for all $i\in\N$ and $s,t\in [0,t_{i}]$ with $s\leq t$. Moreover, $t_{1} > 0$, $(t_{i})_{i\in\N}$ is strictly increasing and $\lim_{i\uparrow\infty} t_{i} = \sup I$.
\end{enumerate}
So, if~(1) and~(2) hold and $\E[|\xi|^{p}]$ is locally essentially bounded, then Theorem~\ref{th:existence of unique solutions 1} yields a strong solution to~\eqref{eq:stochastic Volterra equation} in $\mathcal{L}_{loc}^{\infty,p}(I\times\Omega,E)$ that is unique, up to a weak modification. Further, there is a locally Hölder continuous regular solution $\hat{X}^{\xi}$ such that $\E[|\hat{X}^{\xi}|^{p}]$ is locally bounded once condition~(3) is also satisfied.
\end{Example}

\subsection{Solutions with locally integrable moment functions}\label{se:2.5}

The aim of this section is to deduce a unique solution $X$ to~\eqref{eq:stochastic Volterra equation} as limit of a Picard sequence such that the $p$th moment function $I\rightarrow [0,\infty]$, $t\mapsto \E[|X_{t}|^{p}]$ is locally integrable for $p\geq 2$.

To this end, let $\mathcal{L}_{loc}^{p}(I\times\Omega,E)$ denote the linear space of all $\mathbb{F}$-progressively measurable processes $X:I\times\Omega\rightarrow E$ satisfying $\int_{0}^{T}\E[|X_{t}|^{p}]\,\mathrm{d}t < \infty$ for all $T\in I$, equipped with the topology of convergence with respect to the seminorm
\begin{equation}\label{eq:seminorm for processes 2}
\mathcal{L}_{loc}^{p}(I\times\Omega,E)\rightarrow\Re_{+},\quad X\mapsto\bigg(\int_{0}^{T}\E\big[|X_{t}|^{p}\big]\,\mathrm{d}t\bigg)^{\frac{1}{p}}
\end{equation}
for each $T\in I$. That is, a sequence $(X^{(n)})_{n\in\N}$ converges to a process $X$ in this space if and only if $\lim_{n\uparrow\infty}\int_{0}^{T}\E[|X_{t}^{(n)} - X_{t}|^{p}]\,\mathrm{d}t = 0$ for any $T\in I$, and $\mathcal{L}_{loc}^{p}(I\times\Omega,E)$ is completely pseudometrisable, as we recall in the beginning of Section~\ref{se:5.4}.

Similar to~\eqref{co:1}, let us consider another \emph{affine growth condition} on $\B$ and $\Sigma$:
\begin{enumerate}[label=(C.\arabic*), ref=C.\arabic*, leftmargin=\widthof{(C.4)} + \labelsep]
\setcounter{enumi}{3}
\item\label{co:4} There are non-negative kernels $k_{1}, k_{2}$ on $I$ and $l_{1}\in\hat{\mathcal{K}}^{1}$, $l_{2}\in\hat{\mathcal{K}}^{2}$ satisfying the estimate~\eqref{eq:affine growth condition} and
\begin{equation}\label{eq:integrability condition}
\int_{0}^{T}\bigg(\int_{0}^{t}k_{1}(t,s)\,\mathrm{d}s\bigg)^{p} + \bigg(\int_{0}^{t}k_{2}(t,s)^{2}\,\mathrm{d}s\bigg)^{\frac{p}{2}}\,\mathrm{d}t < \infty
\end{equation}
for all $T\in I$. Moreover, $l_{1}\in\mathcal{K}^{1}$ and $l_{2}\in\mathcal{K}^{2}$ if $p > 2$.
\end{enumerate}

Under~\eqref{co:4}, we define a measurable locally $p$-fold integrable function $k_{0}:I\rightarrow [0,\infty]$ and $l\in\hat{\mathcal{K}}^{2}$ by~\eqref{eq:special integral function} and~\eqref{eq:special transformed kernel}, and we have $l\in\mathcal{K}^{2}$ if $p > 2$, bearing in mind that the statements after Definition~\ref{de:convex cones of non-negative kernels} hold.

Further, for each $n\in\N$ we introduce a non-negative kernel $l_{n,p}$ on $I$ by~\eqref{eq:transformed iterated kernel} and the successive formula when $\mu$ is the Lebesgue measure on $I$ and $\beta = 2$. That is,
\begin{equation}\label{eq:special transformed iterated kernel}
l_{n,p}(t,s) :=  \int_{s}^{t}\bigg(\int_{0}^{\tilde{s}}\R_{l^{2},n}(\tilde{s},r)\,\mathrm{d}r\bigg)^{\frac{p}{2} - 1}\R_{l^{2},n}(\tilde{s},s)\,\mathrm{d}\tilde{s},
\end{equation}
if $p > 2$, and $l_{n,p}(t,s) :=  \int_{s}^{t}\R_{l^{2},n}(\tilde{s},s)\,\mathrm{d}\tilde{s}$, if $p = 2$. Then Propositions~\ref{pr:integral estimate for iterated kernels of first type} and~\ref{pr:integral estimate for iterated kernels of second type} and Remarks~\ref{re:integral estimate for iterated kernels of first type} and~\ref{re:integral estimate for iterated kernels of second type} entail that the increasing function series $c_{l,p}:I\rightarrow\Re_{+}$ defined by
\begin{equation}\label{eq:special transformed function series}
c_{l,p}(t) := \sum_{n=1}^{\infty}\esssup_{s\in [0,t]}l_{n,p}(t,s)^{\frac{1}{p}}
\end{equation}
is indeed finite. This yields the following type of \emph{$L^{p}$-growth estimates} for solutions and Picard sequences to~\eqref{eq:stochastic Volterra equation}.

\begin{Proposition}\label{pr:growth estimates for solutions and Picard iterations 2}
Let~\eqref{co:4} hold and $\xi,X^{(0)}\in\mathcal{L}_{loc}^{p}(I\times\Omega,E)$. Then any solution $X$ to~\eqref{eq:stochastic Volterra equation} in $\mathcal{L}_{loc}^{p}(I\times\Omega,E)$ is subject to
\begin{equation}\label{eq:growth estimate for solutions 2}
\begin{split}
\bigg(\int_{0}^{t}\E\big[|X_{s} - \xi_{s}|^{p}\big]^{\frac{1}{p}}\,\mathrm{d}s\bigg)^{\frac{1}{p}} &\leq \bigg(\int_{0}^{t}k_{0}(s)^{p}\,\mathrm{d}s\bigg)^{\frac{1}{p}} + \sum_{n=1}^{\infty}\bigg(\int_{0}^{t}l_{n,p}(t,s)k_{0}(s)^{p}\,\mathrm{d}s\bigg)^{\frac{1}{p}}\\
&\quad + \sum_{n=1}^{\infty}\bigg(\int_{0}^{t}l_{n,p}(t,s)\E\big[|\xi_{s}|^{p}\big]\,\mathrm{d}s\bigg)^{\frac{1}{p}}
\end{split}
\end{equation}
for any $t\in I$. Further, if $\mathcal{L}^{p}(\Omega,E)\subset\mathcal{D}$, then the sequence $(X^{(n)})_{n\in\N}$ in $\mathcal{L}_{loc}^{p}(I\times\Omega,E)$, recursively given by~\eqref{eq:Picard iteration} for all $n\in\N_{0}$, is well-defined and satisfies
\begin{equation}\label{eq:growth estimate for Picard iterations 2}
\begin{split}
\bigg(\int_{0}^{t}&\E\big[|X_{s}^{(n)} - \xi_{s}|^{p}\big]\,\mathrm{d}s\bigg)^{\frac{1}{p}} \leq \bigg(\int_{0}^{t}k_{0}(s)^{p}\,\mathrm{d}s\bigg)^{\frac{1}{p}} + \sum_{i=1}^{n-1}\bigg(\int_{0}^{t}l_{i,p}(t,s)k_{0}(s)^{p}\,\mathrm{d}s\bigg)^{\frac{1}{p}}\\
&\quad + \bigg(\int_{0}^{t}l_{n,p}(t,s)\E\big[|X_{s}^{(0)} - \xi_{s}|^{p}\big]\,\mathrm{d}s\bigg)^{\frac{1}{p}} + \sum_{i=1}^{n}\bigg(\int_{0}^{t}l_{i,p}(t,s)\E\big[|\xi_{s}|^{p}\big]\,\mathrm{d}s\bigg)^{\frac{1}{p}}
\end{split}
\end{equation}
for all $n\in\N$ and $t\in I$. 
\end{Proposition}

\begin{Remark}
The expressions appearing on the right-hand side in~\eqref{eq:growth estimate for solutions 2} are bounded by $(1 + c_{l,p}(t))(\int_{0}^{t}k_{0}(s)^{p}\,\mathrm{d}s)^{1/p} + c_{l,p}(t)(\int_{0}^{t}\E[|\xi_{s}|^{p}]\,\mathrm{d}s)^{1/p}$ for every $t\in I$.
\end{Remark}

Let us introduce the subsequent \emph{Lipschitz condition} on $\B$ and $\Sigma$:
\begin{enumerate}[label=(C.\arabic*), ref=C.\arabic*, leftmargin=\widthof{(C.5)} + \labelsep]
\setcounter{enumi}{4}
\item\label{co:5} There are $\lambda_{1}\in\hat{\mathcal{K}}^{1}$ and $\lambda_{2}\in\hat{\mathcal{K}}^{2}$ such that the inequality~\eqref{eq:Lipschitz condition} is valid. Further, $\lambda_{1}\in\mathcal{K}^{1}$ and $\lambda_{2}\in\mathcal{K}^{2}$ if $p > 2$.
\end{enumerate}

Under~\eqref{co:5}, we define $\lambda\in\hat{\mathcal{K}}^{2}$, a non-negative kernel $\lambda_{n,p}$ on $I$, where $n\in\N$, and $c_{\lambda,p}:I\rightarrow\Re_{+}$ by~\eqref{eq:special transformed kernel},~\eqref{eq:special transformed iterated kernel} and~\eqref{eq:special transformed function series} when $l_{1} = \lambda_{1}$, $l_{2} = \lambda_{2}$ and $l = \lambda$. Then for any $E$-valued $\mathbb{F}$-progressively measurable process $\tilde{\xi}$ we obtain an \emph{$L^{p}$-comparison estimate}.

\begin{Proposition}\label{pr:comparison of solutions 2}
Assume that~\eqref{co:5} is valid. Then any two solutions $X$ and $\tilde{X}$ to~\eqref{eq:stochastic Volterra equation} in $\mathcal{L}_{loc}^{p}(I\times\Omega,E)$ with respective value conditions $\xi$ and $\tilde{\xi}$ satisfy
\begin{equation}\label{eq:comparison estimate for solutions 2}
\begin{split}
\bigg(\int_{0}^{t}\E\big[|X_{s} - \tilde{X}_{s}|^{p}\big]\,\mathrm{d}s\bigg)^{\frac{1}{p}} &\leq \bigg(\int_{0}^{t}\E\big[|\xi_{s} - \tilde{\xi}_{s}|^{p}\big]\,\mathrm{d}s\bigg)^{\frac{1}{p}}\\
&\quad + \sum_{n=1}^{\infty}\bigg(\int_{0}^{t}\lambda_{n,p}(t,s)\E\big[|\xi_{s} - \tilde{\xi}_{s}|^{p}\big]\,\mathrm{d}s\bigg)^{\frac{1}{p}}
\end{split}
\end{equation}
for every $t\in I$. In particular,~\eqref{eq:stochastic Volterra equation} admits, up to a weak modification, at most a unique solution in $\mathcal{L}_{loc}^{p}(I\times\Omega,E)$.
\end{Proposition}

\begin{Remark}
The series in~\eqref{eq:comparison estimate for solutions 2} is bounded by $c_{\lambda,p}(t)(\int_{0}^{t}\E[|\xi_{s} - \tilde{\xi}_{s}|^{p}]\,\mathrm{d}s)^{1/p}$ for any $t\in I$, according to the definition of $c_{\lambda,p}$.
\end{Remark}

Now we are in a position to state another \emph{existence result} with an \emph{error estimate} for the Picard iteration, and for the coefficients~\eqref{eq:controlled and distribution-dependent coefficients} the solutions can be chosen to be \emph{strong}.

\begin{Theorem}\label{th:existence of unique solutions 2}
Let~\eqref{co:4} and~\eqref{co:5} be hold, $\mathcal{L}^{p}(\Omega,E)\subset\mathcal{D}$ and $\xi, X^{(0)}\in\mathcal{L}_{loc}^{p}(I\times\Omega,E)$. Then, up to a weak modification,~\eqref{eq:stochastic Volterra equation} admits a unique solution
\begin{equation*}
\text{$X^{\xi}$ in $\mathcal{L}_{loc}^{p}(I\times\Omega,E)$}
\end{equation*}
and the Picard sequence $(X^{(n)})_{n\in\N}$ in $\mathcal{L}_{loc}^{p}(I\times\Omega,E)$, recursively given by~\eqref{eq:Picard iteration} for all $n\in\N_{0}$, converges to $X^{\xi}$. More precisely,
\begin{equation}\label{eq:error moment estimate 2}
\bigg(\int_{0}^{t}\E\big[|X_{s}^{(n)} - X_{s}^{\xi}|^{p}\big]\,\mathrm{d}s\bigg)^{\frac{1}{p}} \leq \sum_{i=n}^{\infty}\bigg(\int_{0}^{t}\lambda_{i,p}(t,s)\Delta(s)^{p}\,\mathrm{d}s\bigg)^{\frac{1}{p}}
\end{equation}
for all $n\in\N$ and $t\in I$ with the measurable locally $p$-fold integrable function $\Delta:I\rightarrow [0,\infty]$ given by $\Delta(t) := \E[|X_{s}^{(0)} - X_{s}^{(1)}|^{p}]^{1/p}$, and
\begin{equation}\label{eq:convergence of the error coefficients 2}
\lim_{n\uparrow\infty}\sum_{i=n}^{\infty}\esssup_{s\in [0,t]}\lambda_{i,p}(t,s)^{\frac{1}{p}} = 0\quad\text{for each $t\in I$.}
\end{equation}
\end{Theorem}

\begin{Remark}
For $X^{(0)} = \xi$ Proposition~\ref{pr:growth estimates for solutions and Picard iterations 2} gives $(\int_{0}^{t}\Delta(s)^{p}\,\mathrm{d}s)^{1/p} \leq (\int_{0}^{t}k_{0}(s)^{p}\,\mathrm{d}s)^{1/p}$ $ +\, (\int_{0}^{t}l_{1,p}(t,s)\E[|\xi_{s}|^{p}]\,\mathrm{d}s)^{1/p}$ for all $t\in I$. If instead $X^{(0)} = X^{\xi}$, then $\Delta = 0$ a.e.
\end{Remark}

Based on the progressive $\sigma$-field $\mathcal{A}$ on $I\times\Omega$, let us conclude with \emph{random coefficients that could be of affine type}.

\begin{Example}\label{ex:random coefficients}
We recall the Banach space $\mathcal{L}(E,\tilde{E})$ of all linear continuous maps on $E$ with values in $\tilde{E}$, endowed with the operator norm, whenever $\tilde{E}$ is a Banach space. Thus, let $\mathcal{D} = \mathcal{L}^{1}(\Omega,E)$ and
\begin{equation*}
f_{1},f_{2},g_{1},g_{2}:E\rightarrow E
\end{equation*}
be Lipschitz continuous with Lipschitz constant one and vanish at the origin. Further, for $i\in\{1,2\}$ let $\kappa,\beta^{(i)},\eta$ and $\sigma^{(i)}$ be $\mathcal{B}(I)\otimes\mathcal{A}$-measurable maps on $I\times I\times\Omega$ with respective values in $E$, $\mathcal{L}(E,E)$, $\mathcal{L}_{2}(\ell^{2},E)$ and $\mathcal{L}(E,\mathcal{L}_{2}(\ell^{2},E))$ such that
\begin{equation}\label{eq:random coefficients}
\begin{split}
\B_{t,s}(X_{s}) &= \kappa_{t,s} + \beta_{t,s}^{(1)}\big(f_{1}(X_{s})\big) + \beta_{t,s}^{(2)}\big(f_{2}\big(\E\big[X_{s}\big]\big)\big),\\
\Sigma_{t,s}(X_{s}) &= \eta_{t,s} + \sigma_{t,s}^{(1)}\big(g_{1}(X_{s})\big)+ \sigma_{t,s}^{(2)}\big(g_{2}\big(\E\big[X_{s}\big]\big)\big)
\end{split}
\end{equation}
for all $s,t\in I$ and $X_{s}\in\mathcal{D}$. Then Proposition~\ref{pr:admissible maps 1} ensures that $\B$ and $\Sigma$ are admissible, and by using $|\cdot|$ and $|\cdot|_{2}$ for the respective operator norms on $\mathcal{L}(E,E)$ and $\mathcal{L}(E,\mathcal{L}_{2}(\ell^{2},E))$, we see that the following two assertions hold:
\begin{enumerate}[(1)]
\item Define two non-negative kernels $k_{1}$ and $k_{2}$ on $I$ by $k_{1}(t,s) := \E[|\kappa_{t,s}|^{p}]^{1/p}$ and $k_{2}(t,s) := \E[|\eta_{t,s}|_{2}^{p}]^{1/p}$ and let $l_{1}$ and $l_{2}$ be non-negative kernels on $I$ such that
\begin{equation*}
|\beta_{t,s}^{(1)}| + \E\big[|\beta_{t,s}^{(2)}|^{p}\big]^{\frac{1}{p}} \leq l_{1}(t,s),\quad |\sigma_{t,s}^{(1)}|_{2} + \E\big[|\sigma_{t,s}^{(2)}|_{2}^{p}\big]^{\frac{1}{p}} \leq l_{2}(t,s)
\end{equation*}
for any $s,t\in I$ with $s < t$. Then the affine growth estimate~\eqref{eq:affine growth condition} is valid and the Lipschitz condition~\eqref{eq:Lipschitz condition} holds for $\lambda_{1} = l_{1}$ and $\lambda_{2} = l_{2}$.

\item Define two functions $f_{1},f_{2}:I\times I\times I\rightarrow [0,\infty]$ by $f_{1}(t,s,r) := \E[|\kappa_{t,r} - \kappa_{s,r}|^{p}]^{1/p}$ and $f_{2}(t,s,r) := \E[|\eta_{t,r} - \eta_{s,r}|_{2}^{p}]^{1/p}$ and assume that
\begin{align*}
|\beta_{t,r}^{(1)} - \beta_{s,r}^{(1)}| + \E\big[|\beta_{t,r}^{(2)} - \beta_{s,r}^{(2)}|^{p}\big]^{\frac{1}{p}} \leq g_{1}(t,s,r),\\
|\sigma_{t,r}^{(1)} - \sigma_{s,r}^{(1)}|_{2} + \E\big[|\sigma_{t,r}^{(2)} - \sigma_{s,r}^{(2)}|_{2}^{p}\big]^{\frac{1}{p}} \leq g_{2}(t,s,r)
\end{align*}
for all $r,s,t\in I$ with $r < s < t$ and some $[0,\infty]$-valued measurable functions $g_{1},g_{2}$ on $I\times I\times I$. Then the terms on the left-hand side in~\eqref{eq:regularity condition} do not exceed
\begin{equation*}
\begin{split}
f_{i}(t,s,r) + g_{i}(t,s,r)\E\big[|Y|^{p}\big]^{\frac{1}{p}}
\end{split}
\end{equation*}
for any $i\in\{1,2\}$, $r,s,t\in I$ with $r < s < t$ and $Y\in\mathcal{L}^{1}(\Omega,E)$.
\end{enumerate}
So, let the inequality~\eqref{eq:integrability condition} hold, $l_{1}\in\hat{\mathcal{K}}^{1}$ and $l_{2}\in\hat{\mathcal{K}}^{2}$, and $l_{1}\in\mathcal{K}^{1}$ and $l_{2}\in\mathcal{K}^{2}$ if $p > 2$. If in addition $\E[|\xi|^{p}]$ is locally $p$-fold integrable, then Theorem~\ref{th:existence of unique solutions 2} yields a solution to~\eqref{eq:stochastic Volterra equation} in $\mathcal{L}_{loc}^{p}(I\times\Omega,E)$ that is unique, up to a weak modification.
\end{Example}

\section{General results on stochastic processes}\label{se:3}

In what follows, $(\Omega,\mathcal{F},\P)$ is an arbitrary probability space that does not need to be complete, except in Section~\ref{se:3.4} when the filtration $\mathbb{F}$ is used.

\subsection{Weakly modified processes}\label{se:3.1}

Here, $I$ and $E$ are merely topological spaces and $\mu$ is a Borel measure on $I$. Then by an $E$-valued process we shall mean a map $X:I\times\Omega\rightarrow E$, $(t,\omega)\mapsto X_{t}(\omega)$ such that $X_{t}$ is Borel measurable for any $t\in I$.

\begin{Definition}\label{de:weak modification}
An $E$-valued process $\tilde{X}$ is said to be a \emph{weak modification} of another $E$-valued process $X$ with respect to $\mu$ if $X_{t} = \tilde{X}_{t}$ a.s.~for $\mu$-a.e.~$t\in I$.
\end{Definition}

When $I$ is metrisable and $E$ is a Hausdorff space, we seek to give sufficient conditions, under which any two weakly modified $E$-valued processes are indistinguishable.

To this end, we recall that if $\mu$ is inner regular, then the support of $\mu$, denoted by $\mathrm{supp}\,\mu$, can be defined to be the smallest closed set in $I$ of full $\mu$-measure. In fact, for a metric inducing the topology of $I$ let $B_{\delta}(t)$ denote the open ball around any point $t\in I$ with radius $\delta > 0$. Then the following representation holds.

\begin{Lemma}\label{le:support of a measure}
If the Borel measure $\mu$ on the metrisable space $I$ is inner regular, then $\mathrm{supp}\,\mu = \{t\in I\,|\,\forall\delta > 0: \mu(B_{\delta}(t)) > 0\}$.
\end{Lemma}

\begin{proof}
First, $\{t\in I\,|\,\exists\delta > 0:\mu(B_{\delta}(t)) = 0\}$ is open, since for any $t\in I$ and $\delta > 0$ with $\mu(B_{\delta}(t)) = 0$ the triangle inequality entails that $\mu(B_{\delta/2}(s)) = 0$ for all $s\in B_{\delta/2}(t)$.

Secondly, to show that $\{t\in I\,|\,\forall\delta > 0: \mu(B_{\delta}(t)) > 0\}$ has full measure, it suffices to check that $\mu(K) = 0$ for any compact set $K$ in $\{t\in I\,|\,\exists\delta > 0:\mu(B_{\delta}(t)) = 0\}$. However, as $K$ is compact, there are $n\in\N$, $t_{1},\dots,t_{n}\in K$ and $\delta_{1},\dots,\delta_{n} > 0$ such that $\mu(B_{\delta_{i}}(t_{i})) = 0$ for all $i\in\{1,\dots,n\}$ and $K\subset \bigcup_{i=1}^{n}B_{\delta_{i}}(t_{i})$, which yields that $\mu(K) = 0$.

Finally, if $C$ is a closed set in $I$ of full $\mu$-measure, then for any $t\in C^{c}$ there is $\delta > 0$ such that $B_{\delta}(t)\subset C^{c}$, which implies that $\mu(B_{\delta}(t)) = 0$. So, $\{t\in I\,|\,\forall\delta > 0: \mu(B_{\delta}(t)) > 0\}$ is a subset of $C$.
\end{proof}

Next, we introduce a sequential continuity notion relative to a function $\varphi:I\times I\rightarrow\Re$ that is continuous in the first variable, that is, $\varphi(\cdot,t)$ is continuous for any $t\in I$.

\begin{Definition}\label{de:sequential continuity with respect to a function}
A map $x:I\rightarrow E$ is said to be \emph{sequentially continuous} with respect to $\varphi$ if any sequence $(t_{n})_{n\in\N}$ in $I$ that converges to some $t\in I$ satisfies $\lim_{n\uparrow\infty} x(t_{n}) = x(t)$ as soon as $\varphi(t_{n},t) > 0$ for all $n\in\N$.
\end{Definition}

The set $I_{\varphi} := \{t\in I\,|\,\forall \delta > 0\, \exists s\in B_{\delta}(t): \varphi(s,t) > 0\}$ consists of all points to which this continuity concept applies. In particular, if $\varphi$ takes only positive values, then we recover the ordinary sequential continuity of a map and $I_{\varphi} = I$.

\begin{Example}[Multivariate right-continuity]
For $d\in\N$ let $I$ be a subspace of $\Re^{d}$ and $\varphi(s,t) = \min\{s_{1} - t_{1},\dots,s_{d} - t_{d}\}$ for any $s,t\in I$, which entails that
\begin{equation*}
\varphi(s,t) > 0\quad\Leftrightarrow\quad s_{i} > t_{i}\quad\text{for all $i\in\{1,\dots,d\}$.}
\end{equation*}
Then Definition~\ref{de:sequential continuity with respect to a function} states the right-continuity of a map in a multivariate sense and $I^{\circ}\subset I_{\varphi}$. Further, if there are intervals $I_{1},\dots,I_{d}$ in $\Re$ such that $I = I_{1}\times\cdots\times I_{d}$, then
\begin{equation*}
I_{\varphi} = \{s_{1}\in I_{1}\,|\,s_{1} < \sup I_{1}\}\times\cdots\times\{s_{d}\in I_{d}\,|\,s_{d} < \sup I_{d}\}.
\end{equation*}
\end{Example}

The continuity condition on $\varphi$ ensures that any element in $I_{\varphi}$ can be approximated by a suitable sequence in any dense set.

\begin{Lemma}\label{le:auxiliary density lemma}
Let $J$ be a dense set in the metrisable space $I$. Then each $t\in I_{\varphi}$ is the limit of a sequence $(t_{n})_{n\in\N}$ in $J$ satisfying $\varphi(t_{n},t) > 0$ for any $n\in\N$.
\end{Lemma}

\begin{proof}
For each $n\in\N$ the definition of $I_{\varphi}$ gives $s_{n}\in B_{1/n}(t)$ with $\varphi(s_{n},t) > 0$. By the continuity of $\varphi(\cdot,t)$, there is $\delta_{n}\in ]0,1/n]$ such that $\varphi(s,t) > 0$ for all $s\in B_{\delta_{n}}(s_{n})$. So, we may take $t_{n}\in J$ with $t_{n}\in B_{\delta_{n}}(s_{n})$ and obtain that $t_{n}\in B_{2/n}(t)$ and $\varphi(t_{n},t) > 0$.
\end{proof}

Sequential continuity with respect to $\varphi$ and sequential continuity on $I_{\varphi}^{c}$ outside a countable set relate weak modifications with indistinguishable processes.

\begin{Proposition}\label{pr:weakly modified processes}
Let the metrisable space $I$ be separable, $E$ be a Hausdorff space and $\mu$ be inner regular and have full support, that means,
\begin{equation*}
\mathrm{supp}\,\mu = I.
\end{equation*}
Then any two $E$-valued processes $X$ and $\tilde{X}$ that are weak modifications of each other relative to $\mu$ are indistinguishable if the following two conditions holds:
\begin{enumerate}[(i)]
\item The paths of $X$ and $\tilde{X}$ are sequentially continuous with respect to $\varphi$.

\item There is a countable set $J_{\varphi}$ in $I_{\varphi}^{c}$ such that $X_{t} = \tilde{X}_{t}$ a.s.~for all $t\in J_{\varphi}$ and the paths of $X$ and $\tilde{X}$ are sequentially continuous on $I_{\varphi}^{c}\setminus J_{\varphi}$.
\end{enumerate}
\end{Proposition}

\begin{proof}
Since $\mathrm{supp}\,\mu = I$, any Borel set in $I$ with full $\mu$-measure is dense in $I$. Hence, it suffices to prove that any countable dense set $J$ in $I$ such that $X_{t} = \tilde{X}_{t}$ a.s.~for all $t\in J$ and $J_{\varphi}\subset J$ satisfies
\begin{equation}\label{eq:countable intersection}
\bigcap_{t\in I}\{X_{t} = \tilde{X}_{t}\} = \bigcap_{t\in J}\{X_{t} = \tilde{X}_{t}\}.
\end{equation}
Once this is shown, for each $t\in J$ we take a null set $N_{t}\in\mathcal{F}$ such that $X_{t} = \tilde{X}_{t}$ on $N_{t}^{c}$ and obtain $N^{c}\subset\bigcap_{t\in I}\{X_{t} = \tilde{X}_{t}\}$ for the null set $N:=\bigcup_{t\in J}N_{t}$.

To prove~\eqref{eq:countable intersection}, let $\omega\in\Omega$ satisfy $X_{s}(\omega) = \tilde{X}_{s}(\omega)$ for all $s\in J$ and fix $t\in I$. In the case that $t\in I_{\varphi}$ Lemma~\ref{le:auxiliary density lemma} yields a sequence $(t_{n})_{n\in\N}$ in $J$ converging to $t$ such that $\varphi(t_{n},t) > 0$ for all $n\in\N$. As $E$ is a Hausdorff space,
\begin{equation}\label{eq:sequential limit}
X_{t}(\omega) = \lim_{n\uparrow\infty} X_{t_{n}}(\omega) = \lim_{n\uparrow\infty} \tilde{X}_{t_{n}}(\omega) = \tilde{X}_{t}(\omega).
\end{equation}
If $t\in J_{\varphi}$, then $X_{t}(\omega) = \tilde{X}_{t}(\omega)$, since $J_{\varphi}\subset J$. Otherwise $t\in I_{\varphi}^{c}\setminus J_{\varphi}$ and for any sequence $(t_{n})_{n\in\N}$ in $J$ converging to $t$ the sequential continuity of $X(\omega)$ and $\tilde{X}(\omega)$ implies~\eqref{eq:sequential limit}. 
\end{proof}

\begin{Remark}
If in fact $X$ and $\tilde{X}$ have sequentially continuous paths, then they are indistinguishable, as in this case conditions~(i) and~(ii) hold, regardless of the choice of $\varphi$.
\end{Remark}

\begin{Example}
Let $I$ be a subspace of $\Re$ and $\varphi(s,t) = s - t$ for all $s,t\in I$. Then $X$ and $\tilde{X}$ are indistinguishable, by Proposition~\ref{pr:weakly modified processes}, if the subsequent two conditions hold:
\begin{enumerate}[(i)]
\item $X$ and $\tilde{X}$ have right-continuous paths.

\item If $T := \sup I$ is a maximum of $I$, then $X_{T} = \tilde{X}_{T}$ a.s.~or the paths of $X$ and $\tilde{X}$ are left-continuous at $T$.
\end{enumerate}
\end{Example}

\subsection{Approximations of processes and distribution maps}\label{se:3.2}

Let $I$ be just a non-empty set endowed with a $\sigma$-field $\mathcal{I}$, $E$ be a separable metrisable space, $d$ be a metric inducing the topology of $E$ and $p\geq 1$. Under convexity conditions on $E$ and $d$, we can \emph{approximate the identity map} on $E$.

\begin{Proposition}\label{pr:approximation of the identity map}
Let $E$ be a star convex set in a linear space with $x_{0}\in E$ as center such that
\begin{equation}\label{eq:convexity condition on the metric}
d((1-\lambda)x_{0} + \lambda x,y) \leq (1-\lambda)d(x_{0},y) + \lambda d(x,y)
\end{equation}
for all $\lambda\in [0,1]$ and $x,y\in E$. Then there is a sequence $(\varphi_{n})_{n\in\N}$ of $E$-valued Borel measurable maps on $E$, each taking finitely many values, such that
\begin{equation}\label{eq:properties of the approximating sequence}
d(\varphi_{k}(x),x_{0}) \leq d(x,x_{0})\quad\text{and}\quad d(\varphi_{k+1}(x),x) \leq d(\varphi_{k}(x),x)
\end{equation}
and $\lim_{n\uparrow\infty}\varphi_{n}(x) = x$ for all $k\in\N$ and $x\in E$.
\end{Proposition}

\begin{proof}
For a countable dense set $D_{0}$ in $E$ let $D:=\{(1 - q)x_{0} + qz\,|\,q\in [0,1]\cap\mathbb{Q},\,z\in D_{0}\}$. Then for any $x\in E$ and $\varepsilon > 0$ there is $y\in D$ such that
\begin{equation}\label{eq:required metric inequalities}
d(y,x_{0}) \leq d(x,x_{0})\quad\text{and}\quad d(y,x) < \varepsilon.
\end{equation}
Indeed, let $z\in D_{0}$ satisfy $d(z,x) < \frac{\varepsilon}{2}$. If $d(z,x_{0}) \leq d(x,x_{0})$, then we set $y := z$. Otherwise, $d(y_{\lambda},x_{0}) \leq \lambda d(z,x_{0})$ and $d(y_{\lambda},x) \leq (1-\lambda)d(z,x_{0}) + \frac{\varepsilon}{2}$ for $\lambda\in [0,1]$ and $y_{\lambda} := (1 - \lambda)x_{0} + \lambda z$. As
\begin{equation*}
\lim_{\lambda\uparrow \lambda_{0}} (1 - \lambda)d(z,x_{0}) = d(z,x_{0}) - d(x,x_{0}) < \frac{\varepsilon}{2}\quad\text{for}\quad \lambda_{0} := \frac{d(x,x_{0})}{d(z,x_{0})},
\end{equation*}
there is $q\in [0,1]\cap\mathbb{Q}$ such that $y := y_{q}$ satisfies~\eqref{eq:required metric inequalities}. Further, as $D$ is countable, we can take an increasing sequence $(D_{n})_{n\in\N}$ of non-empty finite sets in $D$ such that $x_{0}\in D_{1}$ and $\bigcup_{n\in\N} D_{n} = D$.

For each $n\in\N$ there are $N_{n}\in\N$ and pairwise distinct $x_{1,n},\dots,x_{N_{n},n}\in D_{n}$ such that $D_{n} = \{x_{1,n},\dots,x_{N_{n},n}\}$. So, for any $x\in E$ the set $D_{n}(x) :=\{y\in D_{n}\,|\, d(y,x_{0})\leq d(x,x_{0})\}$ contains $x_{0}$ and there is a unique $i_{n}(x)\in\{1,\dots,N_{n}\}$ such that
\begin{equation*}
x_{i_{n}(x),n}\in D_{n}(x)\quad\text{and}\quad d(x_{i_{n}(x),n},x) = \min_{y\in D_{n}(x)} d(y,x)
\end{equation*}
and if $N_{n}\geq 2$ and $i_{n}(x)\geq 2$, then there is no $i\in\{1,\dots,i_{n}(x) - 1\}$ satisfying $d(x_{i,n},x_{0})$ $\leq d(x,x_{0})$ and $d(x_{i,n},x) \leq d(y,x)$ for all $y\in D_{n}(x)$. The resulting map $\varphi_{n}:E\rightarrow D_{n}$ given by $\varphi_{n}(x) := x_{i_{n}(x),n}$ is Borel measurable.

This follows from the fact that for any set $C$ in $D_{n}$ with $x_{0}\in C$ the intersection $C_{j,n}$ of $\{x\in E\,|\,\varphi_{n}(x) = x_{j,n}\}$ and $\{x\in E\,|\, D_{n}(x) = C\}$ is Borel for any $j\in\{1,\dots,N_{n}\}$. This in turn is a consequence of the representation
\begin{equation*}
\bigcup_{i=1}^{j}C_{i,n} = \bigcup_{i=1}^{j}\big\{x\in E\,|\,D_{n}(x) = C,\,x_{i,n}\in C,\, d(x_{i,n},x) = \min_{y\in C} d(y,x)\big\},
\end{equation*}
since $\{x\in E\,|\, D_{n}(x) = C\}$ agrees with the set of all $x\in E$  with $\max_{y\in C}d(y,x_{0})$ $\leq d(x,x_{0})$ and $\min_{y\in D_{n}\setminus C} d(y,x_{0}) > d(x,x_{0})$, whenever $C\neq D_{n}$, and $\{x\in E\,|\,D_{n}(x) = D_{n}\}$ is the set of all $x\in E$ with $\max_{y\in D_{n}} d(y,x_{0}) \leq d(x,x_{0})$.

Finally, as $D_{k}\subset D_{k+1}$ for any $k\in\N$, the properties~\eqref{eq:properties of the approximating sequence} follow from the definition of $(\varphi_{n})_{n\in\N}$ and we conclude that $\lim_{n\uparrow\infty}d(\varphi_{n}(x),x)$ $= \inf_{y\in D:\, d(y,x_{0})\leq d(x,x_{0})} d(y,x) = 0$ for all $x\in E$.
\end{proof}

\begin{Remark}
If there is a norm $|\cdot|$ on the underlying linear space such that $d(x,y)$ $= |x - y|$ for all $x,y\in E$, then the convexity condition~\eqref{eq:convexity condition on the metric} is satisfied.
\end{Remark}

\begin{Corollary}\label{co:approximation of product measurable processes}
Let the assumptions of Proposition~\ref{pr:approximation of the identity map} hold and $X:I\times\Omega\rightarrow E$ be product measurable. Then the sequence $(X^{(n)})_{n\in\N}$ of $E$-valued product measurable processes on $I\times\Omega$ given by $X_{t}^{(n)} := \varphi_{n}(X_{t})$ for all $n\in\N$ satisfies the following two properties:
\begin{enumerate}[(i)]
\item For each $n\in\N$ there are $N_{n}\in\N$ and pairwise distinct $x_{1,n},\dots,x_{N_{n},n}\in E$ such that
\begin{equation*}
X_{t}^{(n)} = \sum_{i=1}^{N_{n}} x_{i,n}\mathbbm{1}_{B_{i,n}}(X_{t})\quad\text{for all $t\in I$}
\end{equation*}
with the Borel sets $B_{i,n} :=\{x\in E\,|\,\varphi_{n}(x) = x_{i,n}\}$, where $i\in\{1,\dots,N_{n}\}$, that form a disjoint decomposition of $E$.

\item $d(X_{t}^{(k)},x_{0}) \leq d(X_{t},x_{0})$ and $d(X_{t}^{(k + 1)},X_{t}) \leq d(X_{t}^{(k)},X_{t})$ and $(X_{t}^{(n)})_{n\in\N}$ converges pointwise to $X_{t}$ for any $k\in\N$ and $t\in I$.
\end{enumerate}
\end{Corollary}

\begin{proof}
By Proposition~\ref{pr:approximation of the identity map}, for each $n\in\N$ there are $N_{n}\in\N$ and pairwise distinct $x_{1,n},\dots,x_{N_{n},n}\in E$ such that $\varphi_{n}(E) = \{x_{1,n},\dots,x_{N_{n},n}\}$ and $X_{t}^{(n)} = \sum_{i=1}^{N_{n}}x_{i,n}\mathbbm{1}_{\{X_{t}^{(n)} = x_{i,n}\}}$ for all $t\in I$, which yields~(i). Statement~(ii) follows directly from the properties~\eqref{eq:properties of the approximating sequence} of the sequence $(\varphi_{n})_{n\in\N}$.
\end{proof}

We turn to the convex space $\mathcal{P}_{p}(E)$ of all $\mu\in\mathcal{P}_{0}(E)$ that satisfy $\int_{E}d(x,x_{0})^{p}\mu(\mathrm{d}x) < \infty$ for some $x_{0}\in E$, endowed with the \emph{$p$th Wasserstein metric} given by
\begin{equation}\label{eq:Wasserstein metric}
\mathcal{W}_{p}(\mu,\nu) := \inf_{\theta\in\mathcal{P}(\mu,\nu)}\bigg(\int_{E\times E}d(x,y)^{p}\mathrm{d}\theta(x,y)\bigg)^{\frac{1}{p}},
\end{equation}
where $\mathcal{P}(\mu,\nu)$ denotes the convex space of all probability measures in $\mathcal{P}_{0}(E\times E)$ whose first and second marginal distributions are $\mu$ and $\nu$, respectively, for all $\mu,\nu\in\mathcal{P}_{p}(E)$.

\begin{Lemma}\label{le:compactness of a set of probability measures}
Let $d$ be complete. Then for any $n\in\N$, $\alpha\in ]0,1]^{n}$ with $\sum_{i=1}^{n}\alpha_{i} = 1$ and $\nu\in\mathcal{P}_{p}(E)$, the set of all $(\nu_{1},\dots,\nu_{n})\in\prod_{i=1}^{n}\mathcal{P}_{p}(E)$ satisfying $\sum_{i=1}^{n}\alpha_{i}\nu_{i} = \nu$ is compact.
\end{Lemma}

\begin{proof}
As $E$ is a Polish space, we may infer the assertion from the characterisation of convergence in $\mathcal{P}_{p}(E)$ stated in~\cite[Theorem~7.12]{Vil03} and Prokhorov's theorem. Namely, each sequence $(\nu_{1}^{(k)},\dots,\nu_{n}^{(k)})_{k\in\N}$ in the set $\mathcal{K}_{\nu}^{n,\alpha}$ appearing in the claim satisfies $\alpha_{i}\nu_{i}^{(k)} \leq \nu$ for any $k\in\N$ and $i\in\{1,\dots,n\}$.

This implies that $(\nu_{i}^{(k)})_{k\in\N}$ is tight for all $i\in\{1,\dots,n\}$, since any finite Borel measure on $E$ is inner regular, as shown in~\cite[Lemma 26.2]{Bau01}. By Prokhorov's theorem, there is a strictly increasing sequence $(k_{l})_{l\in\N}$ in $\N$ such that $(\nu_{i}^{(k_{l})})_{l\in\N}$ converges weakly to some $\nu_{i}\in\mathcal{P}_{0}(E)$ for each fixed $i\in\{1,\dots,n\}$.

As $\alpha_{i}\int_{B}d(x,x_{0})^{p}\nu_{i}^{(k)}(\mathrm{d}x) \leq \int_{B}d(x,x_{0})^{p}\,\nu(\mathrm{d}x)$ for all $k\in\N$ and $B\in\mathcal{B}(E)$ and given $x_{0}\in E$, we have $\nu_{i}\in\mathcal{P}_{p}(E)$ and $\lim_{r\uparrow\infty}\sup_{k\in\N}\int_{E\setminus B_{r^{1/p}}(x_{0})} d(x,x_{0})^{p}\nu_{i}^{(k)}(\mathrm{d}x) = 0$ with the open ball $B_{\delta}(x_{0})$ around $x_{0}$ with any radius $\delta > 0$, by dominated convergence.

Hence, $(\nu_{i}^{(k_{l})})_{l\in\N}$ converges to $\nu_{i}$ in $\mathcal{P}_{p}(E)$, by Theorem~7.12 in~\cite{Vil03}. Eventually, $\sum_{i=1}^{n}\alpha_{i}\nu_{i} = \nu$ holds, because $\mathcal{K}_{\nu}^{n,\alpha}$ is closed as preimage of $\{\nu\}$ under the continuous map $\prod_{i=1}^{n}\mathcal{P}_{p}(E)\rightarrow\mathcal{P}_{p}(E)$, $(\nu_{1},\dots,\nu_{n})\mapsto\sum_{i=1}^{n}\alpha_{i}\nu_{i}$. So, $\mathcal{K}_{\nu}^{n,\alpha}$ is sequentially compact.
\end{proof}

Based on Lemma~\ref{le:compactness of a set of probability measures}, we derive \emph{various representations of the Wasserstein distance between a probability measure in $\mathcal{P}_{p}(E)$ and a convex combination of Dirac measures}. Here, let $B(E,\Re_{+}^{n})$ denote the set of all $\Re_{+}^{n}$-valued Borel measurable maps on $E$.

\begin{Proposition}\label{pr:representations of the Wasserstein metric}
Let $d$ be complete and $\mu\in\mathcal{P}_{p}(E)$ be of the form $\mu = \sum_{i=1}^{n}\alpha_{i}\delta_{x_{i}}$ for some $n\in\N$, pairwise distinct $x_{1},\dots,x_{n}\in E$ and $\alpha\in ]0,1]^{n}$ with $\sum_{i=1}^{n}\alpha_{i} = 1$. Then
\begin{align*}
\mathcal{W}_{p}(\mu,\nu)^{p} &= \min_{\substack{\nu_{1},\dots,\nu_{n}\in\mathcal{P}_{p}(E):\\ \sum_{i=1}^{n}\alpha_{i}\nu_{i} = \nu}} \sum_{i=1}^{n}\alpha_{i}\int_{E}d(x,x_{i})^{p}\,\nu_{i}(\mathrm{d}x)\\
&= \min_{\substack{f\in B(E,\Re_{+}^{n}):\\ \sum_{i=1}^{n}\alpha_{i}f_{i} = 1, \int_{E}f_{1}\,\mathrm{d}\nu = \cdots = \int_{E}f_{n}\mathrm{d}\nu = 1}} \int_{E}\sum_{i=1}^{n}\alpha_{i}d(x,x_{i})^{p}f_{i}(x)\,\nu(\mathrm{d}x)
\end{align*}
for any $\nu\in\mathcal{P}_{p}(E)$. If in addition there are $m\in\N$, pairwise distinct $y_{1},\dots,y_{m}\in E$ and $\beta\in ]0,1]^{m}$ such that $\sum_{j=1}^{m}\beta_{j} = 1$ and $\nu = \sum_{j=1}^{m}\beta_{j}\delta_{y_{j}}$, then
\begin{equation*}
\mathcal{W}_{p}(\mu,\nu)^{p} = \min_{\substack{A\in\Re_{+}^{n\times m}:\\ A^{\top}\alpha =\mathbf{1}_{m},\, A\beta = \mathbf{1}_{n}}}\sum_{i=1}^{n}\sum_{j=1}^{m} \alpha_{i}A_{i,j}\beta_{j}d(x_{i},y_{j}),
\end{equation*}
where $\mathbf{1}_{i}\in\Re^{i}$ is given by $\mathbf{1}_{i}:= (1,\dots,1)^{\top}$ for $i\in\{m,n\}$.
\end{Proposition}

\begin{proof}
By Lemma~\ref{le:compactness of a set of probability measures}, the first identity follows once we have shown that $\mathcal{P}(\mu,\nu)$ does not only contain but also consists of all $\theta\in\mathcal{P}_{0}(E\times E)$ of the form $\theta = \sum_{i=1}^{n}\alpha_{i}\nu_{i}\otimes\delta_{x_{i}}$ for some $\nu_{1},\dots,\nu_{n}\in\mathcal{P}_{p}(E)$ with $\sum_{i=1}^{n}\alpha_{i}\nu_{i} = \nu$.

However, for any $\theta\in\mathcal{P}(\mu,\nu)$ and $B,C\in\mathcal{B}(E)$ we have $\theta(B\times C) = 0$ if $x_{1},\dots,x_{n}\in C^{c}$. Otherwise, we may let $k$ be the largest number in $\{1,\dots,n\}$ for which there are pairwise distinct $i_{1},\dots,i_{k}\in\{1,\dots,n\}$ such that $x_{i_{1}},\dots,x_{i_{k}}\in C$, in which case
\begin{equation*}
\theta(B\times C) = \theta(B\times C\cap\{x_{1},\dots,x_{n}\}) = \theta(B\times\{x_{i_{1}},\dots,x_{i_{k}}\}) = \sum_{j=1}^{k}\alpha_{i_{j}}\nu_{i_{j}}(B),
\end{equation*}
where $\nu_{i}\in\mathcal{P}_{p}(E)$ is defined by $\nu_{i}(B) := \alpha_{i}^{-1}\theta(B\times\{x_{i}\})$ for every $i\in\{1,\dots,n\}$. Thus, we have $\theta(B\times C) = \sum_{i=1}^{n}\alpha_{i}\nu_{i}(B)\delta_{x_{i}}(C)$ in either case, as required.

For the second identity it suffices to show that the set of all $(\nu_{1},\dots,\nu_{n})\in\prod_{i=1}^{n}\mathcal{P}_{p}(E)$ satisfying $\sum_{i=1}^{n}\alpha_{i}\nu_{i} = \nu$ does not only contain but also consists of all $n$-tuples $(\nu_{1},\dots,\nu_{n})$ of $\nu_{1},\dots,\nu_{n}\in\mathcal{P}_{0}(E)$ that admit the representations
\begin{equation}\label{eq:representation involving a density function}
\nu_{i}(B) = \int_{B}f_{i}(x)\,\nu(\mathrm{d}x)\quad\text{for all $i\in\{1,\dots,n\}$ and $B\in\mathcal{B}(E)$}
\end{equation}
and some Borel measurable map $f:E\rightarrow\Re_{+}^{m}$ satisfying $\sum_{i=1}^{n}\alpha_{i}f_{i} = 1$.  But for any $\nu_{1},\dots,\nu_{n}\in\mathcal{P}_{0}(E)$ that satisfy $\sum_{i=1}^{n}\alpha_{i}\nu_{i} = \nu$, the Radon-Nikod\'ym Theorem yields a Borel measurable map $g:E\rightarrow [0,\infty]^{n}$ that is unique, up to a $\nu$-null set, such that~\eqref{eq:representation involving a density function} holds when $f_{i}$ is replaced by $g_{i}$.

Since $\nu(B) = \sum_{i=1}^{n}\alpha_{i}\nu_{i}(B) = \int_{B}\sum_{i=1}^{n}\alpha_{i}g_{i}(x)\,\nu(\mathrm{d}x)$ for all $B\in\mathcal{B}(E)$, it holds that $1 = \sum_{i=1}^{n}\alpha_{i}g_{i}$ $\nu$-a.s. So, the map $f := g\mathbbm{1}_{\{\sum_{i=1}^{n}\alpha_{i}g_{i} = 1\}} + \mathbf{1}_{n}\mathbbm{1}_{\{\sum_{i=1}^{n}\alpha_{i}g_{i} = 1\}^{c}}$ is $\Re_{+}^{n}$-valued, Borel measurable and satisfies~\eqref{eq:representation involving a density function} and $\sum_{i=1}^{n}\alpha_{i}f_{i} = 1$.

For the second assertion we readily note that for every map $f\in B(E,\Re_{+}^{n})$ there is a unique matrix $A\in\Re_{+}^{n\times m}$ such that $f_{i}(y_{j}) = A_{i,j}$ for all $i\in\{1,\dots,n\}$ and $j\in\{1,\dots,m\}$, in which case $\int_{E}\sum_{i=1}^{n}\alpha_{i}d(x,x_{i})^{p}f_{i}(x)\nu(\mathrm{d}x) = \sum_{i=1}^{n}\sum_{j=1}^{m}\alpha_{i}A_{i,j}\beta_{j}d(x_{i},y_{j})^{p}$.

Hence, $\sum_{i=1}^{n}\alpha_{i}f_{i}(y_{j}) = 1$ is equivalent to $\sum_{i=1}^{n}\alpha_{i}A_{i,j} = 1$ for all $j\in\{1,\dots,m\}$ and we have $\int_{E}f_{i}\,\mathrm{d}\nu = 1$ $\Leftrightarrow$ $\sum_{j=1}^{m}A_{i,j}\beta_{j} = 1$ for any $i\in\{1,\dots,n\}$, which concludes the proof.
\end{proof}

We recall that an $E$-valued random vector $Y$ is $p$-fold integrable if $\E[d(Y,x_{0})^{p}] < \infty$ for some $x_{0}\in E$, which holds if and only if $\mathcal{L}(Y)\in\mathcal{P}_{p}(E)$. Then  Corollary~\ref{co:approximation of product measurable processes} and Proposition~\ref{pr:representations of the Wasserstein metric} entail a general measurability result.

\begin{Proposition}\label{pr:Borel measurability of the law map}
Let $E$ be a star convex set in a linear space with center $x_{0}\in E$ and $d$ be complete such that the convexity condition~\eqref{eq:convexity condition on the metric} holds. Then for any product measurable $p$-fold integrable process $X:I\times\Omega\rightarrow E$ the distribution map 
\begin{equation*}
I\rightarrow\mathcal{P}_{p}(E),\quad t\mapsto\mathcal{L}(X_{t})
\end{equation*}
is Borel measurable.
\end{Proposition}

\begin{proof}
As $\mathcal{P}_{p}(E)$ equipped with $\mathcal{W}_{p}$ is a Polish space, for any countable dense set $\mathcal{C}$ in $\mathcal{P}_{p}(E)$ the system of all open balls $\{\mu\in\mathcal{P}_{p}(E)\,|\,\mathcal{W}_{p}(\mu,\nu) < \varepsilon\}$ with center $\nu\in\mathcal{C}$ and rational radius $\varepsilon > 0$ is a countable basis for the topology of $\mathcal{P}_{p}(E)$.

If $D$ is a countable dense set in $E$, then for $\mathcal{C}$ we may take the set of all convex combinations of Dirac measures of the form $\sum_{i=1}^{n}\alpha_{i}\delta_{x_{i}}$ with $n\in\N$, $x_{1},\dots,x_{n}\in D$ and $\alpha\in [0,1]^{n}\cap\mathbb{Q}^{n}$ such that $\sum_{i=1}^{n}\alpha_{i} = 1$. See~\cite[Theorem~14.12]{AliBor99}, for example. Therefore, it suffices to show the Borel measurability of the function
\begin{equation}\label{eq:auxiliary Borel measurable function}
I\rightarrow\Re_{+},\quad t\mapsto\mathcal{W}_{p}(\mathcal{L}(X_{t}),\nu)
\end{equation}
when $\nu\in\mathcal{P}_{p}(E)$ is of the form $\nu = \sum_{j=1}^{m}\beta_{j}\delta_{y_{j}}$ for some $m\in\N$, pairwise distinct $y_{1},\dots,y_{m}\in E$ and $\beta\in ]0,1]^{m}$ with $\sum_{j=1}^{m}\beta_{j} = 1$. To this end, we apply Corollary~\ref{co:approximation of product measurable processes} when $x_{0}\in E$ is fixed.

Then the sequence $(X^{(n)})_{n\in\N}$ of $E$-valued product measurable processes on $I\times\Omega$ that appears there satisfies $\mathcal{L}(X_{t}^{(n)}) = \sum_{i=1}^{N_{n}} \P(X_{t}\in B_{i,n}) \delta_{x_{i,n}}$ for all $n\in\N$ and $t\in I$. For any $n\in\N$ let $\Delta_{n}$ be the Borel set of all $(x_{1},\dots,x_{n})\in E^{n}$ such that $x_{1},\dots,x_{n}$ are pairwise distinct. Then $f_{n}:\{\alpha\in [0,1]^{n}\,|\,\sum_{i=1}^{n}\alpha_{i} = 1\}\times\Delta_{n}\rightarrow\Re_{+}$ given by
\begin{equation*}
f_{n}(\alpha,x) := \mathcal{W}_{p}(\alpha_{1}\delta_{x_{1}} + \cdots + \alpha_{n}\delta_{x_{n}},\nu)
\end{equation*}
is Borel measurable, by the third representation of $\mathcal{W}_{p}(\cdot,\nu)$ in Proposition~\ref{pr:representations of the Wasserstein metric} and linear programming. Furthermore, from Fubini's theorem we immediately infer that the map $g_{n}:I\rightarrow\{\alpha\in [0,1]^{N_{n}}\,|\,\sum_{i=1}^{N_{n}}\alpha_{i} = 1\}\times \Delta_{N_{n}}$ defined by
\begin{equation*}
g_{n}(t) := \big(\P(X_{t}\in B_{1,n}),\dots,\P(X_{t}\in B_{N_{n},n}), x_{1,n},\dots,x_{N_{n},n}\big)
\end{equation*} 
is Borel measurable. Consequently, as composition of $f_{N_{n}}$ and $g_{n}$ the function $I\rightarrow\Re_{+}$, $t\mapsto \mathcal{W}_{p}(\mathcal{L}(X_{t}^{(n)}),\nu)$ is also Borel measurable and the triangle inequality yields that
\begin{equation*}
|\mathcal{W}_{p}(\mathcal{L}(X_{t}^{(n)}),\nu) - \mathcal{W}_{p}(\mathcal{L}(X_{t}),\nu)| \leq \mathcal{W}_{p}\big(\mathcal{L}(X_{t}^{(n)}),\mathcal{L}(X_{t})\big) \leq \E\big[d(X_{t}^{(n)},X_{t})^{p}\big]^{\frac{1}{p}}
\end{equation*}
for any $n\in\N$ and $t\in I$. Since $\lim_{n\uparrow\infty} \E\big[d(X_{t}^{(n)},X_{t})^{p}\big] = 0$ follows from dominated convergence, the function~\eqref{eq:auxiliary Borel measurable function} is Borel measurable as pointwise limit of Borel measurable functions.
\end{proof}

\subsection{Sharp moment and integral inequalities}\label{se:3.3}

As in the introduction, let $I$ be a non-degenerate interval in $\Re_{+}$ with $0\in I$, and we take a $\sigma$-finite Borel measure $\mu$ on $I$ and a set $J$ in $I$ of full $\mu$-measure. For $p\geq 1$, $N\in\N$ and $\beta_{1},\dots,\beta_{N}\in ]0,p]$ we set $\beta := \max_{j=1,\dots,N} \beta_{j}$.

We let $Y$ be an $[0,\infty]$-valued product measurable process and for non-negative kernels $k_{1},\dots,k_{N}$ on $I$ we define a measurable function $v_{j}:I\rightarrow [0,\infty]$ by
\begin{equation*}
v_{j}(t) := \bigg(\int_{[0,t]}k_{j}(t,s)^{\beta_{j}}\,\mu(\mathrm{d}s)\bigg)^{\frac{1}{\beta_{j}}}\quad\text{for any $j\in\{1,\dots,N\}$}
\end{equation*}
and let $v:I\rightarrow [0,\infty]$ be given by $v(t) := \E[Y_{t}^{p}]^{1/p} + v_{1}(t) + \cdots + v_{N}(t)$. Further, for non-negative kernels $l_{1},\dots,l_{N}$ on $I$ we define a non-negative kernel $\hat{l}_{j}$ on $I$ by
\begin{equation*}
\hat{l}_{j}(t,s) := \min\bigg\{l_{j}(t,s)^{\beta_{j}}\bigg(\int_{[0,t]}l_{j}(t,\tilde{s})^{\beta_{j}}\,\mu(\mathrm{d}\tilde{s})\bigg)^{\frac{\beta}{\beta_{j}} - 1},\bigg(\int_{[0,t]}l_{j}(t,\tilde{s})^{\frac{\beta_{j}\beta}{\beta - \beta_{j}}}\,\mu(\mathrm{d}\tilde{s})\bigg)^{\frac{\beta}{\beta_{j}} - 1}\bigg\}^{\frac{1}{\beta}},
\end{equation*}
if $\beta_{j} < \beta$, and $\hat{l}_{j}(t,s) := l_{j}(t,s)$, if $\beta_{j} = \beta$, for each $j\in\{1,\dots,N\}$. This yields the non-negative kernel on $I$ defined via
\begin{equation*}
l := N\max_{j=1,\dots,N}\hat{l}_{j}.
\end{equation*}
Finally, suppose that $\beta_{j}\geq 1$ for any $j\in\{1,\dots,N\}$ with $k_{j}\neq 0$ and $\beta\geq 1$. In this setting, we deduce a \emph{sharp moment inequality for sequences of processes} from~\cite{Kal24}.

\begin{Proposition}\label{pr:resolvent sequence inequality for processes}
Let $(X^{(n)})_{n\in\N_{0}}$ be a sequence of product measurable processes with values in $[0,\infty]$ for which there are sequences $(K^{(n,1)})_{n\in\N},\dots,(K^{(n,N)})_{n\in\N}$ of $[0,\infty]$-valued product measurable functions on $I\times I\times\Omega$ such that
\begin{equation}\label{eq:condition 1 for the resolvent sequence inequality}
\E\big[\big(X_{t}^{(n)}\big)^{p}\big]^{\frac{1}{p}} \leq \E\big[Y_{t}^{p}\big]^{\frac{1}{p}} + \sum_{j=1}^{N}\E\bigg[\bigg(\int_{[0,t]}\big(K_{t,s}^{(n,j)}\big)^{\beta_{j}}\,\mu(\mathrm{d}s)\bigg)^{\frac{p}{\beta_{j}}}\bigg]^{\frac{1}{p}}
\end{equation}
for any $n\in\N$ and $t\in J$ and
\begin{equation}\label{eq:condition 2 for the resolvent sequence inequality}
\E\big[\big(K_{t,s}^{(n,j)}\big)^{p}\big]^{\frac{1}{p}} \leq k_{j}(t,s) + l_{j}(t,s)\E\big[\big(X_{s}^{(n-1)}\big)^{p}\big]^{\frac{1}{p}}
\end{equation}
for all $n\in\N$, $j=1,\dots,N$ and $s,t\in J$ with $s\leq t$. Then
\begin{equation}\label{eq:resolvent sequence inequality for processes}
\begin{split}
\E\big[\big(X_{t}^{(n)}\big)^{p}\big]^{\frac{1}{p}} &\leq v(t) + \sum_{i=1}^{n-1}\bigg(\int_{[0,t]}\R_{l^{\beta},\mu,i}(t,s)v(s)^{\beta}\,\mu(\mathrm{d}s)\bigg)^{\frac{1}{\beta}}\\
&\quad + \bigg(\int_{[0,t]}\R_{l^{\beta},\mu,n}(t,s)\E\big[\big(X_{s}^{(0)}\big)^{p}\big]^{\frac{\beta}{p}}\,\mu(\mathrm{d}s)\bigg)^{\frac{1}{\beta}}
\end{split}
\end{equation}
for any $n\in\N$ and $t\in J$, and equality holds if $N = \beta = p = 1$ and the inequalities~\eqref{eq:condition 1 for the resolvent sequence inequality} and~\eqref{eq:condition 2 for the resolvent sequence inequality} are equations.
\end{Proposition}

\begin{proof}
We first apply Minkowski's integral inequality and then we use the inequalities of Jensen and Hölder to estimate that
\begin{align*}
\E\bigg[\bigg(\int_{[0,t]}\big(K_{t,s}^{(n,j)}\big)^{\beta_{j}}\,\mu(\mathrm{d}s)\bigg)^{\frac{p}{\beta_{j}}}\bigg]^{\frac{1}{p}} &\leq \bigg(\int_{[0,t]}\E\big[\big(K_{t,s}^{(n,j)}\big)^{p}\big]^{\frac{\beta_{j}}{p}}\,\mu(\mathrm{d}s)\bigg)^{\frac{1}{\beta_{j}}}\\\nonumber
&\leq v_{j}(t) + \bigg(\int_{[0,t]}\hat{l}_{j}(t,s)^{\beta}\E\big[\big(X_{s}^{(n-1)}\big)^{p}\big]^{\frac{\beta}{p}}\,\mu(\mathrm{d}s)\bigg)^{\frac{1}{\beta}}
\end{align*}
for any $n\in\N$, $j=1,\dots,N$ and $t\in J$, since $\beta_{j}\geq 1$ whenever $k_{j}\neq 0$. Hence, we obtain that
\begin{equation}\label{eq:implied condition 1 for the resolvent sequence inequality}
\E\big[\big(X_{t}^{(n)}\big)^{p}\big]^{\frac{1}{p}} \leq v(t) +  \bigg(\int_{[0,t]}l(t,s)^{\beta}\E\big[\big(X_{s}^{(n-1)}\big)^{p}\big]^{\frac{\beta}{p}}\,\mu(\mathrm{d}s)\bigg)^{\frac{1}{\beta}}
\end{equation}
for all $n\in\N$ and $t\in J$. Thus, the asserted inequality and statement follow directly from the $L^{\beta}$-resolvent sequence inequality in~\cite[Proposition~1.4]{Kal24}.
\end{proof}

\begin{Remark}\label{re:specific resolvent sequence inequality}
If for given $n\in\N$, $j=1,\dots,N$ and $t\in J$ we have $\beta_{j} = p$, $k_{j} = 0$ and $\E[(K_{t,s}^{(n,j)})^{p}] = l_{j}(t,s)^{p}\E[(X_{s}^{(n-1)})^{p}]$ for all $s\in J$ with $s\leq t$, then Fubini's theorem gives
\begin{equation*}
\E\bigg[\bigg(\int_{[0,t]}\big(K_{t,s}^{(n,j)}\big)^{\beta_{j}}\,\mu(\mathrm{d}s)\bigg)^{\frac{p}{\beta_{j}}}\bigg]^{\frac{1}{p}} = \bigg(\int_{[0,t]}l_{j}(t,s)^{\beta_{j}}\E\big[\big(X_{s}^{(n-1)}\big)^{p}\big]^{\frac{\beta_{j}}{p}}\,\mu(\mathrm{d}s)\bigg)^{\frac{1}{\beta_{j}}}.
\end{equation*}
\end{Remark}

Under the hypotheses of Proposition~\ref{pr:resolvent sequence inequality for processes}, we observe that if the initial process $X^{(0)}$ satisfies the condition
\begin{equation}\label{eq:condition 2 for the resolvent inequality}
\lim_{n\uparrow\infty} \int_{[0,t]}\R_{l^{\beta},\mu,n}(t,s)\E\big[\big(X_{s}^{(0)}\big)^{p}\big]^{\frac{\beta}{p}}\,\mu(\mathrm{d}s) = 0
\end{equation}
for any given $t\in J$, then it immediately follows from~\eqref{eq:resolvent sequence inequality for processes} that
\begin{equation*}
\limsup_{n\uparrow\infty} \E\big[\big(X_{t}^{(n)}\big)^{p}\big]^{\frac{1}{p}} \leq v(t) + \sum_{n=1}^{\infty}\bigg(\int_{[0,t]}\R_{l^{\beta},\mu,n}(t,s)v(s)^{\beta}\,\mu(\mathrm{d}s)\bigg)^{\frac{1}{\beta}}.
\end{equation*}
In particular, for $\beta = 1$ the series is of the form $\int_{[0,t]}\R_{l,\mu}(t,s)v(s)\,\mu(\mathrm{d}s)$, by monotone convergence. This entails the following type of \emph{moment resolvent inequality}.

\begin{Corollary}\label{co:resolvent inequality for processes}
Let $X$ and $X^{(0)}$ be product measurable processes and $K^{(1)},\dots,K^{(N)}$ be product measurable functions on $I\times I\times\Omega$, each taking all its values in $[0,\infty]$, such that $\E[(X_{t}^{(0)})^{p}] \leq \E[X_{t}^{p}]$,
\begin{equation}\label{eq:condition 1 for the resolvent inequality}
\E\big[X_{t}^{p}\big]^{\frac{1}{p}} \leq \E\big[Y_{t}^{p}\big]^{\frac{1}{p}} + \sum_{j=1}^{N}\E\bigg[\bigg(\int_{[0,t]}\big(K_{t,s}^{(j)}\big)^{\beta_{j}}\,\mu(\mathrm{d}s)\bigg)^{\frac{p}{\beta_{j}}}\bigg]^{\frac{1}{p}} 
\end{equation}
and~\eqref{eq:condition 2 for the resolvent inequality} holds for any $t\in J$. Further, assume that
\begin{equation}\label{eq:condition 3 for the resolvent inequality}
\E\big[\big(K_{t,s}^{(j)}\big)^{p}\big]^{\frac{1}{p}} \leq k_{j}(t,s) + l_{j}(t,s)\E\big[\big(X_{s}^{(0)}\big)^{p}\big]^{\frac{1}{p}}
\end{equation}
for all $j=1,\dots,N$ and $s,t\in J$ with $s\leq t$. Then 
\begin{align*}
\E\big[X_{t}^{p}\big]^{\frac{1}{p}} & \leq v(t) + \sum_{n=1}^{\infty}\bigg(\int_{[0,t]}\R_{l^{\beta},\mu,n}(t,s)v(s)^{\beta}\,\mu(\mathrm{d}s)\bigg)^{\frac{1}{\beta}}
\end{align*}
for any $t\in J$, and equality holds if $N = \beta = p = 1$ and the estimates~\eqref{eq:condition 1 for the resolvent inequality} and~\eqref{eq:condition 3 for the resolvent inequality} are identities.
\end{Corollary}

\begin{proof}
By Minkowski's integral inequality and the inequalities of Jensen and Hölder, the estimate~\eqref{eq:implied condition 1 for the resolvent sequence inequality} holds in the case that $X^{(n)} = X$ for each $n\in\N$. Hence, the assertions are consequences of the $L^{\beta}$-resolvent inequality in~\cite[Corollary~1.5]{Kal24}.

Alternatively, we may apply Proposition~\ref{pr:resolvent sequence inequality for processes} when $X^{(n)} = X$ for all $n\in\N$ and note that the sequence $(\sum_{i=1}^{n-1}(\int_{[0,t]}\R_{l^{\beta},\mu,i}(t,s)v(s)^{\beta}\,\mu(\mathrm{d}s))^{1/\beta})_{n\in\N}$ converges to its supremum for all $t\in I$, which is the same procedure as in the proof of Corollary~1.5 in~\cite{Kal24}.
\end{proof}

For the subsequent \emph{sharp integral estimates} we let $N = 1$ and set $\beta := \beta_{1}$, $k := k_{1}$ and $l : = l_{1}$. In addition, for each $n\in\N$ we define a non-negative kernel $l_{\mu,n,\beta,p}$ on $I$ by
\begin{equation}\label{eq:transformed iterated kernel}
l_{\mu,n,\beta,p}(t,s) :=  \int_{[s,t]}\bigg(\int_{[0,\tilde{s}]}\R_{l^{\beta},\mu,n}(\tilde{s},r)\,\mu(\mathrm{d}r)\bigg)^{\frac{p}{\beta} - 1}\R_{l^{\beta},\mu,n}(\tilde{s},s)\,\mu(\mathrm{d}\tilde{s}),
\end{equation}
if $\beta < p$, and $l_{\mu,n,\beta,p}(t,s) :=  \int_{[s,t]}\R_{l^{\beta},\mu,n}(\tilde{s},s)\,\mu(\mathrm{d}\tilde{s})$, if $\beta = p$.

\begin{Corollary}\label{co:integral resolvent sequence inequality for processes}
Suppose that $(X^{(n)})_{n\in\N_{0}}$ is a sequence of product measurable processes with values in $[0,\infty]$ such that
\begin{equation}\label{eq:condition for the integral resolvent sequence inequality}
\E\big[\big(X_{t}^{(n)}\big)^{p}\big]^{\frac{1}{p}} \leq v(t) + \bigg(\int_{[0,t]}l(t,s)^{\beta}\E\big[\big(X_{s}^{(n-1)}\big)^{p}\big]^{\frac{\beta}{p}}\,\mu(\mathrm{d}s)\bigg)^{\frac{1}{\beta}}
\end{equation}
for all $n\in\N$ and $t\in J$. Then
\begin{align*}
\bigg(\int_{[0,t]}\E\big[\big(X_{s}^{(n)}\big)^{p}\big]\,\mu(\mathrm{d}s)\bigg)^{\frac{1}{p}} &\leq \bigg(\int_{[0,t]}v(s)^{p}\,\mu(\mathrm{d}s)\bigg)^{\frac{1}{p}} + \sum_{i=1}^{n-1}\bigg(\int_{[0,t]}l_{\mu,i,\beta,p}(t,s)v(s)^{p}\,\mu(\mathrm{d}s)\bigg)^{\frac{1}{p}}\\
&\quad + \bigg(\int_{[0,t]}l_{\mu,n,\beta,p}(t,s)\E\big[\big(X_{s}^{(0)}\big)^{p}\big]\,\mu(\mathrm{d}s)\bigg)^{\frac{1}{p}}
\end{align*}
for any $n\in\N$ and $t\in J$, and equality holds if $\beta = p = 1$ and the inequality~\eqref{eq:condition for the integral resolvent sequence inequality} is an equation.
\end{Corollary}

\begin{proof}
Proposition~\ref{pr:resolvent sequence inequality for processes} and Minkowski's inequality yield that
\begin{align*}
\bigg(\int_{[0,t]}\E\big[\big(X_{s}^{(n)}\big)^{p}\big]\,\mu(\mathrm{d}s)\bigg)^{\frac{1}{p}} &\leq \bigg(\int_{[0,t]}v(s)^{p}\,\mu(\mathrm{d}s)\bigg)^{\frac{1}{p}}\\
&\quad + \sum_{i=1}^{n-1}\bigg(\int_{[0,t]}\bigg(\int_{[0,s]}\R_{l^{\beta},\mu,i}(s,r)v(r)^{\beta}\,\mu(\mathrm{d}r)\bigg)^{\frac{p}{\beta}}\,\mu(\mathrm{d}s)\bigg)^{\frac{1}{p}}\\
&\quad + \bigg(\int_{[0,t]}\bigg(\int_{0}^{s}\R_{l^{\beta},\mu,n}(s,r)\E\big[\big(X_{r}^{(0)}\big)^{p}\big]^{\frac{\beta}{p}}\,\mu(\mathrm{d}r)\bigg)^{\frac{p}{\beta}}\,\mu(\mathrm{d}s)\bigg)^{\frac{1}{p}}
\end{align*}
for all $n\in\N$ and $t\in J$, and equality holds if $\beta = p = 1$ and the estimate~\eqref{eq:condition for the integral resolvent sequence inequality} is an identity. Further, each measurable function $u:I\rightarrow [0,\infty]$ satisfies
\begin{equation}\label{eq:an application of the Jensen inequality}
\int_{[0,t]}\bigg(\int_{[0,s]}\R_{l^{\beta},\mu,i}(s,r)u(r)^{\beta}\,\mu(\mathrm{d}r)\bigg)^{\frac{p}{\beta}}\,\mu(\mathrm{d}s) \leq \int_{[0,t]}l_{\mu,i,\beta,p}(t,s)u(s)^{p}\,\mu(\mathrm{d}s)
\end{equation}
for any $i\in\N$ and $t\in I$, by Jensen's inequality and Fubini's theorem. As for $\beta = p = 1$ Jensen's inequality does not need to be applied, the whole claim follows.
\end{proof}

\begin{Corollary}\label{co:integral resolvent inequality for processes}
Assume that $X$ and $X^{(0)}$ are $[0,\infty]$-valued product measurable processes satisfying $\E[(X_{t}^{(0)})^{p}] \leq \E[X_{t}^{p}]$,
\begin{equation}\label{eq:condition for the integral resolvent inequality}
\E\big[X_{t}^{p}\big]^{\frac{1}{p}} \leq v(t) + \bigg(\int_{[0,t]}l(t,s)^{\beta}\E\big[\big(X_{s}^{(0)}\big)^{p}\big]^{\frac{\beta}{p}}\,\mu(\mathrm{d}s)\bigg)^{\frac{1}{\beta}}
\end{equation}
and $\lim_{n\uparrow\infty}\int_{[0,t]}l_{\mu,n,\beta,p}(t,s)\E[(X_{s}^{(0)})^{p}]\,\mu(\mathrm{d}s) = 0$ for any $t\in J$. Then
\begin{align*}
\bigg(\int_{[0,t]}\E\big[X_{s}^{p}\big]\,\mu(\mathrm{d}s)\bigg)^{\frac{1}{p}} &\leq \bigg(\int_{[0,t]}v(s)^{p}\,\mu(\mathrm{d}s)\bigg)^{\frac{1}{p}} + \sum_{n=1}^{\infty}\bigg(\int_{[0,t]}l_{\mu,n,\beta,p}(t,s)v(s)^{p}\,\mu(\mathrm{d}s)\bigg)^{\frac{1}{p}}
\end{align*}
for every $t\in J$, and equality holds if $\beta = p = 1$ and the inequality symbol in~\eqref{eq:condition for the integral resolvent inequality} is replaced by an equal sign.
\end{Corollary}

\begin{proof}
We invoke Corollary~\ref{co:integral resolvent sequence inequality for processes} in the case that $X^{(n)} = X$ for each $n\in\N$ and utilise that the sequence $(\sum_{i=1}^{n-1}(\int_{[0,t]}l_{\mu,i,\beta,p}(t,s)v(s)^{p}\,\mu(\mathrm{d}s))^{1/p})_{n\in\N}$ converges to its supremum for any $t\in I$.
\end{proof}

\subsection{Sharp moment estimates for stochastic Volterra processes}\label{se:3.4}

In this section, we suppose that $\B$ and $\Sigma$ are two $\mathcal{B}(I)\otimes\mathcal{A}$-measurable maps on $I\times I\times\Omega$ with respective values in $E$ and $\mathcal{L}_{2}(\ell^{2},E)$. Then $N_{t}:=\{\int_{0}^{t}|\B_{t,s}| + |\Sigma_{t,s}|_{2}^{2}\,\mathrm{d}s = \infty\}$ is an event in $\mathcal{F}_{t}$ for any $t\in I$, due to Proposition~\ref{pr:progressively measurable integral processes} and Corollary~\ref{co:approximation of product measurable processes}.

If $N_{t}$ is null for a.e.~$t\in I$, then by a \emph{stochastic Volterra process} with coefficients $\xi$, $\B$ and $\Sigma$ we shall mean an $E$-valued $\mathbb{F}$-progressively measurable process $X$ such that
\begin{equation}\label{eq:Volterra process}
X_{t} = \xi_{t} + \int_{0}^{t}\B_{t,s}\,\mathrm{d}s + \int_{0}^{t}\Sigma_{t,s}\,\mathrm{d}W_{s}\quad\text{a.s.}
\end{equation}
for a.e.~$t\in I$. In this case, the set $I_{X}$ of all $t\in I$ satisfying $\P(N_{t}) = 0$ and~\eqref{eq:Volterra process} is Borel and has full measure. Further, from Proposition~\ref{pr:resolvent sequence inequality for processes} and Corollary~\ref{co:resolvent inequality for processes} we obtain \emph{$L^{p}$-estimates}.

\begin{Proposition}\label{pr:moment estimates for stochastic Volterra processes}
Let $X$ be an $E$-valued product measurable process and $k_{1},k_{2},l_{1},l_{2}$ be non-negative kernels on $I$ such that
\begin{equation}\label{eq:abstract affine growth condition}
\E\big[|\B_{t,s}|^{p}\big]^{\frac{1}{p}}\mathbbm{1}_{\{1\}}(i) + \E\big[|\Sigma_{t,s}|_{2}^{p}\big]^{\frac{1}{p}}\mathbbm{1}_{\{2\}}(i) \leq k_{i}(t,s) + l_{i}(t,s)\E\big[|X_{s}|^{p}\big]^{\frac{1}{p}}
\end{equation}
for all $i\in\{1,2\}$ and $s,t\in I$ with $s < t$, and define $k_{0}:I\rightarrow [0,\infty]$ and a non-negative kernel $l$ on $I$ by~\eqref{eq:special integral function} and~\eqref{eq:special transformed kernel}. Then the following two assertions hold:
\begin{enumerate}[(i)]
\item For any $t\in I$ such that $k_{0}(t) < \infty$ and $l(t,\cdot)\E[|X|^{p}]^{1/p}$ is square-integrable, the event $N_{t}$ is null and 
\begin{align*}
\E\bigg[\bigg|\int_{0}^{t}\B_{t,s}\,\mathrm{d}s + \int_{0}^{t}\Sigma_{t,s}\,\mathrm{d}W_{s}\bigg|^{p}\bigg]^{\frac{1}{p}} & \leq \E\bigg[\bigg(\int_{0}^{t}|\B_{t,s}|\,\mathrm{d}s\bigg)^{p}\bigg]^{\frac{1}{p}} + w_{p}\E\bigg[\bigg(\int_{0}^{t}|\Sigma_{t,s}|_{2}^{2}\,\mathrm{d}s\bigg)^{\frac{p}{2}}\bigg]^{\frac{1}{p}}\\
&\leq k_{0}(t) + \bigg(\int_{0}^{t}l(t,s)^{2}\E\big[|X_{s}|^{p}\big]^{\frac{2}{p}}\,\mathrm{d}s\bigg)^{\frac{1}{2}}.
\end{align*}

\item If $N_{t}$ is null for a.e.~$t\in I$ and $X$ is a Volterra process with coefficients $\xi$, $\B$ and $\Sigma$, then~\eqref{eq:growth estimate for solutions 1} holds for a.e.~$t\in I$ as soon as
\begin{equation}\label{eq:condition for a moment estimate for stochastic Volterra processes}
\lim_{n\uparrow\infty}\int_{0}^{t}\R_{l^{2},n}(t,s)\E\big[|X_{s} - \xi_{s}|^{p}\big]^{\frac{2}{p}}\,\mathrm{d}s = 0\quad\text{for a.e.~$t\in I$.}
\end{equation}
\end{enumerate}
\end{Proposition}

\begin{proof}
(i) The second estimate follows from Proposition~\ref{pr:resolvent sequence inequality for processes} for the choice $N = 2$, $\beta_{1} = 1$ and $\beta_{2} = 2$. Namely, there we may take $Y = 0$, $X^{(0)} = |X|$ and $X_{t}^{(1)} = \E[(\int_{0}^{t}|\B_{t,s}|\,\mathrm{d}s)^{p}]^{1/p}$ $ +\, w_{p}\E[(\int_{0}^{t}|\Sigma_{t,s}|_{2}^{2}\,\mathrm{d}s)^{p/2}]^{1/p}$ for any $t\in I$.

Hence, the $p$th moment of the $[0,\infty]$-valued random variable $\int_{0}^{t}|\B_{t,s}| + |\Sigma_{t,s}|_{2}^{2}\,\mathrm{d}s$ is finite, which ensures that $\P(N_{t}) = 0$. Now the first estimate is implied by the moment estimate~\eqref{eq:stochastic integral estimate}.

(ii) From~(i) and Minkowski's inequality we directly infer that the measurable function $k_{0,l,\xi}:I\rightarrow [0,\infty]$ given by $k_{0,l,\xi}(t) := k_{0}(t) + (\int_{0}^{t}l(t,s)^{2}\E[|\xi_{s}|^{p}]^{2/p}\,\mathrm{d}s)^{1/2}$ satisfies
\begin{equation*}
\E\big[|X_{t} - \xi_{t}|^{p}\big]^{\frac{1}{p}} \leq k_{0,l,\xi}(t) + \bigg(\int_{0}^{t}l(t,s)^{2}\E\big[|X_{s} - \xi_{s}|^{p}\big]^{\frac{2}{p}}\,\mathrm{d}s\bigg)^{\frac{1}{2}}
\end{equation*}
for any $t\in I_{X}$. So, an application of Corollary~\ref{co:resolvent inequality for processes} for the choice $N = 2$, $\beta_{1} = 1$ and $\beta_{2} = 2$ in combination with Minkowski's inequality yield the claim, since
\begin{equation}\label{eq:special property of the resolvent sequence}
\int_{0}^{t}\R_{l^{2},i}(t,s)\int_{0}^{s}l(s,r)^{2}\E[|\xi_{r}|^{p}]^{\frac{2}{p}}\,\mathrm{d}r\,\mathrm{d}s = \int_{0}^{t}\R_{l^{2},i+1}(t,s)\E[|\xi_{s}|^{p}]^{\frac{2}{p}}\,\mathrm{d}s
\end{equation}
for any $i\in\N$ and $t\in I$, by Fubini's theorem.
\end{proof}

Based on the preceding proposition, we can give an \emph{$L^{p}$-estimate for the increments} of a Volterra process.

\begin{Lemma}\label{le:regularity of stochastic Volterra processes}
Let $X$ be an $E$-valued product measurable process for which there are $[0,\infty]$-valued measurable functions $f_{1},f_{2},g_{1},g_{2}$ on $I\times I\times I$ such that
\begin{equation*}
\E\big[|\B_{t,r} - \B_{s,r}|^{p}\big]^{\frac{1}{p}}\mathbbm{1}_{\{1\}}(i) + \E\big[|\Sigma_{t,r} - \Sigma_{s,r}|_{2}^{p}\big]^{\frac{1}{p}}\mathbbm{1}_{\{2\}}(i) \leq f_{i}(t,s,r) + g_{i}(t,s,r)\E\big[|X_{r}|^{p}\big]^{\frac{1}{p}}
\end{equation*}
for all $i\in\{1,2\}$ and $r,s,t\in I$ with $r < s < t$, and define the non-negative kernel $f$ on $I$ and $g:I\times I\times I\rightarrow [0,\infty]$ by $f(t,s) := \int_{0}^{s}f_{1}(t,s,r)\,\mathrm{d}r + w_{p}(\int_{0}^{s}f_{2}(t,s,r)^{2}\,\mathrm{d}r)^{1/2}$ and
\begin{equation*}
g(t,s,r) := 2\max\bigg\{\min\bigg\{g_{1}(t,s,r)\int_{0}^{s}g_{1}(t,s,\tilde{r})\,\mathrm{d}\tilde{r},\int_{0}^{s}g_{1}(t,s,\tilde{r})^{2}\,\mathrm{d}\tilde{r}\bigg\}^{\frac{1}{2}},w_{p}g_{2}(t,s,r)\bigg\}.
\end{equation*}
Then for all $s,t\in I$ with $s\leq t$ we have
\begin{align*}
\E\bigg[\bigg(\int_{0}^{s}|\B_{t,r} - \B_{s,r}|\,\mathrm{d}r\bigg)^{p}\bigg]^{\frac{1}{p}} &+ w_{p}\E\bigg[\bigg(\int_{0}^{s}|\Sigma_{t,r} - \Sigma_{s,r}|_{2}^{2}\,\mathrm{d}r\bigg)^{\frac{p}{2}}\bigg]^{\frac{1}{p}}\\
&\leq f(t,s) + \bigg(\int_{0}^{s}g(t,s,r)^{2}\E\big[|X_{r}|^{p}\big]^{\frac{2}{p}}\,\mathrm{d}r\bigg)^{\frac{1}{2}}.
\end{align*}
Further, if $k_{1},k_{2},l_{1},l_{2}$ are non-negative kernels on $I$ satisfying the estimate~\eqref{eq:abstract affine growth condition}, $N_{t}$ is null for a.e.~$t\in I$ and $X$ is a Volterra process with coefficients $\xi$, $\B$ and $\Sigma$, then
\begin{align*}
\E\big[|X_{s} - X_{t}|^{p}\big]^{\frac{1}{p}} &\leq \E\big[|\xi_{s} - \xi_{t}|^{p}\big]^{\frac{1}{p}} + \int_{s}^{t}k_{1}(t,\tilde{s})\,\mathrm{d}\tilde{s} + w_{p}\bigg(\int_{s}^{t}k_{2}(t,\tilde{s})^{2}\,\mathrm{d}\tilde{s}\bigg)^{\frac{1}{2}} + f(t,s)\\
&\quad + \bigg(\int_{s}^{t}l(t,\tilde{s})^{2}\E\big[|X_{\tilde{s}}|^{p}\big]^{\frac{2}{p}}\,\mathrm{d}\tilde{s}\bigg)^{\frac{1}{2}} + \bigg(\int_{0}^{s}g(t,s,r)^{2}\E\big[|X_{r}|^{p}\big]^{\frac{2}{p}}\,\mathrm{d}r\bigg)^{\frac{1}{2}}
\end{align*}
for any $s,t\in I_{X}$ with $s\leq t$ and the non-negative kernel $l$ on $I$ given by~\eqref{eq:special transformed kernel}.
\end{Lemma}

\begin{proof}
The first estimate is a special case of Proposition~\ref{pr:resolvent sequence inequality for processes} for $N = 1$, $\beta_{1} = 1$ and $\beta_{2} = 2$, because for any given $t\in I$ we may choose $Y_{s} = 0$, $X_{s}^{(0)} = |X_{s}|$ and $X_{s}^{(1)}$ $=\E[(\int_{0}^{s}|\B_{t,r} - \B_{s,r}|\,\mathrm{d}r)^{p}]^{1/p} + w_{p}\E[(\int_{0}^{s}|\Sigma_{t,r} - \Sigma_{s,r}|_{2}^{2}\,\mathrm{d}r)^{p/2}]^{1/p}$ for all $s\in [0,t]$.

Regarding the second claim, we note that $X_{t} - X_{s} = \xi_{t} - \xi_{s} + \int_{s}^{t}\B_{t,\tilde{s}}\,\mathrm{d}\tilde{s} + \int_{s}^{t}\Sigma_{t,\tilde{s}}\,\mathrm{d}W_{\tilde{s}}$ $ +\, \int_{0}^{s}\B_{t,r} - \B_{s,r}\,\mathrm{d}r + \int_{0}^{s}\Sigma_{t,r} - \Sigma_{s,r}\,\mathrm{d}W_{r}$ a.s. Thus, the first estimate and Proposition~\ref{pr:moment estimates for stochastic Volterra processes} imply the second estimate.
\end{proof}

\section{Supplementary analytic results}\label{se:4}

\subsection{Integral estimates for iterated kernels}\label{se:4.1}

We deduce \emph{two types of integral estimates for the iterated kernels of a non-negative kernel} $k$ on $I$ relative to a $\sigma$-finite Borel measure $\mu$ on $I$ that are recursively given by~\eqref{eq:resolvent sequence} when $I$ is only a non-degenerate interval in $\Re$.

\begin{Proposition}\label{pr:integral estimate for iterated kernels of first type}
Suppose that $c_{0} := \sup_{t\in I}\int_{I(t)}k(t,s)^{p}\,\mu(\mathrm{d}s)$ and
\begin{equation*}
\varepsilon_{0} := \lim_{\delta\downarrow 0}\sup_{\substack{r,t\in I:\\ r \leq t \leq r + \delta}}\int_{[r,t]}k(t,s)^{p}\,\mu(\mathrm{d}s)
\end{equation*}
are finite for some $p > 0$, where $I(t) := \{s\in I\,|\,s\leq t\}$ for all $t\in I$. Then for every $\varepsilon > \varepsilon_{0}$ there is $\delta > 0$ such that
\begin{equation}\label{eq:integral estimate for iterated kernels of first type}
\sup_{\substack{r,t\in I:\\
r \leq t \leq r + m\delta}}\int_{[r,t]}\R_{k^{p},\mu,n}(t,s)\,\mu(\mathrm{d}s) \leq (nc_{\varepsilon})^{m-1}\varepsilon^{n}
\end{equation}
for all $m,n\in\N$ with $c_{\varepsilon} := \max\{1,\frac{c_{0}}{\varepsilon}\}$. Moreover, the inequality remains valid if in the definitions of $c_{0}$ and $\varepsilon_{0}$ and in~\eqref{eq:integral estimate for iterated kernels of first type} the suprema are replaced by essential suprema.
\end{Proposition}

\begin{proof}
We may assume that $p = 1$, since along with $k$ the $p$th power of $k$ is a non-negative kernel on $I$ that satisfies the assumptions of the proposition in the case $p = 1$.

By hypothesis, there is $\delta > 0$ such that $\int_{[r,t]}k(t,s)\,\mu(\mathrm{d}s) \leq \varepsilon$ for all $r,t\in I$ with $r \leq t \leq r + \delta$. Thus, if~\eqref{eq:integral estimate for iterated kernels of first type} holds for $m = 1$ and some $n\in\N$, then Fubini's theorem yields that
\begin{equation}\label{eq:application of the Fubini theorem}
\int_{[r,t]}\R_{k,\mu,n + 1}(t,s)\,\mu(\mathrm{d}s) = \int_{[r,t]}k(t,\tilde{s})\int_{[r,\tilde{s}]}\R_{k,\mu,n}(\tilde{s},s)\,\mu(\mathrm{d}s)\,\mu(\mathrm{d}\tilde{s}) \leq \varepsilon^{n+1}
\end{equation}
for any $r,t\in I$ with $r\leq t\leq r + \delta$. So,~\eqref{eq:integral estimate for iterated kernels of first type} is valid in the case $m = 1$ for each $n\in\N$. Next, suppose that~\eqref{eq:integral estimate for iterated kernels of first type} holds for some $m\in\N$ and any $n\in\N$.

As $\int_{[r,t]}k(t,s)\,\mu(\mathrm{d}s)$ $\leq (nc_{\varepsilon})^{m}\varepsilon$ for any $r,t\in I$ with $r\leq t$, we may also assume that~\eqref{eq:integral estimate for iterated kernels of first type} holds for $m + 1$ instead of $m$ and some fixed $n\in\N$. Then the identity in~\eqref{eq:application of the Fubini theorem} gives
\begin{align*}
\int_{[r,t]}\R_{k,\mu,n+1}(t,s)\,\mu(\mathrm{d}s) &\leq c_{\varepsilon}\varepsilon\sup_{\tilde{s}\in [r,t_{0}]}\int_{[r,\tilde{s}]}\R_{k,\mu,n}(\tilde{s},s)\,\mu(\mathrm{d}s) + \varepsilon\sup_{\tilde{s}\in [r,t]}\int_{[r,\tilde{s}]}\R_{k,\mu,n}(\tilde{s},s)\,\mu(\mathrm{d}s)\\
&\leq (c_{\varepsilon}\varepsilon)(nc_{\varepsilon})^{m-1}\varepsilon^{n} + \varepsilon(nc_{\varepsilon})^{m}\varepsilon^{n} \leq \big((n+1)c_{\varepsilon}\big)^{m}\varepsilon^{n+1}
\end{align*}
for all $r,t\in I$ with $r\leq t \leq r + (m + 1)\delta$, where $t_{0}\in [r,t]$ satisfies $t_{0}\leq r + m\delta$ and $t\leq t_{0} + \delta$. This completes the nested induction proofs and the second claim follows from a short review.
\end{proof}

\begin{Remark}\label{re:integral estimate for iterated kernels of first type}
If $I$ is bounded, then monotone convergence implies that
\begin{equation*}
\sup_{t\in I}\int_{I(t)}\R_{k^{p},\mu,n}(t,s)\,\mu(\mathrm{d}s) \leq (nc_{\varepsilon})^{m-1}\varepsilon^{n}
\end{equation*}
for all $n\in\N$ with $m\in\N$ satisfying $\mathrm{diam}(I) \leq m\delta$. Moreover, if $\varepsilon < 1$, which is only possible in the case that $\varepsilon_{0} < 1$, then
\begin{equation}\label{eq:regularity of the sum of iterated kernels}
\sum_{n=1}^{\infty}\sup_{\substack{r,t\in I:\\
r \leq t \leq r + m\delta}}\bigg(\int_{[r,t]}\R_{k^{p},\mu,n}(t,s)\,\mu(\mathrm{d}s)\bigg)^{q} \leq c_{\varepsilon}^{q(m-1)}\sum_{n=1}^{\infty} n^{q(m-1)}\varepsilon^{q n}
\end{equation}
for any $m\in\N$ and $q > 0$ and the series on the right-hand side converges, by the ratio test.
\end{Remark}

We obtain a similar integral estimate by using the fact that
\begin{equation}\label{eq:property of the iterated kernels}
\R_{k,\mu,m+n}(t,s) = \int_{[s,t]}\R_{k,\mu,m}(t,\tilde{s})\R_{k,\mu,n}(\tilde{s},s)\,\mu(\mathrm{d}\tilde{s})
\end{equation}
for any $m,n\in\N$ and $s,t\in I$ with $s\leq t$, which follows inductively inductively from the recursive definition~\eqref{eq:resolvent sequence}.

\begin{Proposition}\label{pr:integral estimate for iterated kernels of second type}
Suppose that $c_{0} := \sup_{s\in I}\int_{I(s)'}k(t,s)^{p}\,\mu(\mathrm{d}t)$ and
\begin{equation*}
\varepsilon_{0} := \lim_{\delta\downarrow 0}\sup_{\substack{r,t\in I:\\ r\leq t\leq r + \delta}}\int_{[r,t]}k(s,r)^{p}\,\mu(\mathrm{d}s) 
\end{equation*}
are finite for some $p > 0$, where $I(s)' := \{t\in I\,|\,t\geq s\}$ for all $s\in I$. Then for each $\varepsilon > \varepsilon_{0}$ there is $\delta > 0$ such that
\begin{equation}\label{eq:integral estimate for iterated kernels of second type}
\sup_{\substack{r,t\in I:\\ r \leq t \leq r + m\delta}}\int_{[r,t]}\R_{k^{p},\mu,n}(s,r)\,\mu(\mathrm{d}s) \leq (nc_{\varepsilon})^{m-1}\varepsilon^{n}
\end{equation}
for any $m,n\in\N$ with $c_{\varepsilon} := \max\{1,\frac{c_{0}}{\varepsilon}\}$. Further, the estimate remains valid if in the definitions of $c_{0}$ and $\varepsilon_{0}$ and in~\eqref{eq:integral estimate for iterated kernels of second type} the suprema are replaced by essential suprema.
\end{Proposition}

\begin{proof}
By the same reasoning as in Proposition~\ref{pr:integral estimate for iterated kernels of first type}, we may assume that $p  = 1$. Further, we take $\delta > 0$ such that $\int_{[r,t]}k(s,r)\,\mu(\mathrm{d}s)\leq \varepsilon$ for any $r,t\in I$ with $r\leq t\leq r + \delta$ and suppose that~\eqref{eq:integral estimate for iterated kernels of second type} holds for $m = 1$ and some $n\in\N$. Then
\begin{align*}
\int_{[r,t]}\R_{k,\mu,n + 1}(s,r)\,\mu(\mathrm{d}s) &= \int_{[r,t]}\int_{[\tilde{r},t]}k(s,\tilde{r})\,\mu(\mathrm{d}s)\,\R_{k,\mu,n}(\tilde{r},r)\,\mu(\mathrm{d}\tilde{r}) \leq \varepsilon^{n+1}
\end{align*}
for any $r,t\in I$ with $r\leq t \leq r + \delta$, by Fubini's theorem. This shows~\eqref{eq:integral estimate for iterated kernels of second type} for $m = 1$ and any $n\in\N$ and we may assume that~\eqref{eq:integral estimate for iterated kernels of second type} holds for some $m\in\N$ and each $n\in\N$. 

In addition, we may suppose that~\eqref{eq:integral estimate for iterated kernels of second type} is valid for $m + 1$ instead of $m$ and some fixed $n\in\N$. Then from the identity~\eqref{eq:property of the iterated kernels} and Fubini's theorem we infer that
\begin{align*}
\int_{[r,t]}\R_{k,\mu,n+1}(s,r)\,\mu(\mathrm{d}s) &\leq \varepsilon\sup_{\tilde{r}\in [r,t_{0}]}\int_{[\tilde{r},t]}\R_{k,\mu,n}(s,\tilde{r})\,\mu(\mathrm{d}s) + c_{\varepsilon}\varepsilon\sup_{\tilde{r}\in [t_{0},t]}\int_{[\tilde{r},t]}\R_{k,\mu,n}(s,\tilde{r})\,\mu(\mathrm{d}s)\\
&\leq \varepsilon(nc_{\varepsilon})^{m}\varepsilon^{n} + (c_{\varepsilon}\varepsilon)(nc_{\varepsilon})^{m-1}\varepsilon^{n} \leq ((n+1)c_{\varepsilon})^{m}\varepsilon^{n+1}
\end{align*}
for any $r,t\in I$ with $r\leq t\leq r + (m + 1)\delta$ and $t_{0}\in [r,t]$ satisfying $t_{0}\leq r + \delta$ and $t\leq t_{0} + m\delta$. So, the nested induction proofs are complete, and we obtain the second assertion by reviewing what we have shown.
\end{proof}

\begin{Remark}\label{re:integral estimate for iterated kernels of second type}
Similarly as before, if $I$ is bounded, then from monotone convergence we obtain that
\begin{equation*}
\sup_{s\in I}\int_{I(s)'}\mathrm{R}_{k^{p},\mu,n}(t,s)^{p}\,\mu(\mathrm{d}t) \leq (nc_{\varepsilon})^{m-1}\varepsilon^{n}
\end{equation*}
for all $n\in\N$ with $m := \lceil \mathrm{diam}(I)/\delta\rceil$. Further, if $\varepsilon\wedge \varepsilon_{0} < 1$, then the finite term appearing on the right-hand side in~\eqref{eq:regularity of the sum of iterated kernels} bounds the series
\begin{equation*}
\sum_{n=1}^{\infty}\sup_{\substack{r,t\in I:\\ r \leq t \leq r + m\delta}}\bigg(\int_{[r,t]}\R_{k^{p},\mu,n}(s,r)\,\mu(\mathrm{d}s)\bigg)^{q}\quad\text{for all $m\in\N$ and $q > 0$.}
\end{equation*}
\end{Remark}

\subsection{A metrical decomposition principle}\label{se:4.2}

Let $\ell_{+}$ denote the convex cone of all sequences $a = (a_{i})_{i\in\N}$ in $\Re_{+}$, $e_{i}$ be the $i$th canonical basis vector in $\ell_{+}$ for each $i\in\N$, $\mathbf{0} := (0)_{i\in\N}$ and $S$ be a non-empty set. For a sequence $(d_{i})_{i\in\N}$ of $\Re_{+}$-valued functions on $S\times S$ we define a map $d_{0}:S\times S\rightarrow\ell_{+}$ by
\begin{equation*}
d_{0}(x,y) := (d_{i}(x,y))_{i\in\N}
\end{equation*}
and call a functional $F$ on $\ell_{+}$ \emph{increasing} if it satisfies $F(a) \leq F(b)$ for any $a,b\in\ell_{+}$ with $a_{i}\leq b_{i}$ for each $i\in\N$. Then a \emph{metric of functional type} can be constructed as follows.

\begin{Proposition}\label{pr:functional metric}
Let $F:\ell_{+}\rightarrow\Re_{+}$ be increasing and subadditive and vanish only at $\mathbf{0}$. Then a sequence $(d_{i})_{i\in\N}$ of pseudometrics on $S$ induces another pseudometric $d$ by
\begin{equation}\label{eq:functional metric}
d(x,y) := F(d_{0}(x,y)),
\end{equation}
which is a metric if and only if any $x,y\in S$ with $d_{0}(x,y) = 0$ agree, and the topology induced by $d$ is finer than that of $d_{i}$ for each $i\in\N$. Moreover, if
\begin{equation}\label{eq:functional metric condition 1}
\lim_{v\rightarrow 0}F\bigg(\sum_{i=1}^{j}v_{i}e_{i}\bigg) = 0\quad\text{for all $j\in\N$}\quad\text{and}\quad \lim_{j\uparrow\infty} \sup_{x,y\in S}F\bigg(\sum_{i=j}^{\infty}d_{i}(x,y)e_{i}\bigg) = 0,
\end{equation}
then the following four assertions hold:
\begin{enumerate}[(i)]
\item A sequence in $S$ converges relative to $d$ if and only if there is a point in $S$ to which it converges relative to $d_{i}$ for any $i\in\N$.

\item A sequence in $S$ is Cauchy relative to $d$ if and only if it is Cauchy with respect to $d_{i}$ for all $i\in\N$.

\item If for any sequence $(x_{n})_{n\in\N}$ in $S$ converging relative to $d_{i}$ for each $i\in\N$ there is $x\in S$ such that $\lim_{n\uparrow\infty} d_{i}(x_{n},x) = 0$ for all $i\in\N$, then the completeness of $d_{i}$ for each $i\in\N$ implies that of $d$.

\item If $d$ turns $S$ into a separable space, then so does $d_{i}$ for each $i\in\N$ and there is a countable set that is dense with respect to $d_{i}$ for any $i\in\N$. Conversely, if
\begin{equation}\label{eq:functional metric condition 3}
\text{the topology induced by $d_{i+1}$ is finer than that of $d_{i}$ for each $i\in\N$,}
\end{equation}
\end{enumerate}
then $S$ equipped with $d$ is separable as soon as it is separable if endowed with $d_{i}$ for any $i\in\N$. 
\end{Proposition}

\begin{proof}
The triangle inequality holds for $d$, as $F$ increasing and subadditive, and we have $d(x,y) = 0$ $\Leftrightarrow$ $d_{0}(x,y) = 0$ for any $x,y\in S$, since $F$ vanishes only at $\mathbf{0}$. This explains the first claim, and in what follows, we take a sequence $(x_{n})_{n\in\N}$ and a point $\hat{x}$ in $S$.

If $\lim_{n\uparrow\infty} d(x_{n},\hat{x}) = 0$, then $\lim_{n\uparrow\infty} F(d_{i}(x_{n},\hat{x})e_{i}) = 0$ for any fixed $i\in\N$ and there cannot exist $\varepsilon > 0$ such that $d_{i}(x_{n},\hat{x})\geq \varepsilon$ for infinitely many $n\in\N$, as $F(\varepsilon e_{i}) > 0$. So, $\lim_{n\uparrow\infty} d_{i}(x_{n},\hat{x}) = 0$ in this case, which shows the second claim.

(i) The only if-statement follows from what we have just shown. For the if-direction we suppose that $\lim_{n\uparrow\infty} d_{i}(x_{n},\hat{x}) = 0$ for every $i\in\N$ and note that
\begin{align}\label{eq:functional metric 1}
d(x,y) \leq F\bigg(\sum_{i=1}^{j}d_{i}(x,y)e_{i}\bigg) + F\bigg(\sum_{i=j+1}^{\infty} d_{i}(x,y)e_{i}\bigg)
\end{align}
for any $j\in\N$ and $x,y\in S$, by the subadditivity of $F$. In combination with the two conditions in~\eqref{eq:functional metric condition 1}, this implies that $(x_{n})_{n\in\N}$ converges to $\hat{x}$ with respect to $d$.

(ii) If $(x_{n})_{n\in\N}$ is Cauchy relative to $d$, then $\lim_{n\uparrow\infty} \sup_{m\in\N:\, m\geq n} F(d_{i}(x_{n},x_{m})e_{i}) = 0$ for each fixed $i\in\N$ and there is no $\varepsilon > 0$ satisfying $d_{i}(x_{n},x_{m}) > \varepsilon$ for infinitely many $m,n\in\N$. Thus, $(x_{n})_{n\in\N}$ must be Cauchy with respect to $d_{i}$ in this case.

Conversely, let $(x_{n})_{n\in\N}$ be Cauchy relative to $d_{i}$ for each $i\in\N$. Then for $\varepsilon > 0$ the second condition in~\eqref{eq:functional metric condition 1} gives $j_{0}\in\N$ such that
\begin{equation}\label{eq:functional metric 2}
F\bigg(\sum_{i=j+1}^{\infty}d_{i}(x,y)e_{i}\bigg) < \frac{\varepsilon}{2}\quad\text{for all $j\in\N$ with $j\geq j_{0}$}
\end{equation}
and $x,y\in S$. By the first condition in~\eqref{eq:functional metric condition 1}, there is $n_{0}\in\N$ such that $F(\sum_{i=1}^{j_{0}} d_{i}(x_{m},x_{n})e_{i})$ $< \varepsilon/2$ for all $m,n\in\N$ with $m\wedge n\geq n_{0}$. Then~\eqref{eq:functional metric 1} gives $d(x_{m},x_{n})$ $< \varepsilon$ for all such $m,n\in\N$.

(iii) Let $(x_{n})_{n\in\N}$ be Cauchy relative to $d$. By (ii), there is $x\in S$ to which $(x_{n})_{n\in\N}$ converges relative to $d_{i}$ for all $i\in\N$ and it remains to apply (i).

(iv) Let $R$ be a countable set in $S$ that is dense with respect to $d$. Then for any $x\in S$ there is a sequence $(y_{n})_{n\in\N}$ in $R$ such that $\lim_{n\uparrow\infty} d(y_{n},x) = 0$. So, from (i) we obtain that $\lim_{n\uparrow\infty} d_{i}(y_{n},x) = 0$ for each $i\in\N$ and the first claim is proven.

For the last claim let $R_{i}$ be a countable set in $S$ that is dense relative to $d_{i}$ for any $i\in\N$. For $x\in S$ and $\varepsilon > 0$, there are $j_{0}\in\N$ and a sequence $(x_{n})_{n\in\N}$ in $R_{j_{0}}$ such that~\eqref{eq:functional metric 2} holds for all $y\in S$ and $\lim_{n\uparrow\infty} d_{j_{0}}(x_{n},x) = 0$. 

By combining the first condition in~\eqref{eq:functional metric condition 1} with~\eqref{eq:functional metric condition 3}, we obtain $n_{0}\in\N$ such that $F(\sum_{i=1}^{j_{0}}d_{i}(x_{n},x)e_{i}) < \varepsilon/2$ for all $n\in\N$ with $n\geq n_{0}$. So, $d(x_{n},x) < \varepsilon$ for any such $n\in\N$, which shows that $R:=\bigcup_{i\in\N} R_{i}$ is dense relative to $d$.
\end{proof}

\begin{Remark}\label{re:functional metric}
The second condition in~\eqref{eq:functional metric condition 1} holds if $\lim_{j\uparrow\infty}\sup_{a\in\ell_{+}}F(\sum_{i = j}^{\infty}a_{i}e_{i}) = 0$,  and~\eqref{eq:functional metric condition 3} is satisfied if $d_{i} \leq c_{i}d_{i+1}$ for any $i\in\N$ and some sequence $(c_{i})_{i\in\N}$ in $\Re_{+}$.
\end{Remark}

\begin{Example}\label{ex:functional metric}
Let $(f_{n})_{n\in\N}$ denote a sequence of $\Re_{+}$-valued increasing and subadditive functions on $\Re_{+}$ that vanish only at $0$ such that $\sum_{n=1}^{\infty}f_{n}(a_{n}) < \infty$ and $F(a) = \sum_{n=1}^{\infty}f_{n}(a_{n})$ for all $a\in\ell_{+}$. If
\begin{equation*}
f_{n}\leq b_{n}\quad\text{for all $n\in\N$}
\end{equation*}
and some summable sequence $b\in\ell_{+}$, then the second condition in~\eqref{eq:functional metric condition 1} is valid, by Remark~\ref{re:functional metric}. If in addition $f_{n}$ is continuous at $0$ for each $n\in\N$, then~\eqref{eq:functional metric condition 1} holds. For instance,
\begin{equation*}
f_{n}(x) = q^{n-1}\min\{1,x\}
\end{equation*}
for all $n\in\N$ and $x\geq 0$ with $q\in ]0,1[$ is such a feasible choice.
\end{Example}

Now we suppose that $E$ is merely a topological space whose topology is induced by a metric $d$ and $I$ is just a non-empty set for which there is an increasing sequence $(I_{n})_{n\in\N}$ of non-empty sets in $I$ such that $\bigcup_{n\in\N} I_{n} = I$.

Further, let $S$ be a set in $E^{I}$ such that each $x\in E^{I}$ lies in $S$ if there is a sequence $(x_{n})_{n\in\N}$ in $S$ such that $x = x_{n}$ on $I_{n}$ for each $n\in\N$. Then we obtain the following \emph{metrical decomposition principle}.

\begin{Corollary}\label{co:metrical decomposition}
Let $F:\ell_{+}\rightarrow\Re_{+}$ be increasing and subadditive and vanish only at $\mathbf{0}$. Further, let $(d_{i})_{i\in\N}$ be a sequence of pseudometrics on $S$ such that~\eqref{eq:functional metric condition 1} and~\eqref{eq:functional metric condition 3} hold and for any $i\in\N$ and $x,y\in S$ we have
\begin{equation*}
d_{i}(x,y) = 0 \quad\Leftrightarrow\quad x = y\quad\text{on $I_{i}$}.
\end{equation*}
Then the metric $d$ given by~\eqref{eq:functional metric} has the properties (i) and (ii) of Proposition~\ref{pr:functional metric}. Further, $d$ is complete if $d_{i}$ is for any $i\in\N$, and $d$ turns $S$ into a separable space if $d_{i}$ does for each $i\in\N$.
\end{Corollary}

\begin{proof}
Since $\bigcup_{n\in\N} I_{n} = I$, any $x,y\in S$ with $d_{0}(x,y) = 0$ coincide. By Proposition~\ref{pr:functional metric}, it suffices to check that for any sequence $(x_{n})_{n\in\N}$ in $S$ that converges to some $\hat{x}_{i}\in S$ relative to $d_{i}$ for any $i\in\N$ there is $\hat{x}\in S$ such that $\lim_{n\uparrow\infty} d_{i}(x_{n},\hat{x}) = 0$ for all $i\in\N$.

But according to~\eqref{eq:functional metric condition 3} we must have $\hat{x}_{i+1} = \hat{x}_{i}$ on $I_{i}$ for any $i\in\N$. Thus, $\hat{x}\in E^{I}$ given by $\hat{x}(t) := \hat{x}_{i}(t)$ with $i\in\N$ such that $t\in I_{i}$ is well-defined and lies in $S$. In addition, $d_{i}(x_{n},\hat{x})$ $\leq d_{i}(x_{n},\hat{x}_{i})$ for any $i,n\in\N$, which yields the desired result.
\end{proof}

\section{Proofs of the preliminary and main results}\label{se:5}

\subsection{Proofs for stochastic Volterra integrals in Banach spaces}\label{se:5.1}

\begin{proof}[Proof of Proposition~\ref{pr:progressively measurable integral processes}]
It suffices to prove the claim when $U$ is of the form $U = \mathbbm{1}_{C}$ for some $C\in\mathcal{B}(I)\otimes\mathcal{A}$. In the general case, Corollary~\ref{co:approximation of product measurable processes} yields a sequence $(U^{(n)})_{n\in\N}$ of $E$-valued $\mathcal{B}(I)\otimes\mathcal{A}$-measurable maps on $I\times I\times\Omega$, each taking finitely many values, such that $|U_{t,s}^{(k)}| \leq |U_{t,s}|$ and $\lim_{n\uparrow\infty} U_{t,s}^{(n)} = U_{t,s}$ for all $k\in\N$ and $s,t\in I$.

Then it follows from dominated convergence that the sequence $(X^{(n)})_{n\in\N}$ of $E$-valued $\mathcal{A}$-measurable processes given by $X_{t}^{(n)}(\omega) := \int_{[0,t]}U_{t,s}^{(n)}(\omega)\,\mu(\mathrm{d}s)$ for all $n\in\N$ converges pointwise to $X$. So, $X$ is also $\mathcal{A}$-measurable and the assumption on $U$ is justified.

Next, we note that the set $\mathcal{C}$ of all $C\in\mathcal{B}(I)\otimes\mathcal{A}$ for which the process $I\times\Omega\rightarrow\Re_{+}$, $(t,\omega)\mapsto\int_{[0,t]}\mathbbm{1}_{C}(t,s,\omega)\,\mu(\mathrm{d}s)$ is $\mathcal{A}$-measurable is a $d$-system. As Fubini's theorem entails that the continuous process $I\times\Omega\rightarrow\Re_{+}$, $(t,\omega)\mapsto\int_{[0,t]}\mathbbm{1}_{A}(s,\omega)\,\mu(\mathrm{d}s)$ is $\mathbb{F}$-adapted for any $A\in\mathcal{A}$, we have $\mathcal{B}(I)\times\mathcal{A}\subset \mathcal{C}$. Thus, we obtain $\mathcal{C} = \mathcal{B}(I)\otimes\mathcal{A}$ from Dynkin's lemma.
\end{proof}

Before we turn to the proofs of Propositions~\ref{pr:approximation of stochastic Volterra integrals} and~\ref{pr:progressively measurable stochastic Volterra integrals} that involve stochastic integrals with values in $E$, we convince ourselves that $E$ is $2$-smooth if and only if the inequality~\eqref{eq:2-smoothness condition} is satisfied by a norm $\|\cdot\|$ that is equivalent to $|\cdot|$ and some $\hat{c}\geq 2$.

For this purpose, we shall assume that $E$ is just a normed space satisfying $E\neq\{0\}$. Then the \emph{modulus of smoothness} with respect to a norm $\|\cdot\|$ on $E$ is the modulus of continuity $\rho_{\|\cdot\|}:\Re_{+}\rightarrow [0,\infty]$ defined by
\begin{equation*}
\rho_{\|\cdot\|}(t) := \sup\bigg\{ \frac{\|x + ty\| + \|x- ty\|}{2} - 1\,\bigg|\,x,y\in E:\,\|x\| = \|y\| = 1\bigg\}
\end{equation*}
that satisfies $\rho_{\|\cdot\|}(t)\leq t$ for all $t\geq 0$, by the triangle inequality. Based on~\cite{Pis75}, the normed space $E$ is called $q$-smooth for $q\in ]1,2]$ if there are a norm $\|\cdot\|$ that is equivalent to the underlying norm $|\cdot|$ on $E$ such that
\begin{equation}\label{eq:q-smoothness condition 1}
\rho_{\|\cdot\|}(t) \leq ct^{q}\quad\text{for all $t\geq 0$}
\end{equation}
and some $c > 0$. To see that this estimate holds if the inequality~\eqref{eq:q-smoothness condition 2} below is valid for some $\hat{c}\geq 2$, let us recall two basic inequalities. Namely, if $p\geq 1$ and $y\geq 0$, then
\begin{equation}\label{eq:basic inequalities}
2^{1-p}y^{p} - x^{p} \leq (y - x)^{p} \leq y^{p} - x^{p}
\end{equation}
for all $x\in [0,y]$. For $ p = 1$ or $y = 0$ equality holds in~\eqref{eq:basic inequalities}. Otherwise, $\varphi:[0,y]\rightarrow\Re$ given by $\varphi(\tilde{x}) := y - \tilde{x}^{p} - (y - \tilde{x})^{p}$ is continuously differentiable and the only zero $\frac{1}{2}y$ of $\varphi'$ is a global maximum point of $\varphi$, as $\varphi'(x) > 0$ $\Leftrightarrow$ $x < \frac{1}{2}y$. So, $\varphi(0) \leq \varphi(x) \leq \varphi(\frac{1}{2}y)$, which is equivalent to~\eqref{eq:basic inequalities}.
 
\begin{Lemma}\label{le:smoothness of a normed space}
The normed space $E$ is $q$-smooth if there are a norm $\|\cdot\|$ that is equivalent to $|\cdot|$ and $\hat{c}\geq 2$ such that
\begin{equation}\label{eq:q-smoothness condition 2}
\|x + y\|^{q} + \|x-y\|^{q} \leq 2\|x\|^{q} + \hat{c}\|y\|^{q}\quad\text{for all $x,y\in E$,}
\end{equation}
in which case~\eqref{eq:q-smoothness condition 1} is valid for $c = \frac{1}{2}\hat{c}$. The converse holds if $E$ is complete.
\end{Lemma}

\begin{proof}
To check the first implication, let $x,y\in E$ with $\|x\| = \|y\| = 1$ and $t\geq 0$. Then it follows from~\eqref{eq:basic inequalities} that $\|x + ty\| \leq (2 + \hat{c}t^{q} - \|x - ty\|^{q})^{1/q} \leq 2 + \hat{c}t^{q} - \|x - ty\|$, since $\hat{y} := 2 + \hat{c}t^{q}$ satisfies $\hat{y}\leq 2^{1 - q}\hat{y}^{q}$. So~\eqref{eq:q-smoothness condition 1} is valid for $c = \frac{1}{2}\hat{c}$.

For the converse direction let $E$ be complete. Then Theorem~3.1 in~\cite{Pis75} yields a norm $\|\cdot\|$ that is equivalent to $|\cdot|$ and $c \geq 1$ such that $c^{-1}|x| \leq \|x\| \leq |x|$ and $\|x + y\|^{q} + \|x - y\|^{q}$ $\leq 2\|x\|^{q} + 2|y|^{q}$ for any $x,y\in E$. This shows~\eqref{eq:q-smoothness condition 2} for $\hat{c} = 2c^{q}$.
\end{proof}

\begin{proof}[Proof of Proposition~\ref{pr:approximation of stochastic Volterra integrals}]
We may assume that $U = 0$, as otherwise we replace $U^{(n)}$ by $U^{(n)} - U$ and use that $\int_{0}^{\cdot}U_{T,s}^{(n)} - U_{s}\,\mathrm{d}W_{s} = \int_{0}^{\cdot}U_{T,s}^{(n)}\,\mathrm{d}W_{s}$ $ -\, \int_{0}^{\cdot}U_{T,s}\,\mathrm{d}W_{s}$ a.s.~for each $n\in\N$. Then
\begin{equation*}
\P\bigg(\sup_{t\in [0,T]}\bigg|\int_{0}^{t\wedge\tau}U_{T,s}^{(n)}\,\mathrm{d}W_{s}\bigg|\geq\varepsilon\bigg) \leq \frac{w_{2}}{\varepsilon}\E\bigg[\int_{0}^{T\wedge\tau}\big|U_{T,s}^{(n)}\big|_{2}^{2}\,\mathrm{d}s\bigg]^{\frac{1}{2}}
\end{equation*}
for any $n\in\N$ and each $(\mathcal{F}_{t})_{t\in [0,T]}$-stopping time $\tau$, by the inequalities of Markov and Hölder and the moment estimate~\eqref{eq:stochastic integral estimate}. Consequently, we fix $\eta > 0$ and obtain that
\begin{equation*}
\P\bigg(\sup_{t\in [0,T]}\bigg|\int_{0}^{t}U_{T,s}^{(n)}\,\mathrm{d}W_{s}\bigg|\geq\varepsilon\bigg) \leq \P\bigg(\sup_{t\in [0,T]}\bigg|\int_{0}^{t\wedge\tau_{n}}U_{T,s}^{(n)}\,\mathrm{d}W_{s}\bigg|\geq\varepsilon\bigg) + \P(\tau_{n}\leq T)
\end{equation*}
for any $n\in\N$ and the $(\mathcal{F}_{t})_{t\in [0,T]}$-stopping time $\tau_{n} := \inf\{t\in [0,T]\,|\,\int_{0}^{t}\mathbbm{1}_{N_{n}^{c}}|U_{T,s}^{(n)}|_{2}^{2}\,\mathrm{d}s\geq \eta\}$, where $N_{n}:=\{\int_{0}^{T}|U_{T,s}^{(n)}|_{2}^{2}\,\mathrm{d}s = \infty\}$ lies in $\mathcal{F}_{0}$, as $\P(N_{n}) = 0$. Hence, from Theorem~21.4 in~\cite{Bau01} we infer that
\begin{equation*}
\limsup_{n\uparrow\infty}\P\bigg(\sup_{t\in [0,T]}\bigg|\int_{0}^{t}U_{T,s}^{(n)}\,\mathrm{d}W_{s}\bigg|\geq\varepsilon\bigg) \leq \frac{w_{2}}{\varepsilon}\lim_{n\uparrow\infty}\E\bigg[\int_{0}^{T\wedge\tau_{n}}\big|U_{T,s}^{(n)}\big|_{2}^{2}\,\mathrm{d}s\bigg]^{\frac{1}{2}} = 0,
\end{equation*}
since $[0,T]\times\Omega\rightarrow\Re_{+}$, $(t,\omega)\mapsto\int_{0}^{t}\mathbbm{1}_{N_{n}^{c}}(\omega)|U_{T,s}(\omega)|_{2}^{2}\,\mathrm{d}s$ is an $(\mathcal{F}_{t})_{t\in [0,T]}$-adapted, increasing and continuous process starting at $0$, which ensures that $\int_{0}^{T\wedge\tau_{n}}|U_{T,s}^{(n)}|_{2}^{2}\,\mathrm{d}s \leq \eta$ on $N_{n}^{c}$ and $\{\tau_{n} \leq T\}$ $= \{\int_{0}^{T}|U_{T,s}^{(n)}|_{2}^{2}\,\mathrm{d}s \geq \eta\}\cap N_{n}^{c}$ for each $n\in\N$.
\end{proof}

\begin{proof}[Proof of Proposition~\ref{pr:progressively measurable stochastic Volterra integrals}]
It suffices to show that there is a sequence $(X^{(n)})_{n\in\N}$ of $E$-valued $\mathbb{F}$-progressively measurable processes such that $\lim_{n\uparrow\infty}\P(|X_{t}^{(n)} - \int_{0}^{t}U_{t,s}\,\mathrm{d}W_{s}|\geq\varepsilon) = 0$ for a.e.~$t\in I$ and each $\varepsilon > 0$.

Indeed, in this case $\lim_{n\uparrow\infty}\sup_{m\in\N:\,m\geq n}\int_{0}^{T}\P(|X_{t}^{(n)} - X_{t}^{(m)}|\geq\varepsilon)\,\mathrm{d}t = 0$ for any $T\in I$, by dominated convergence. As $E$ is complete and $\mathcal{A}$ is a $\sigma$-field, there exists an $E$-valued $\mathbb{F}$-progressively measurable process $X$ such that $
\lim_{n\uparrow\infty}\int_{0}^{T}\P\big(|X_{t}^{(n)} - X_{t}| \geq \varepsilon\big)\,\mathrm{d}t = 0$ for all $T\in I$ and $\varepsilon > 0$.

By Theorem~20.7 in~\cite{Bau01}, which extends to Borel measurable maps with values in a separable Banach space, a subsequence of $(X_{t}^{(n)})_{n\in\N}$ converges in probability to $X_{t}$ for a.e.~$t\in I$. From this~\eqref{eq:weak modification of a stochastic Volterra integral} follows.

So, let us first prove the stronger statement under the additional assumption that $\lim_{t\downarrow s}\P(\int_{0}^{s}|U_{t,r} - U_{s,r}|_{2}^{2}\,\mathrm{d}r \geq \varepsilon) = 0$ for any $s\in I$ with $s < \sup I$ and $\varepsilon > 0$. If $\sup I\notin I$, then we take a strictly increasing sequence $(T_{n})_{n\in\N}$ in $I$ such that $\sup_{n\in\N} T_{n} = \sup I$. Otherwise, we set $T_{n} := \max I$ for all $n\in\N$.

For each $n\in\N$ let $\mathbb{T}_{n}$ be a partition of $[0,T_{n}]$ of the form $\mathbb{T}_{n} = \{t_{0,n},\dots,t_{N_{n},n}\}$ with $N_{n}\in\N$ and $t_{0,n},\dots,t_{N_{n},n}\in [0,T_{n}]$ satisfying $0 = t_{0,n} < \cdots < t_{N_{n},n} = T_{n}$ and whose mesh $\max_{i=0,\dots,N_{n} - 1} t_{i+1,n} - t_{i,n}$ is denoted by $|\mathbb{T}_{n}|$.

Then $U^{(n)}:I\times I\times\Omega\rightarrow\mathcal{L}_{2}(\ell^{2},E)$ given by $U_{t,s}^{(n)} := \sum_{j=0}^{N_{n} - 1} U_{t_{j+1,n},s}\mathbbm{1}_{ ]t_{j,n},t_{j+1,n}]}(t)$ is $\mathcal{B}(I)\otimes\mathcal{A}$-measurable and $\int_{0}^{t}U_{t,s}^{(n)}\,\mathrm{d}W_{s} = \sum_{j=0}^{N_{n} - 1}\mathbbm{1}_{]t_{j,n},t_{j+1,n}]}(t)\int_{0}^{t\wedge t_{j+1,n}}U_{t_{j+1,n},s}\,\mathrm{d}W_{s}$ a.s. for any $t\in I$. Hence, the process 
\begin{equation}\label{eq:specific stochastic integral}
I\times\Omega\rightarrow E,\quad (t,\omega)\mapsto\int_{0}^{t}U_{t,s}^{(n)}\,\mathrm{d}W_{s}(\omega)
\end{equation}
has an $\mathbb{F}$-progressively measurable modification. Now we may certainly suppose that $\lim_{n\uparrow\infty}|\mathbb{T}_{n}| = 0$. Then condition~\eqref{eq:condition for the approximation of stochastic Volterra integrals} holds and Proposition~\ref{pr:approximation of stochastic Volterra integrals} gives~\eqref{eq:approximation of stochastic Volterra integrals} for any $T\in I$ and $\varepsilon > 0$, which is more than required.

Secondly, we show the stronger statement when $U$ is bounded. So, let $(a_{n})_{n\in\N}$ be a strictly decreasing zero sequence in $]0,\infty[$ and for any $n\in\N$ define a $\mathcal{B}(I)\otimes\mathcal{A}$-measurable map $U^{(n)}:I\times I\times\Omega\rightarrow\mathcal{L}_{2}(\ell^{2},E)$ by $U_{t,s}^{(n)} := \frac{1}{a_{n}}\int_{(t - a_{n})^{+}}^{t}U_{\tilde{s},s}\,\mathrm{d}\tilde{s}$.

Since $\lim_{t\downarrow s}\int_{0}^{s}|U_{t,r}^{(n)} - U_{s,r}^{(n)}|_{2}^{2}\,\mathrm{d}r = 0$ for each $s\in I$ with $s < \sup I$, the process~\eqref{eq:specific stochastic integral} has an $\mathbb{F}$-progressively measurable weak modification, by the preceding part of the proof, and $|U_{t,s}^{(n)} - U_{t,s}|_{2} \leq \frac{1}{a_{n}}\int_{t-a_{n}}^{t}|U_{\tilde{s},s} - U_{t,s}|_{2}\,\mathrm{d}\tilde{s}$ for all $s,t\in I$ with $t \geq a_{n}$. The set $C$ of all $(t,s,\omega)\in I\times I\times\Omega$ with
\begin{equation*}
\limsup_{n\uparrow\infty}\frac{1}{a_{n}}\int_{(t-a_{n})^{+}}^{t}|U_{\tilde{s},s}(\omega) - U_{t,s}(\omega)|_{2}\,\mathrm{d}\tilde{s} > 0
\end{equation*}
lies in $\mathcal{B}(I)\otimes\mathcal{A}$ and $\int_{I}\int_{I}\E[\mathbbm{1}_{C}(t,s,\cdot)]\,\mathrm{d}s\,\mathrm{d}t = 0$, by Fubini's theorem and Lebesgue's differentiation theorem. So, $\lim_{n\uparrow\infty} U_{t,s}^{(n)} = U_{t,s}$ a.s.~for a.e.~$(t,s)\in I\times I$, and dominated convergence yields $\lim_{n\uparrow\infty}\int_{0}^{t}|U_{t,s}^{(n)} - U_{t,s}|_{2}^{2}\,\mathrm{d}s = 0$ a.s.~for a.e.~$t\in I$. By Proposition~\ref{pr:approximation of stochastic Volterra integrals}, $(\int_{0}^{t}U_{t,s}^{(n)}\,\mathrm{d}W_{s})_{n\in\N}$ converges in probability to $\int_{0}^{t}U_{t,s}\,\mathrm{d}W_{s}$ for a.e.~$t\in I$.

Finally, we drop all the preceding hypotheses and note for each $n\in\N$ that the process $I\times\Omega\rightarrow E$, $(t,\omega)\mapsto\int_{0}^{t}U_{t,s}\mathbbm{1}_{\{|U_{t,s}|_{2}\leq n\}}\,\mathrm{d}W_{s}(\omega)$ admits an $\mathbb{F}$-progressively measurable weak modification, by what we have just shown.

Because dominated convergence gives $\lim_{n\uparrow\infty}\int_{0}^{t}|U_{t,s}|_{2}^{2}\mathbbm{1}_{\{|U_{t,s}|_{2} > n\}}\,\mathrm{d}s = 0$ a.s.~for each $t\in I$, Proposition~\ref{pr:approximation of stochastic Volterra integrals} entails that $(\int_{0}^{t}U_{t,s}\mathbbm{1}_{\{|U_{t,s}|_{2}\leq n\}}\,\mathrm{d}W_{s})_{n\in\N}$ converges in probability to $\int_{0}^{t}U_{t,s}\,\mathrm{d}W_{s}$ for each $t\in I$, which completes the proof.
\end{proof}

\subsection{Proofs for admissible coefficients}\label{se:5.2}

\begin{proof}[Proof of Proposition~\ref{pr:admissible maps 1}]
For any $D$-valued $\mathbb{F}$-progressively measurable process $X$ such that $X_{s}\in\mathcal{D}$ for all $s\in I$ the $\mathcal{L}(X_{s})$-integrability of $G_{t,s}(X_{s}(\omega),\cdot)(\omega)$ is ensured for any $s,t\in I$ and $\omega\in\Omega$. So, the measure transformation formula yields that
\begin{equation}\label{eq:integral representation}
F_{t,s}(X_{s})(\omega) = \int_{\Omega}G_{t,s}(X_{s}(\omega),X_{s}(\omega'))(\omega)\,\mathbb{P}(\mathrm{d}\omega').
\end{equation}
This representation entails the $\mathcal{B}(I)\otimes\mathcal{A}$-measurability of the map~\eqref{eq:measurable map}, as we now show by using Dynkin's lemma, Corollary~\ref{co:approximation of product measurable processes} and dominated convergence.

First, the system of all sets $C\in\mathcal{B}(I)\otimes\mathcal{A}\otimes\mathcal{B}(D)\otimes\mathcal{B}(D)$ for which the function $I\times I\times\Omega\rightarrow [0,1]$, $(t,s,\omega)\mapsto\int_{\Omega}\mathbbm{1}_{C}(t,s,\omega,X_{s}(\omega),X_{s}(\omega'))\,\mathbb{P}(\mathrm{d}\omega')$ is $\mathcal{B}(I)\otimes\mathcal{A}$-measurable is a $d$-system. Further,
\begin{equation*}
\int_{\Omega}\mathbbm{1}_{C}(t,s,\cdot,X_{s},X_{s}(\omega'))\,\mathbb{P}(\mathrm{d}\omega') = \mathbbm{1}_{B_{1}}(t)\mathbbm{1}_{A}(s,\cdot)\mathbbm{1}_{B_{2}}(X_{s})\P(X_{s}\in B_{3})
\end{equation*}
for all $s,t\in I$ whenever $C$ is a set in $I\times I\times\Omega\times D\times D$ of the form $C = B_{1}\times A\times B_{2}\times B_{3}$ for some $B_{1}\in\mathcal{B}(I)$, $A\in\mathcal{A}$ and $B_{2},B_{3}\in\mathcal{B}(D)$. Here, the function $I\times\Omega\rightarrow [0,1]$, $(s,\omega)\mapsto\P(X_{s}\in B_{3})$ is $\mathcal{B}(I)\otimes\mathcal{F}_{0}$-measurable, by Fubini's theorem.

Hence, as $X^{-1}(B_{2})\in\mathcal{A}$ and $\mathcal{B}(I)\otimes\mathcal{F}_{0}\subset\mathcal{A}$, Dynkin's lemma and the linearity of the Bochner integral imply that for every map $H:I\times I\times\Omega\times D\times D\rightarrow\tilde{E}$ that is $\mathcal{B}(I)\otimes\mathcal{A}\otimes\mathcal{B}(D)\otimes\mathcal{B}(D)$-measurable and takes finitely many values, the map
\begin{equation*}
I\times I\times\Omega\rightarrow\tilde{E},\quad (t,s,\omega)\mapsto\int_{\Omega}H_{t,s}(X_{s}(\omega),X_{s}(\omega'))(\omega)\,\mathbb{P}(\mathrm{d}\omega')
\end{equation*}
is $\mathcal{B}(I)\otimes\mathcal{A}$-measurable. Next, Corollary~\ref{co:approximation of product measurable processes} yields a sequence $(G^{(n)})_{n\in\N}$ of $\tilde{E}$-valued maps on $I\times I\times\Omega\times D\times D$ that are measurable relative to the same $\sigma$-field as $H$ and take finitely many values such that $|G^{(k)}| \leq |G|$ for all $k\in\N$ and $(G^{(n)})_{n\in\N}$ converges pointwise to $G$. By dominated convergence,
\begin{equation*}
\lim_{n\uparrow\infty}\int_{\Omega}G_{t,s}^{(n)}(X_{s}(\omega),X_{s}(\omega'))(\omega)\,\mathbb{P}(\mathrm{d}\omega') = \int_{\Omega}G_{t,s}(X_{s}(\omega),X_{s}(\omega'))(\omega)\,\mathbb{P}(\mathrm{d}\omega')
\end{equation*}
for any $s,t\in I$ and $\omega\in\Omega$. Consequently, the map~\eqref{eq:measurable map} is product measurable as pointwise limit of a sequence of product measurable maps when $I\times\Omega$ is equipped with $\mathcal{A}$.

Finally, for each $t\in I$ any $\mathbb{F}$-progressively measurable weak modification $\tilde{X}$ of $X$ satisfies $\tilde{X}_{s}\in\mathcal{D}$ and $F_{t,s}(X_{s}) = \int_{\Omega}G_{t,s}(X_{s},\tilde{X}_{s}(\omega'))\,\mathbb{P}(\mathrm{d}\omega') = F_{t,s}(\tilde{X}_{s})$ a.s.~for a.e.~$s\in I$, by the representation~\eqref{eq:integral representation}.
\end{proof}

\begin{proof}[Proof of Proposition~\ref{pr:admissible maps 2}]
For every $D$-valued $\mathbb{F}$-progressively measurable process $X$ that satisfies $\mathcal{L}(X_{s},\alpha_{s})\in\mathcal{P}$ for all $s\in I$ the map
\begin{equation*}
H_{X}:I\times I\times\Omega\rightarrow I\times I\times\Omega\times D\times\mathcal{P},\quad (t,s,\omega)\mapsto \big(t,s,\omega,X_{s}(\omega),\mathcal{L}(X_{s},\alpha_{s})\big)
\end{equation*}
is product measurable once we equip $I\times\Omega$ with the progressive $\sigma$-field $\mathcal{A}$. In fact, for any $B_{1}\in\mathcal{B}(I)$, $A\in\mathcal{A}$, $B_{2}\in\mathcal{B}(D)$ and $B_{3}\in\mathcal{B}(\mathcal{P})$, the preimage $C$ of $B_{1}\times A\times B_{2}\times B_{3}$ under $H_{X}$ is of the form
\begin{equation*}
C = B_{1}\times A\cap X^{-1}(B_{2})\cap\big(\{s\in I\,|\,\mathcal{L}(X_{s},\alpha_{s})\in B_{3}\}\times\Omega\big),
\end{equation*}
which entails that $C\in\mathcal{B}(I)\otimes\mathcal{A}$, as $X^{-1}(B_{2})\in\mathcal{A}$ and the map $I\rightarrow\mathcal{P}$, $s\mapsto\mathcal{L}(X_{s},\alpha_{s})$ is Borel measurable. Thus, let $X_{s}\in\mathcal{D}$ for all $s\in I$, in which case $\mathcal{L}(X_{s},\alpha_{s})\in\mathcal{P}$ holds for any $s\in I$. Then the map~\eqref{eq:measurable map} is $\mathcal{B}(I)\otimes\mathcal{A}$-measurable as composition of $G$ and $H_{X}$.

Eventually, for each $t\in I$ and every $\mathbb{F}$-progressively measurable weak modification $\tilde{X}$ of $X$ we have $\tilde{X}_{s}\in\mathcal{D}$ and $F_{t,s}(X_{s}) = G_{t,s}(X_{s},\mathcal{L}(\tilde{X}_{s},\alpha_{s})) = F_{t,s}(\tilde{X}_{s})$ a.s.~for a.e.~$s\in I$, as required.
\end{proof}

\subsection{Derivation of the first type of solutions}\label{se:5.3}

First, we consider the space in which the solutions to~\eqref{eq:stochastic Volterra equation} will lie in. In this regard, $E$ does not need to be $2$-smooth and $p\in [1,2[$ is possible.

\begin{Lemma}\label{le:completely metrisable space}
The linear space $\mathcal{L}_{loc}^{\infty,p}(I\times\Omega,E)$, endowed with the topology of convergence relative to the seminorm~\eqref{eq:seminorm for processes} for each $T\in I$, is completely pseudometrisable.
\end{Lemma}

\begin{proof}
By Corollary~\ref{co:metrical decomposition}, we may assume that $I$ is compact and take a Cauchy sequence $(X^{(n)})_{n\in\N}$ in $\mathcal{L}_{loc}^{\infty,p}(I\times\Omega,E)$. Then for any $k\in\N$ there is $n_{k}\in\N$ such that for any $m,n\in\N$ with $m\wedge n\geq n_{k}$ there is a Lebesgue null set $N_{k,m,n}\in\mathcal{B}(I)$ such that $\E[|X_{t}^{(n)} - X_{t}^{(m)}|^{p}]^{1/p}$ $\leq \frac{1}{k}$ for all $t\in N_{k,m,n}^{c}$.

Thus, $N:=\bigcup_{\substack{k,m,n\in\N:\,m\wedge n\geq n_{k}}} N_{k,m,n}$ has Lebesgue measure zero and $(X_{t}^{(n)})_{n\in\N}$ is a Cauchy sequence in $\mathcal{L}^{p}(\Omega,E)$ for any $t\in N^{c}$. By the Riesz-Fischer Theorem, there is $\tilde{X}_{t}\in\mathcal{L}^{p}(\Omega,E)$ to which  $(X_{t}^{(n)})_{n\in\N}$ converges in $p$th mean, and $\E[|X_{t}^{(n)} - \tilde{X}_{t}|^{p}]^{1/p} \leq \frac{1}{k}$ for all $k,n\in\N$ with $n\geq n_{k}$.

Hence, $\esssup_{t\in I}\E[|\tilde{X}_{t}|^{p}] < \infty$ and $\lim_{n\uparrow\infty}\esssup_{t\in I}\E[|X_{t}^{(n)} - \tilde{X}_{t}|^{p}] = 0$. As $(X^{(n)})_{n\in\N}$ is also Cauchy sequence in $\mathcal{L}_{loc}^{p}(I\times\Omega,E)$, we have $\lim_{n\uparrow\infty}\int_{I}\E[|X_{t}^{(n)} - X_{t}|^{p}]\,\mathrm{d}t = 0$ for some $X\in\mathcal{L}_{loc}^{p}(I\times \Omega,E)$. Since $X$ must be a weak modification of $\tilde{X}$, we conclude that $X\in\mathcal{L}_{loc}^{\infty,p}(I\times\Omega,E)$ and $\lim_{n\uparrow\infty}\esssup_{t\in I}\E[|X_{t}^{(n)} - X_{t}|^{p}] = 0$.
\end{proof}

By Propositions~\ref{pr:progressively measurable integral processes} and~\ref{pr:progressively measurable stochastic Volterra integrals}, for each $X\in\mathcal{L}_{loc}^{\infty,p}(I\times\Omega,E)$ that satisfies $X_{s}\in\mathcal{D}$ for a.e.~$s\in I$ and for which the event~\eqref{eq:null event} is null for a.e.~$t\in I$, we may choose an $E$-valued $\mathbb{F}$-progressively measurable process $\Psi(X)$ such that
\begin{equation}\label{eq:stochastic Volterra integral operator}
\Psi_{t}(X) = \int_{0}^{t}\B_{t,s}(X_{s})\,\mathrm{d}s + \int_{0}^{t}\Sigma_{t,s}(X_{s})\,\mathrm{d}W_{s}\quad\text{a.s.}\quad\text{for a.e.~$t\in I$.}
\end{equation}

\begin{Lemma}\label{le:stochastic Volterra integral operator 1}
Let~\eqref{co:1} hold and $X\in\mathcal{L}_{loc}^{\infty,p} (I\times\Omega,E)$ satisfy $X_{s}\in\mathcal{D}$ for a.e.~$s\in I$. Then the event~\eqref{eq:null event} is null for a.e.~$t\in I$ and $\Psi(X)\in\mathcal{L}_{loc}^{\infty,p}(I\times\Omega,E)$.
\end{Lemma}

\begin{proof}
Both claims follow from Proposition~\ref{pr:moment estimates for stochastic Volterra processes}. Namely, $k_{0}$ is locally essentially bounded and $l$ lies in $\mathcal{K}^{2}$. Hence, $k_{0}(t) < \infty$ and $l(t,\cdot)\E[|X|^{p}]^{1/p}$ is square-integrable for a.e.~$t\in I$, which entails the first assertion.

Moreover, $\esssup_{t\in [0,T]} \E\big[|\Psi_{t}(X)|^{p}\big]^{1/p} \leq c_{0} + c_{1}\esssup_{t\in [0,T]} \E\big[|X_{t}|^{p}\big]^{1/p}$ for any $T\in I$ whenever $c_{0},c_{1}\geq 0$ satisfy $k_{0}(t) \leq c_{0}$ and $\int_{0}^{t}l(t,s)^{2}\,\mathrm{d}s \leq c_{1}^{2}$ for a.e.~$t\in [0,T]$.
\end{proof}

\begin{proof}[Proof of Proposition~\ref{pr:growth estimates for solutions and Picard iterations 1}]
We notice that the required condition~\eqref{eq:condition for a moment estimate for stochastic Volterra processes} in Proposition~\ref{pr:moment estimates for stochastic Volterra processes} for the bound~\eqref{eq:growth estimate for solutions 1} to be valid holds, since $\lim_{n\uparrow\infty}\int_{0}^{t}\R_{l^{2},n}(t,s)\,\mathrm{d}s = 0$ for a.e.~$t\in I$, by the local essential boundedness of the function series $\mathrm{I}_{l}$ given by~\eqref{eq:function series}.

Now we assume that $\mathcal{L}^{p}(\Omega,E)\subset\mathcal{D}$. Then it follows inductively from Lemma~\ref{le:stochastic Volterra integral operator 1} and Fubini's theorem that $I_{n}$ is Borel and has full measure and $X^{(n)}$ is well-defined and belongs to $\mathcal{L}_{loc}^{\infty,p}(I\times\Omega,E)$ for any $n\in\N$. From Proposition~\ref{pr:moment estimates for stochastic Volterra processes} and Minkowski's inequality we obtain that
\begin{equation}\label{eq:auxiliary moment estimate for Picard iterations}
\E\big[|X_{t}^{(n)} - \xi_{t}|^{p}\big]^{\frac{1}{p}} \leq k_{0,l,\xi}(t) + \bigg(\int_{0}^{t}l(t,s)^{2}\E\big[|X_{s}^{(n-1)} - \xi_{s}|^{p}\big]^{\frac{2}{p}}\,\mathrm{d}s\bigg)^{\frac{1}{2}}
\end{equation}
for all $n\in\N$ and $t\in I_{n}$, where the measurable function $k_{0,l,\xi}:I\rightarrow [0,\infty]$ is given by $k_{0,l,\xi}(t) := k_{0}(t) + (\int_{0}^{t}l(t,s)^{2}\E[|\xi_{s}|^{p}]^{2/p}\,\mathrm{d}s)^{1/2}$. Hence, by recalling the identity~\eqref{eq:special property of the resolvent sequence}, we see that Proposition~\ref{pr:resolvent sequence inequality for processes} and Minkowski's inequality yield the estimate~\eqref{eq:growth estimate for Picard iterations 1}.
\end{proof}

\begin{proof}[Proof of Proposition~\ref{pr:comparison of solutions 1}]
The second claim follows from the first, which entails that $X$ and $\tilde{X}$ are weak modifications of each other as soon as $\xi$ and $\tilde{\xi}$ are. To show the first assertion, we use Proposition~\ref{pr:moment estimates for stochastic Volterra processes}, which yields that
\begin{equation}\label{eq:auxiliary comparison moment estimate}
\E\big[|X_{t} - \tilde{X}_{t}|^{p}\big]^{\frac{1}{p}} \leq \E\big[|\xi_{t} - \tilde{\xi}_{t}|^{p}\big]^{\frac{1}{p}} + \bigg(\int_{0}^{t}\lambda(t,s)^{2}\E\big[|X_{s} - \tilde{X}_{s}|^{p}\big]^{\frac{2}{p}}\,\mathrm{d}s\bigg)^{\frac{1}{2}}
\end{equation}
for a.e.~$t\in I$, because $X - \tilde{X}$ is a Volterra process with coefficients $\xi - \tilde{\xi}$, $\B(X) - \B(\tilde{X})$ and $\Sigma(X) - \Sigma(\tilde{X})$. So, from an application of Corollary~\ref{co:resolvent inequality for processes} in the case $N = 2$, $\beta_{1} = 1$ and $\beta_{2} = 2$ we obtain the claimed estimate~\eqref{eq:comparison estimate for solutions 1}.
\end{proof}

\begin{Lemma}\label{le:regularity of the stochastic Volterra integral operator}
Let~\eqref{co:3} hold and $X\in\mathcal{L}_{loc}^{\infty,p}(I\times\Omega,E)$ satisfy $X_{s}\in\mathcal{D}$ for a.e.~$s\in I$. Then the event~\eqref{eq:null event} is null for all $t\in I$ and there is an $E$-valued $\mathbb{F}$-adapted process $\tilde{X}$ whose paths are $\beta$-Hölder continuous on $[0,t_{i}]$ for all $i\in\N$ and $\beta\in ]0,\beta_{i} - \frac{1}{p}[$ such that
\begin{equation*}
\tilde{X}_{t} = \xi_{t} + \int_{0}^{t}\B_{t,s}(X_{s})\,\mathrm{d}s + \int_{0}^{t}\Sigma_{t,s}(X_{s})\,\mathrm{d}W_{s}\quad\text{a.s.}
\end{equation*}
for any $t\in I$. Further, $\tilde{X}$ is a weak modification of $\xi + \Psi(X)$ and $\E[|\tilde{X} - \xi|^{p}]$ is locally bounded.
\end{Lemma}

\begin{proof}
Proposition~\ref{pr:moment estimates for stochastic Volterra processes} ensures that the event~\eqref{eq:null event} is null for any $t\in I$ and that for each $E$-valued $\mathbb{F}$-adapted process $Z$ satisfying $Z_{t} = \int_{0}^{t}\B_{t,s}(X_{s})\,\mathrm{d}s$ $ +\, \int_{0}^{t}\Sigma_{t,s}(X_{s})\,\mathrm{d}W_{s}$ a.s.~for all $t\in I$, the $p$th moment function $\E[|Z|^{p}]$ is locally bounded.

Thus, from Lemma~\ref{le:stochastic Volterra integral operator 1} and the definition~\eqref{eq:stochastic Volterra integral operator} of $\Psi(X)$ we obtain the second claim. Moreover, Lemma~\ref{le:regularity of stochastic Volterra processes} shows that for each $i\in\N$ there is $c_{i}' > 0$ such that $\E[|Z_{s} - Z_{t}|^{p}]^{1/p}$ $\leq c_{i}'(t-s)^{\beta_{i}}$ for all $s,t\in [0,t_{i}]$ with $s\leq t$.

As any subset of a null event lies in $\mathcal{F}_{0}$, every modification of $Z$ is $\mathbb{F}$-adapted. For this reason, the first assertion follows from the Kolmogorov-Chentsov Theorem. For instance, see~\cite[Proposition~12]{ConKal20} for a quantitative version.
\end{proof}

\begin{proof}[Proof of Theorem~\ref{th:existence of unique solutions 1}]
It suffices to prove the existence and convergence assertions, as the uniqueness claim is a consequence of Proposition~\ref{pr:comparison of solutions 1} and all the assertions under~\eqref{co:3} are implied by Lemma~\ref{le:regularity of the stochastic Volterra integral operator}, Remark~\ref{re:notion of a solution} and Proposition~\ref{pr:growth estimates for solutions and Picard iterations 1}.

As mentioned in the proof of Proposition~\ref{pr:growth estimates for solutions and Picard iterations 1}, it follows inductively from Lemma~\ref{le:stochastic Volterra integral operator 1} that $X^{(n)}$ is well-defined and belongs to $\mathcal{L}_{loc}^{\infty,p}(I\times\Omega,E)$. Furthermore, the Borel set $I_{\infty}$, which equals $\bigcap_{n\in\N} I_{n}$, has full measure and we have
\begin{equation}\label{eq:auxiliary Picard moment estimate}
\E\big[|X_{t}^{(n)} - X_{t}^{(n+1)}|^{p}\big]^{\frac{1}{p}} \leq \bigg(\int_{0}^{t}\lambda(t,s)^{2}\E\big[|X_{s}^{(n-1)} - X_{s}^{(n)}|^{p}\big]^{\frac{2}{p}}\,\mathrm{d}s\bigg)^{\frac{1}{2}}
\end{equation}
for any $n\in\N$ and $t\in I_{\infty}$, as Proposition~\ref{pr:moment estimates for stochastic Volterra processes} entails. Next, may apply Proposition~\ref{pr:resolvent sequence inequality for processes} in the case $N = 1$ and $\beta_{1} = 2$, due to Remark~\ref{re:specific resolvent sequence inequality}. Then we obtain that
\begin{equation}\label{eq:preliminary error moment estimate 1}
\E\big[|X_{t}^{(n)} - X_{t}^{(m)}|^{p}\big]^{\frac{1}{p}} \leq \sum_{i=n}^{m-1}\bigg(\int_{0}^{t}\R_{\lambda^{2},i}(t,s)\Delta(s)^{2}\,\mathrm{d}s\bigg)^{\frac{1}{2}}
\end{equation}
for all $m,n\in\N$ with $m > n$ and $t\in I_{\infty}$, by the triangle inequality in the $L^{p}$-norm. Thus, $(X^{(n)})_{n\in\N}$ is a Cauchy sequence in $\mathcal{L}_{loc}^{\infty,p}(I\times\Omega,E)$. Namely, as~\eqref{eq:convergence of the error coefficients 1} follows directly from~\eqref{eq:error series estimate}, we have
\begin{equation*}
\lim_{n\uparrow\infty}\sup_{m\in\N:\,m\geq n}\esssup_{t\in [0,T]}\E\big[|X_{t}^{(n)} - X_{t}^{(m)}|^{p}\big] = 0\quad\text{for any $T\in I$.}
\end{equation*}
By Lemma~\ref{le:completely metrisable space}, there is $X^{\xi}\in\mathcal{L}_{loc}^{\infty,p}(I\times\Omega,E)$ that is unique, up to a weak modification, such that~\eqref{eq:convergence in pth moment, locally essentially uniformly in time} holds when $X$ is replaced by $X^{\xi}$. Hence, the error estimate~\eqref{eq:error moment estimate 1} is a direct implication of the inequality~\eqref{eq:preliminary error moment estimate 1}.

Finally, the sequential continuity of $\Psi$ on $\mathcal{L}_{loc}^{\infty,p}(I\times\Omega,E)$, which follows from another application of Proposition~\ref{pr:moment estimates for stochastic Volterra processes}, entails that $\lim_{n\uparrow\infty} \esssup_{t\in [0,T]}\E[|X_{t}^{(n+1)} - \xi_{t} - \Psi_{t}(X^{\xi})|^{p}]$ $= 0$ for all $T\in I$. Thus, as weak modification of $\xi + \Psi(X^{\xi})$ the process $X^{\xi}$ is a solution to~\eqref{eq:stochastic Volterra equation}, by Remark~\ref{re:notion of a solution}.
\end{proof}

\subsection{Derivation of the second type of solutions}\label{se:5.4}

By using the progressive $\sigma$-field $\mathcal{A}$ on $I\times\Omega$, let us first of all verify that the linear space $\mathcal{L}_{loc}^{p}(I\times\Omega,E)$, equipped with the seminorm~\eqref{eq:seminorm for processes 2} for every $T\in I$, is in fact completely pseudometrisable, even if $E$ fails to be $2$-smooth or $p\in [1,2[$.

In the case that $I$ is compact we recover the linear space of all $E$-valued $\mathcal{A}$-measurable maps on $I\times\Omega$ that are $p$-fold integrable with respect to the finite measure
\begin{equation*}
\mu:\mathcal{A}\rightarrow\Re_{+}, \quad A\mapsto\int_{0}^{T}\E\big[\mathbbm{1}_{A}(t,\cdot)\big]\,\mathrm{d}t,
\end{equation*}
endowed wit the seminorm $\mathcal{L}_{loc}^{p}(I\times\Omega,E)\rightarrow\Re_{+}$, $X\mapsto (\int_{I\times\Omega} |X|^{p}\,\mathrm{d}\mu)^{1/p}$. So, if $I$ is compact, then $\mathcal{L}_{loc}^{p}(I\times\Omega,E)$ is complete, by the Riesz-Fischer Theorem. Consequently, the general case follows from Corollary~\ref{co:metrical decomposition}.

Next, we notice that for any $X\in\mathcal{L}_{loc}^{p}(I\times\Omega,E)$ satisfying $X_{s}\in\mathcal{D}$ for a.e.~$s\in I$ and for which the event~\eqref{eq:null event} is null for a.e.~$t\in I$, we may still define the $E$-valued $\mathbb{F}$-progressively measurable process $\Psi(X)$ by the requirement~\eqref{eq:stochastic Volterra integral operator}.

\begin{Lemma}\label{le:stochastic Volterra integral operator 2}
Let~\eqref{co:4} be valid and $X\in\mathcal{L}_{loc}^{p}(I\times\Omega,E)$ satisfy $X_{s}\in\mathcal{D}$ for a.e.~$s\in I$. Then the event~\eqref{eq:null event} is null for a.e.~$t\in I$ and $\Psi(X)\in\mathcal{L}_{loc}^{p}(I\times\Omega,E)$.
\end{Lemma}

\begin{proof}
According to Corollary~\ref{co:integral resolvent sequence inequality for processes}, for the measurable function $u:I\rightarrow [0,\infty]$ defined by $u(t) := k_{0}(t) + (\int_{0}^{t}l(t,s)^{2}\E[|X_{s}|^{p}]^{2/p}\,\mathrm{d}s)^{1/2}$ we have
\begin{equation*}
\bigg(\int_{0}^{t}u(s)^{p}\,\mathrm{d}s\bigg)^{\frac{1}{p}} \leq \bigg(\int_{0}^{t}k_{0}(s)^{p}\,\mathrm{d}s\bigg)^{\frac{1}{p}} + \bigg(\int_{0}^{t}l_{1,p}(t,s)\E\big[|X_{s}|^{p}\big]\,\mathrm{d}s\bigg)^{\frac{1}{p}} < \infty
\end{equation*}
for each $t\in I$. In particular, $u(t) < \infty$ for a.e.~$t\in I$. Hence, both claims are implied by Proposition~\ref{pr:moment estimates for stochastic Volterra processes}.
\end{proof}

\begin{proof}[Proof of Proposition~\ref{pr:growth estimates for solutions and Picard iterations 2}]
The moment estimate~\eqref{eq:growth estimate for solutions 2} is implied by Proposition~\ref{pr:moment estimates for stochastic Volterra processes}, Corollary~\ref{co:integral resolvent inequality for processes} and Minkowski's inequality. Indeed, the condition in the corollary that $\lim_{n\uparrow\infty}\int_{0}^{t}l_{n,p}(t,s)\E[|X_{s}|^{p}]\,\mathrm{d}s = 0$ for any $t\in I$ holds, as $c_{l,p}$ is finite.

Next, we suppose that $\mathcal{L}^{p}(\Omega,E)\subset\mathcal{D}$. By Lemma~\ref{le:stochastic Volterra integral operator 2}, Fubini's theorem and induction, $I_{n}$ is Borel and has full measure and $X^{(n)}$ is well-defined and lies in $\mathcal{L}_{loc}^{p}(I\times\Omega,E)$ for each $n\in\N$. Hence, Proposition~\ref{pr:moment estimates for stochastic Volterra processes} yields the inequality~\eqref{eq:auxiliary moment estimate for Picard iterations}. For this reason, the moment bound~\eqref{eq:growth estimate for Picard iterations 2} follows from an application of Corollary~\ref{co:integral resolvent sequence inequality for processes} and Minkowski's inequality.
\end{proof}

\begin{proof}[Proof of Proposition~\ref{pr:comparison of solutions 2}]
It suffices to show the claimed estimate, which entails that if $\xi$ and $\tilde{\xi}$ are weak modifications, then so are $X$ and $\tilde{X}$. We recall that Proposition~\ref{pr:moment estimates for stochastic Volterra processes} implies the inequality~\eqref{eq:auxiliary comparison moment estimate}, because $X - \tilde{X}$ is a Volterra process with coefficients $\xi - \tilde{\xi}$, $\B(X) - \B(\tilde{X})$ and $\Sigma(X) - \Sigma(\tilde{X})$.

For this reason, Corollary~\ref{co:integral resolvent inequality for processes} gives the desired estimate~\eqref{eq:comparison estimate for solutions 2}. In fact, the required condition that $\lim_{n\uparrow\infty}\int_{0}^{t}\lambda_{n,p}(t,s)\E[|X_{s} - \tilde{X}_{s}|^{p}]\,\mathrm{d}s = 0$ for any $t\in I$ is satisfied, since $c_{\lambda,p}$ is finite.
\end{proof}

\begin{proof}[Proof of Theorem~\ref{th:existence of unique solutions 2}]
The uniqueness claim is shown in Proposition~\ref{pr:comparison of solutions 2}. So, it remains to establish the assertions on existence and convergence. By Lemma~\ref{le:stochastic Volterra integral operator 2} and induction, $X^{(n)}$ is well-defined and lies in $\mathcal{L}_{loc}^{p}(I\times\Omega,E)$ for any $n\in\N$. Since Proposition~\ref{pr:moment estimates for stochastic Volterra processes} entails the estimate~\eqref{eq:auxiliary Picard moment estimate}, we have
\begin{equation}\label{eq:preliminary error moment estimate 2}
\bigg(\int_{0}^{t}\E\big[|X_{s}^{(n)} - X_{s}^{(m)}|^{p}\big]\,\mathrm{d}s\bigg)^{\frac{1}{p}} \leq \sum_{i=n}^{m-1}\bigg(\int_{0}^{t}\lambda_{i,p}(t,s)\Delta(s)^{p}\,\mathrm{d}s\bigg)^{\frac{1}{p}}
\end{equation}
for all $m,n\in\N$ with $m > n$ and $t\in I$, due to Corollary~\ref{co:integral resolvent sequence inequality for processes} and Minkowski's inequality. Hence, $(X^{(n)})_{n\in\N}$ is a Cauchy sequence in $\mathcal{L}_{loc}^{p}(I\times\Omega,E)$. In fact, as the function $c_{\lambda,p}$ is finite, the limit~\eqref{eq:convergence of the error coefficients 2} holds. So,
\begin{equation*}
\lim_{n\uparrow\infty}\sup_{m\in\N:\,m\geq n}\int_{0}^{t}\E\big[|X_{s}^{(n)} - X_{s}^{(m)}|^{p}\big]\,\mathrm{d}s = 0\quad\text{for any $t\in I$}
\end{equation*}
and this implies that there is $X^{\xi}\in\mathcal{L}_{loc}^{p}(I\times\Omega,E)$ that is unique, up to a weak modification, such that $\lim_{n\uparrow\infty}\int_{0}^{t}\E[|X_{s}^{(n)} - X_{s}^{\xi}|^{p}]\,\mathrm{d}s = 0$ for all $t\in I$.  Now the error estimate~\eqref{eq:error moment estimate 2} follows by taking the limit $m\uparrow\infty$ in~\eqref{eq:preliminary error moment estimate 2}.

As Proposition~\ref{pr:moment estimates for stochastic Volterra processes} and Corollary~\ref{co:integral resolvent sequence inequality for processes} entail that $\Psi$ is sequentially continuous on $\mathcal{L}_{loc}^{p}(I\times\Omega,E)$, we have  $\lim_{n\uparrow\infty}\int_{0}^{t}\E[|X_{s}^{(n+1)} - \xi_{s} - \Psi_{s}(X^{\xi})|^{p}]\,\mathrm{d}s = 0$ for any $t\in I$. So, $X^{\xi}$ is a weak modification of $\xi + \Psi(X^{\xi})$, which shows us that it solves~\eqref{eq:stochastic Volterra equation}, by Remark~\ref{re:notion of a solution}.
\end{proof}


\begin{thebibliography}{10}

\bibitem{Abi21}
E.~Abi~Jaber.
\newblock Weak existence and uniqueness for affine stochastic {V}olterra
  equations with {$L^1$}-kernels.
\newblock {\em Bernoulli}, 27(3):1583--1615, 2021.

\bibitem{AbiCucLar21}
E.~Abi~Jaber, C.~Cuchiero, M.~Larsson, and S.~Pulido.
\newblock A weak solution theory for stochastic {V}olterra equations of
  convolution type.
\newblock {\em Ann. Appl. Probab.}, 31(6):2924--2952, 2021.

\bibitem{AbiLarPul19}
E.~Abi~Jaber, M.~Larsson, and S.~Pulido.
\newblock Affine {V}olterra processes.
\newblock {\em Ann. Appl. Probab.}, 29(5):3155--3200, 2019.

\bibitem{AbiNeu25}
E.~Abi~Jaber and E.~Neuman.
\newblock Optimal liquidation with signals: the general propagator case.
\newblock {\em Math. Finance}, 35(4):841--866, 2025.

\bibitem{AccBacCar19}
B.~Acciaio, J.~Backhoff-Veraguas, and R.~Carmona.
\newblock Extended mean field control problems: stochastic maximum principle
  and transport perspective.
\newblock {\em SIAM J. Control Optim.}, 57(6):3666--3693, 2019.

\bibitem{AliBor99}
C.~D. Aliprantis and K.~C. Border.
\newblock {\em Infinite-dimensional analysis}.
\newblock Springer-Verlag, Berlin, second edition, 1999.
\newblock A hitchhiker's guide.

\bibitem{AloNua97}
E.~Al\`os and D.~Nualart.
\newblock Anticipating stochastic {V}olterra equations.
\newblock {\em Stochastic Process. Appl.}, 72(1):73--95, 1997.

\bibitem{Bau01}
H.~Bauer.
\newblock {\em Measure and integration theory}, volume~26 of {\em De Gruyter
  Studies in Mathematics}.
\newblock Walter de Gruyter \& Co., Berlin, 2001.
\newblock Translated from the German by Robert B.\ Burckel.

\bibitem{BerMiz80}
M.~A. Berger and V.~J. Mizel.
\newblock Volterra equations with {I}t\^o{} integrals. {I}.
\newblock {\em J. Integral Equations}, 2(3):187--245, 1980.

\bibitem{Bermiz80-2}
M.~A. Berger and V.~J. Mizel.
\newblock Volterra equations with {I}t\^o{} integrals. {II}.
\newblock {\em J. Integral Equations}, 2(4):319--337, 1980.

\bibitem{BonPulSco24}
A.~Bondi, S.~Pulido, and S.~Scotti.
\newblock The rough {H}awkes {H}eston stochastic volatility model.
\newblock {\em Math. Finance}, 34(4):1197--1241, 2024.

\bibitem{CocLeePot95}
W.~G. Cochran, J.-S. Lee, and J.~Potthoff.
\newblock Stochastic {V}olterra equations with singular kernels.
\newblock {\em Stochastic Process. Appl.}, 56(2):337--349, 1995.

\bibitem{ConKal20}
R.~Cont and A.~Kalinin.
\newblock On the support of solutions to stochastic differential equations with
  path-dependent coefficients.
\newblock {\em Stochastic Process. Appl.}, 130(5):2639--2674, 2020.

\bibitem{DaiXia20}
X.~Dai and A.~Xiao.
\newblock L\'evy-driven stochastic {V}olterra integral equations with doubly
  singular kernels: existence, uniqueness, and a fast {EM} method.
\newblock {\em Adv. Comput. Math.}, 46(2):Paper No. 29, 23, 2020.

\bibitem{NunGio24}
G.~di~Nunno and M.~Giordano.
\newblock Stochastic {V}olterra equations with time-changed {L}\'evy noise and
  maximum principles.
\newblock {\em Ann. Oper. Res.}, 336(1-2):1265--1287, 2024.

\bibitem{DueGeeVerWel10}
L.~D\"umbgen, S.~A. van~de Geer, M.~C. Veraar, and J.~A. Wellner.
\newblock Nemirovski's inequalities revisited.
\newblock {\em Amer. Math. Monthly}, 117(2):138--160, 2010.

\bibitem{EleRos18}
O.~El~Euch and M.~Rosenbaum.
\newblock Perfect hedging in rough {H}eston models.
\newblock {\em Ann. Appl. Probab.}, 28(6):3813--3856, 2018.

\bibitem{EleRos19}
O.~El~Euch and M.~Rosenbaum.
\newblock The characteristic function of rough {H}eston models.
\newblock {\em Math. Finance}, 29(1):3--38, 2019.

\bibitem{GatJaiRos18}
J.~Gatheral, T.~Jaisson, and M.~Rosenbaum.
\newblock Volatility is rough.
\newblock {\em Quant. Finance}, 18(6):933--949, 2018.

\bibitem{Gri80}
G.~Gripenberg.
\newblock On the resolvents of nonconvolution {V}olterra kernels.
\newblock {\em Funkcial. Ekvac.}, 23(1):83--95, 1980.

\bibitem{JieLuoZha24}
L.~Jie, L.~Luo, and H.~Zhang.
\newblock One-dimensional {M}c{K}ean-{V}lasov stochastic {V}olterra equations
  with {H}\"older diffusion coefficients.
\newblock {\em Statist. Probab. Lett.}, 205:Paper No. 109970, 11, 2024.

\bibitem{Kal21}
A.~Kalinin.
\newblock Support characterization for regular path-dependent stochastic
  {V}olterra integral equations.
\newblock {\em Electron. J. Probab.}, 26:Paper No. 29, 29, 2021.

\bibitem{Kal24}
A.~Kalinin.
\newblock Resolvent and {G}ronwall inequalities and fixed points of evolution
  operators.
\newblock {\em arXiv preprint arXiv:2412.20764}, 2024.

\bibitem{KalMeyPro24}
A.~Kalinin, T.~Meyer-Brandis, and F.~Proske.
\newblock Stability, uniqueness and existence of solutions to
  {M}c{K}ean-{V}lasov {SDE}s: a multidimensional {Y}amada-{W}atanabe approach.
\newblock {\em Stoch. Dyn.}, 24(5):Paper No. 2450039, 49, 2024.

\bibitem{Ond04}
M.~Ondrej\'at.
\newblock Uniqueness for stochastic evolution equations in {B}anach spaces.
\newblock {\em Dissertationes Math. (Rozprawy Mat.)}, 426:63, 2004.

\bibitem{ParPro90}
E.~Pardoux and P.~Protter.
\newblock Stochastic {V}olterra equations with anticipating coefficients.
\newblock {\em Ann. Probab.}, 18(4):1635--1655, 1990.

\bibitem{Pin94}
I.~Pinelis.
\newblock Optimum bounds for the distributions of martingales in {B}anach
  spaces.
\newblock {\em Ann. Probab.}, 22(4):1679--1706, 1994.

\bibitem{Pis75}
G.~Pisier.
\newblock Martingales with values in uniformly convex spaces.
\newblock {\em Israel J. Math.}, 20(3-4):326--350, 1975.

\bibitem{ProSch23}
D.~J. Pr\"omel and D.~Scheffels.
\newblock On the existence of weak solutions to stochastic {V}olterra
  equations.
\newblock {\em Electron. Commun. Probab.}, 28:Paper No. 52, 12, 2023.

\bibitem{ProSch23-2}
D.~J. Pr\"omel and D.~Scheffels.
\newblock Stochastic {V}olterra equations with {H}\"older diffusion
  coefficients.
\newblock {\em Stochastic Process. Appl.}, 161:291--315, 2023.

\bibitem{ProSch25}
D.~J. Pr{\"o}mel and D.~Scheffels.
\newblock Mean-field stochastic {V}olterra equations.
\newblock {\em arXiv preprint arXiv:2307.13775v4}, 2025.

\bibitem{Pro85}
P.~Protter.
\newblock Volterra equations driven by semimartingales.
\newblock {\em Ann. Probab.}, 13(2):519--530, 1985.

\bibitem{NeeVerWei07}
J.~M. A.~M. van Neerven, M.~C. Veraar, and L.~Weis.
\newblock Stochastic integration in {UMD} {B}anach spaces.
\newblock {\em Ann. Probab.}, 35(4):1438--1478, 2007.

\bibitem{Vil03}
C.~Villani.
\newblock {\em Topics in optimal transportation}, volume~58 of {\em Graduate
  Studies in Mathematics}.
\newblock American Mathematical Society, Providence, RI, 2003.

\bibitem{Wan22}
T.~Wang.
\newblock Backward stochastic {V}olterra integro-differential equations and
  applications in optimal control problems.
\newblock {\em SIAM J. Control Optim.}, 60(4):2393--2419, 2022.

\bibitem{Wan08}
Z.~Wang.
\newblock Existence and uniqueness of solutions to stochastic {V}olterra
  equations with singular kernels and non-{L}ipschitz coefficients.
\newblock {\em Statist. Probab. Lett.}, 78(9):1062--1071, 2008.

\bibitem{Zha10}
X.~Zhang.
\newblock Stochastic {V}olterra equations in {B}anach spaces and stochastic
  partial differential equation.
\newblock {\em J. Funct. Anal.}, 258(4):1361--1425, 2010.

\end{thebibliography}
\end{document}